\documentclass{article}
\usepackage[T1]{fontenc}
\usepackage[utf8x]{inputenc}
\usepackage{amsmath, amssymb, amsthm, mathrsfs,cleveref}
\usepackage{bbold}
\usepackage{tikz-cd}
\usepackage[margin=3cm]{geometry}
\usepackage[shortlabels]{enumitem}
\usepackage{pifont}
\usetikzlibrary{decorations.markings,arrows}
\usepackage{fixme}
\newtheorem{thm}{Theorem}[section] \newtheorem*{theorem*}{Theorem}
\newtheorem{lemma}[thm]{Lemma} \newtheorem*{lemma*}{Lemma}
\newtheorem{corollary}[thm]{Corollary} \newtheorem*{corollary*}{Corollary}
\newtheorem{prop}[thm]{Proposition} \newtheorem*{proposition*}{Proposition}

\newtheorem{definition}[thm]{Definition}
\newtheorem*{definition*}{Definition}
\newtheorem{rmk}[thm]{Remark}

\newcommand{\Z}{\mathbb Z}
\newcommand{\N}{\mathbb N}
\newcommand{\C}{\mathbb C}

\newcommand{\R}{\mathbb R}

\newcommand{\E}{\mathbb E}

\newcommand\cB{{\mathcal B}}

\newcommand\fund{{\mathcal F}}

\newcommand\cH{{\mathcal H}}

\newcommand\cL{{\mathcal L}}

\newcommand\cR{{\mathcal R}}

\newcommand\cW{{\mathcal W}}
\newcommand\cP{{\mathcal P}}

\newcommand\bC{{\mathbb C}}

\newcommand\bS{{\mathbb S}}

\newcommand\bZ{{\mathbb Z}}

\newcommand\eps{\varepsilon}
\newcommand\vf{\varphi}

\DeclareMathOperator{\hol}{hol}
\DeclareMathOperator{\Deck}{Deck}

\newcommand\nb{{\mathbf n}}
\newcommand\mb{{\mathbf m}}
\def\1{\mathbb 1}

\newcommand*\diff{\mathop{}\!\mathrm{d}}

\def \Leb {\boldsymbol{m}}
\def \Xcomp {X}
\def \Xcover {X_{\Gamma}}
\def \psicomp {\psi}
\def \psicover {\psi_\Gamma}
\def \vcover {v_\Gamma}
\def \wcover {w_\Gamma}
\def \Gcover {G_\Gamma}
\def \Lcomp{\cL}
\def \Lcover {\cL_\Gamma}
\def \coverproj {p}
\def \coverhomo {\zeta}

\def \deck {\Delta}
\def \Ess {\operatorname{Ess}}

\def \KK {K}
\newcommand{\kpol}{\k}

\title{On asymptotic expansions of ergodic integrals for $\Z^d$-extensions of translation flows.
\thanks{This research was the result of the 2023 Research-in-Teams project 0223
``Limit Theorems for Parabolic Dynamical Systems''
at the Erwin Schr\"odinger Institute, Vienna}
}
\author{Henk Bruin\thanks{Faculty of Mathematics, University of Vienna,
		Oskar Morgensternplatz 1, 1090 Vienna, Austria; {\it henk.bruin@univie.ac.at} },
Charles Fougeron\thanks{LAGA - Universit\'e Sorbonne Paris Nord, 99 Av. Jean Baptiste Cl\'ement, 93430 Villetaneuse, France;
{\it charles.fougeron@math.cnrs.fr} },
Davide Ravotti\thanks{Universit\'e de Lille, CNRS, UMR 8524 - Laboratoire Paul Painlev\'e, F-59000 Lille, France; {\it davide.ravotti@univ-lille.fr} },
Dalia Terhesiu\thanks{Institute of Mathematics, University of Leiden,
 	Einsteinweg 55 2333 CC Leiden, The Netherlands,
 		{\it daliaterhesiu@gmail.com}}
}
\date{\today}

\begin{document}
\maketitle

\begin{abstract}
We obtain expansions of ergodic integrals for $\Z^d$-covers of compact self-similar translation flows, and as a consequence we obtain a form of weak rational ergodicity with optimal rates.
As examples, we consider the so-called self-similar $(s,1)$-staircase flows ($\Z$-extensions of self-similar translations flows of genus-$2$ surfaces),
and particular cases of the Ehrenfest wind-tree model.
\iffalse
\\[4mm]
Nous \'etablissons des d\'eveloppements asymptotiques d'int\'egrales ergodiques pour des $\Z^d$-revê\^etements de flots directionnels compacts auto-similaires:
De cela d\'ecoule un r\'esultat d'ergodicit\'e rationnelle faible avec des taux optimaux.
À titre d'exemple, nous considérons des flots sur des escaliers auto-similaires de type $(s,1)$ (qui sont des $\Z$-extensions de flots directionnels auto-similaires sur des surfaces de genre $2$) ainsi qu'une variante du mod\`ele de vent dans les arbres des Ehrenfest.
\fi
\\[3mm] {\it 2000 Mathematics Subject Classification.} Primary 37E35, Secondary 28D05, 37H05, 14H30, 60F15
\\[1mm] {\it Keywords and phrases.} translation flow, $\Z$-extension, wind-tree model, weak rational ergodicity
%
%
% 05C81 - Random walks on graphs
% 37H05 - General theory of random and stochastic dynamical systems
% 37F30 - Quasiconformal methods and Teichmüller theory; Fuchsian and Kleinian groups as dynamical systems
% 11Y40 - Algebraic number theory computations
% 14H30 - Coverings of curves, fundamental group
% 52C99 - Discrete geometry
% 28D05 - Measure-preserving transformations
% 60F15 - Strong limit theorems
\end{abstract}

\listoffixmes

\section{Introduction}\label{sec:intro}

Given a measure preserving flow $(\phi_t)_{t \in \R}$ on a measure space $(\Xcomp, \mu)$, one is interested in describing the almost everywhere behaviour of its orbits. If the flow is ergodic and if $\mu(X) < \infty$, Birkhoff's Ergodic Theorem states that, for any integrable observable $G \colon \Xcomp \to \R$, the time averages $\frac{1}{T} \int_0^T G \circ \phi_t \diff t$ converge almost everywhere to $\frac{1}{\mu(X)} \int_X G \diff \mu$. On the other hand, if $\mu(X) = \infty$, for any ergodic, conservative flow $(\phi_t)_{t \in \R}$, its time averages
for integrable functions converge to $0$ almost everywhere.
The situation does not improve even if we replace $\frac{1}{T}$ with any other normalizing family of functions $a(T)$, see Aaronson \cite[Theorem 2.4.2]{Aaronson}: for any non-negative integrable function $G$, either $\liminf_{T \to \infty} \frac{1}{a(T)} \int_0^T G \circ \phi_t \diff t = 0$ or  $\limsup_{T \to \infty}  \frac{1}{a(T)} \int_0^T G \circ \phi_t \diff t = \infty$ almost everywhere.

However, one can still hope to describe the almost everywhere behaviour of the ergodic integrals in some weaker sense. In particular, for an integrable function $G \colon X \to \R$, we seek an expression of the form
\begin{equation}\label{eq:general_int}
\int_0^T G \circ \phi_t(x) \diff t  = a(T) \left( \int_X G \diff \mu \right) \cdot \Phi_T(x) (1 + o(1)),
\end{equation}
where $a(T)$ describes the \lq\lq correct (almost everywhere) size\rq\rq\ of the ergodic averages (which, at least for us, is $o(T)$) and $\Phi_T(x)$ is an \lq\lq oscillating\rq\rq\ term which, although not convergent almost everywhere, converges in some weaker sense (and, crucially, depends only on the point $x$, not on the function $G$).

In this paper, we consider a translation flow $(\phi_t)_{t \in \R}$ on a space $\Xcover$
which is a $\Z^d$-cover of a compact translation surface $\Xcomp$ with  projection $p: \Xcover \to \Xcomp$.
The Lebesgue measure $\Leb$ is infinite on $\Xcover$ and invariant w.r.t.\ both the flow $\phi_t$ and the deck-transformations associated to the cover.
Our main result is that, under certain assumptions described below, an expression as \eqref{eq:general_int} holds for all continuous functions $G \colon \Xcover \to \R$ with compact support, with $a(T) \sim T (\log T)^{-d/2}$ and where $\sqrt{\log(\phi_T \circ p)}$ converges in distribution to a Gaussian random variable.

Results of this type have been proved by many authors in several settings, including \cite{LS} for $\Z^d$ covers of horocycle flows, and \cite{ADDS} for $\Z$-covers of a translation torus.
%%%%%%%%%%
Furthermore, in \cite{ADDS}, the authors used this result to prove temporal limit theorems for circle rotations $R_\theta : \bS^1 \to \bS^1$
for observables with $\int_{\bS^1} G \, \diff\Leb = 0$ and specific (namely, quadratically irrational) rotation angles $\theta$.
This amounts to determining the asymptotics of $\sum_{i=0}^{n-1} G \circ R_\theta^i(x)$
for a fixed $x$ and increasing time intervals $[0,n]$.
The crucial idea in their proofs is renormalization, allowing one to speed up a translation flow
$\phi_t$ on a $\Z$-cover $\Xcover$ of a two-dimensional twice punctured torus $\Xcomp$ in the contracting direction of a linear pseudo-Anosov lift\footnote{By this we mean the lift of a pseudo-Anosov automorphism; as $\Xcover$ is not compact, and may have singularities with an infinite or ill-defined cone angle, $\Xcover$ may not carry proper pseudo-Anosov automorphisms.} $\psicover$ of $\Xcover$
according to
\begin{equation}\label{eq:renorm}
\psicover \circ \phi_t = \phi_{\lambda t} \circ \psicover \quad \text{for every } t \in \R,
\end{equation}
where $\lambda \in (0,1)$ is the contraction factor of $\psicover$.
Therefore the asymptotics of ergodic integrals $\int_0^T G \circ \phi_t \, \diff t$
for compactly supported observables $G:\Xcover \to \R$ can be estimated using the asymptotics of $\int_{\fund} G \circ \psi^k \diff\Leb$,
where $\fund$ is a fundamental domain and $T \approx \lambda^{-k}$.

The central result in \cite{ADDS} in our notation is
\begin{equation}\label{eq:adds}
\int_0^T G \circ \phi_t(x) \, \diff t =
\left(\int_{\Xcover} G \, \diff\Leb  \right)
\ \frac{ (1+o(1)) T}{ \sigma \sqrt{2K} }
\exp\left( -\frac{1+o(1)}{2\sigma^2} \frac{\left( \xi \circ \psicover^{\KK}(x) \right)^2}{K} \right)
\quad \text{ as } T \to \infty,
\end{equation}
where $K \sim \log^* T := \lceil -\frac{\log T}{\log \lambda} \rceil$ and $x$ is such that it has zero average drift under iteration of $\psicover$,
and $\xi:\Xcover \to \Z$ is the projection on the $\Z$-part of the cover.

In this paper, we extend these results to (i) include higher-order error terms of
the asymptotics, making the $o(1)$ terms in \eqref{eq:adds} explicit, and (ii) allow more general translation surfaces than tori. For instance, we cover a particular case of Ehrenfest's wind-tree model.
Our proofs continue to rely on the renormalization formula \eqref{eq:renorm}, hence restricting the direction of the translation flow to quadratically irrational slopes, but are on the whole simpler than those of \cite{ADDS}, and pertain to $\Z^d$-covers as well.

Ergodicity of the flow seems to be a non-generic property;
there are several results in the literature showing that for
$\Z$- or $\Z^d$-extensions of many compact translation surfaces, the translation flow in a generic direction is non-ergodic, and even has uncountably many ergodic components, cf.\ \cite{RT12,RT17,BU18,FU14,FU23}.
The landmark result of Fr\kpol{a}czek and Ulcigrai \cite{FU14} proves the existence of uncountably many ergodic components for the square wind-tree model and Lebesgue a.e.\ direction.

In contrast, it was proved in \cite{H} for wind-tree models with square obstacles, and more generally for rational rectangles, that in a dense set of directions (of Hausdorff dimension more than half), the billiard flow is ergodic.

Our result applies in particular to wind-tree models that have two finite-horizon directions, where the corresponding cylinders have commensurable moduli.
It was proved in \cite{HLT} that this is the case for most rational-length obstacles of the rectangular wind-tree model — and the result of \cite{H} mentioned before is a consequence of this property.
We generalize ergodicity results to wind-trees of different shapes and to $\Z^d$-covers for higher $d$, which has been little studied to our knowledge. In addition, we provide finer ergodicity results in these cases by describing the asymptotic behavior of Birkhoff integrals.

Phrased in dimension $d=1$ (but see \Cref{thm:f} for the precise formulation for $d=1$ and $d=2$), our main result reads as follows:

\begin{thm}\label{thm:main1}
Let $G \in C^1(\Xcover)$ be compactly supported.
Then, there exist real bounded functions $g_{k,j}$ so that for all $N \ge 1$ and $\Leb$-a.e.\ $x \in \Xcover$,
\begin{align*}
\int_0^{T} G \circ \phi_t(x)\, \diff t &=\frac{\int_{\Xcover} G\, \diff\Leb}{\sigma \sqrt{2\pi}}
\cdot e^{-\frac{\xi\left(\psicover^{\KK}(x)\right)^2}{ 2 \sigma^2 \KK} } \, \frac{T}{\sqrt{\KK}}  \\
&\quad \times \left( 1 + \, \sum_{k=1}^N \frac{1}{K^k}  \sum_{j=0}^{2k}
 g_{k,j}(x) \, \xi(\psicover^{K}(x))^{2k-j} + \ O\left(\frac{1}{K^{N+1}} \right)
\right)
\end{align*}
as $T \to \infty$ and $K = \log^* T = \lceil -\frac{\log T}{\log \lambda} \rceil$.
\end{thm}

The term $\frac{\xi(\psicover^{\KK}(x))^2}{2\sigma^2 \KK}$ is oscillating and does not converge almost everywhere, but after integration over the space,
it does lead to a form of weak rational ergodicity for $C^1$ observables with optimal rates, see \Cref{thm:rateswrd}.
Weak rational ergodicity~\cite{A13} means that there is a set $\fund \subset \Xcover$ of positive finite measure (possibly but not necessarily a fundamental domain of the $\Z^d$-cover) such that
$$
\lim_{T\to\infty} \frac{1}{a_T(\fund)} \int_0^T \mu(A \cap \phi_t(B) ) \, \diff t = \mu(A) \mu(B)
$$
for all measurable sets $A,B \subset \fund$, and
$a_T(\fund) := \int_0^T \mu(\fund \cap \phi_t(\fund) ) \, dt$
is called the {\em return sequence}.
% Further, as in \cite{ADDS} and \cite{LS}, we obtain higher order rational ergodicity, see \Cref{thm:raterg}. The higher order rational ergodicity is defined as in the statement of \Cref{thm:raterg} below.

The paper is organized as follows.
In~\Cref{sec:covers}
we formalize the concept of $\Z^d$-cover over a translation surface
and study the automorphisms that commute with deck-transformations.
We discuss the example of the $(s,1)$-staircase at length, which is the direct generalization of the model used in \cite{ADDS} (where $s = 2$).
We give direct proofs of ergodicity of the pseudo-Anosov lift
$\psicover$ and the translation flow $\phi_t$ although this also follows
from the results of~\Cref{sec:limitlaws}.
\Cref{sec:covers} finishes with a version of the Ehrenfest wind-tree model,
as well as an example of the classical Ehrfest wind-tree model with $\frac12 \times \frac12$ squares as obstacles;
which is our examples of a $\Z^2$-cover where the main theory applies.
\Cref{sec:limitlaws} gives the core of the argument in an abstract setup, based on local limit laws of twisted transfer operators $\Lcomp_u$ acting on appropriate
 anisotropic Banach spaces.
 %Apart from some notation introduced earlier on, the results of \Cref{sec:limitlaws} are independent of \Cref{sec:covers}
% \Cref{sec:otherapporach} gives an alternative approach to introducing a twisted transfer operator which is closer to the nature of $\Z^d$-covers of the compact translation surface.
Finally, in the Appendix we review the tensor calculus we are using, and prove some technical lemma.
\\[4mm]
{\bf Acknowledgements:} The authors would like to thank the Erwin Schr\"odinger Institute
where this paper was initiated during a “Research in Teams” project in 2023.
We are also grateful to the anonymous referee for the suggestions that helped us to seriously strengthen our paper.

\section{Abelian covers and homogeneous automorphisms}\label{sec:covers}

A translation surface $\Xcomp$ is a connected topological surface with an atlas of charts $\phi_i:U_i \to \Xcomp \setminus \Sigma$, $U_i \subset \R^2$ open and connected, such that each $\phi_j^{-1} \circ \phi_i$ is a translation from
$U_i \cap \phi_i^{-1} \circ \phi_j(U_j)$ to its image.
Here $\Sigma$ is a discrete set of conical singularities, i.e., their cone angles are an integer multiple of $2\pi$ (but not $2\pi$ itself) and potentially marked point, see \Cref{sec:homological}.
See \cite{DHV} for an extensive monograph on translation surfaces, from which we have just paraphrased Definition 1.1.6.

A standard example is a right-angled polygon with pairwise identified sided via translations, see~\Cref{fig:31staircase} (left).
%
% \begin{figure}[ht]
% \begin{center}
% \begin{minipage}[l]{0.4\textwidth}
% \begin{tikzpicture}[scale=0.35]
% \draw[-] (-12,0) -- (24,0) -- (24,12) -- (-12,12) -- (-12,0);
% % \draw[-, dashed] (12,0) -- (12,12);
% % \draw[-, dashed] (0,0) -- (0,12);
% \node at (-12,0) {$\bullet$}; \draw[->] (-11, 0) arc (0:90:1); \node at (-11, 1) {\small $1$};
% \node at (0,0) {$\bullet$}; \draw[->] (1,0) arc (0:180:1); \node at (0, 1.5) {\small $4$};
% \node at (12,0) {$\bullet$}; \draw[->] (13,0) arc (0:180:1); \node at (12, 1.5) {\small $7$};
% \node at (24,0) {$\bullet$}; \draw[->] (24, 1) arc (90:180:1); \node at (23, 1) {\small $2$};
% \node at (-12,12) {$\bullet$}; \draw[->] (-12,11) arc (270:360:1); \node at (-11, 11) {\small $6$};
% \node at (0,12) {$\bullet$}; \draw[->] (-1,12) arc (180:360:1); \node at (0, 10.5) {\small $3$};
% \node at (12,12) {$\bullet$}; \draw[->] (11,12) arc (180:360:1); \node at (12, 10.5) {\small $8$};
% \node at (24,12) {$\bullet$}; \draw[->] (23,12) arc (180:270:1); \node at (23, 11) {\small $5$};
% %%%
% \draw[-] (5.8,11.7) -- (6.0,12.3);\draw[-] (5.8,-0.3) -- (6,0.3);
%   \node at (-6,0.1) {\small $\uparrow^{+1}$}; \node at (18,12.1) {\small $\uparrow^{+1}$};
% \draw[-] (6,11.7) -- (6.2,12.3);\draw[-] (6,-0.3) -- (6.2,0.3);
%   \node at (-6,12.1) {\small $\Uparrow^{-1}$};
%    \node at (18,0.1) {\small $\Uparrow^{-1}$};
% \end{tikzpicture}
% \end{minipage}
% \caption{The $(3,1)$-rectangle: $\Sigma$ consists of a single point with cone angle $6\pi$.}
% \label{fig:31}
% \end{center}
% \end{figure}
%
%

A $\Z^d$-cover $\Xcover$ of a translation surface is a new (infinite) translation surface with a continuous projection $p:\Xcover \to \Xcomp$
such that the quotient space $\Xcover/\Xcomp \simeq \Z^d$.
Any region $\fund$ such that copies by deck-transformations
$\{ \Delta_{\nb} \}_{\nb \in \Z^d}$ form a partition of $\Xcover$
is called a {\em fundamental domain}.
We can depict $\Xcover$ as an infinite polygon (the countable union of copies of $\Xcomp$) in $\R^2$, but with sides identified
in a different way from $\Xcomp$, see~\Cref{fig:31staircase}; a collection $\Gamma$ of $d$ independent primitive loops determines how these identifications are done, see~\Cref{fig:31staircase} (right).

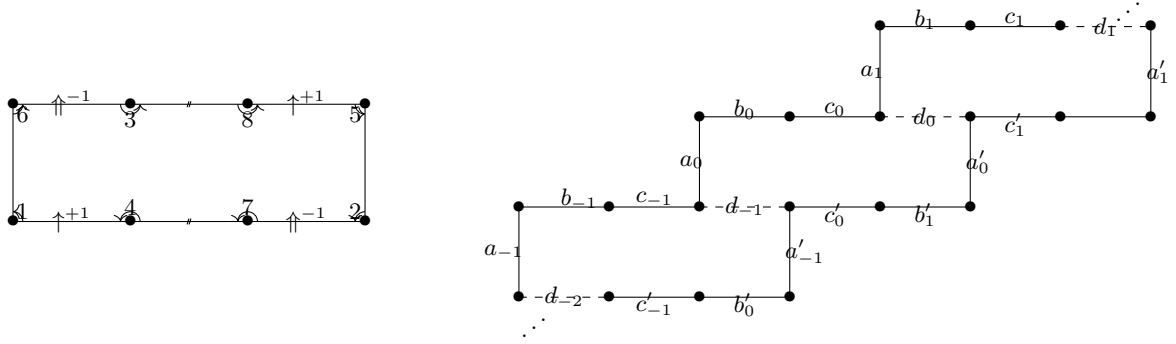
\begin{figure}[ht]
\begin{center}
\begin{minipage}[l]{0.4\textwidth}
\begin{tikzpicture}[scale=0.13]
\draw[-] (-12,0) -- (24,0) -- (24,12) -- (-12,12) -- (-12,0);
% \draw[-, dashed] (12,0) -- (12,12);
% \draw[-, dashed] (0,0) -- (0,12);
\node at (-12,0) {$\bullet$}; \draw[->] (-11, 0) arc (0:90:1); \node at (-11, 1) {\small $1$};
\node at (0,0) {$\bullet$}; \draw[->] (1,0) arc (0:180:1); \node at (0, 1.5) {\small $4$};
\node at (12,0) {$\bullet$}; \draw[->] (13,0) arc (0:180:1); \node at (12, 1.5) {\small $7$};
\node at (24,0) {$\bullet$}; \draw[->] (24, 1) arc (90:180:1); \node at (23, 1) {\small $2$};
\node at (-12,12) {$\bullet$}; \draw[->] (-12,11) arc (270:360:1); \node at (-11, 11) {\small $6$};
\node at (0,12) {$\bullet$}; \draw[->] (-1,12) arc (180:360:1); \node at (0, 10.5) {\small $3$};
\node at (12,12) {$\bullet$}; \draw[->] (11,12) arc (180:360:1); \node at (12, 10.5) {\small $8$};
\node at (24,12) {$\bullet$}; \draw[->] (23,12) arc (180:270:1); \node at (23, 11) {\small $5$};
%%%
\draw[-] (5.8,11.7) -- (6.0,12.3);\draw[-] (5.8,-0.3) -- (6,0.3);
  \node at (-6,0.1) {\small $\uparrow^{+1}$}; \node at (18,12.1) {\small $\uparrow^{+1}$};
\draw[-] (6,11.7) -- (6.2,12.3);\draw[-] (6,-0.3) -- (6.2,0.3);
  \node at (-6,12.1) {\small $\Uparrow^{-1}$};
   \node at (18,0.1) {\small $\Uparrow^{-1}$};
\end{tikzpicture}
\end{minipage}
\begin{minipage}[r]{0.55\textwidth}
\begin{tikzpicture}[scale=0.1]
\draw[-] (0,0) -- (24,0) -- (24,12);
\draw[-] (12,12) -- (-12,12) -- (-12,0);
\draw[-, dashed] (-12,0) -- (0,0);
\draw[-, dashed] (12,12) -- (24,12);
\draw[-] (24,12) -- (48,12) -- (48,24);
\draw[-] (12,12) -- (12,24) -- (36,24);
\draw[-] (-24,-12) -- (0,-12) -- (0,0);
\draw[-] (-12,0) -- (-36,0) -- (-36,-12);
\draw[-, dashed] (-36,-12) -- (-24,-12);
\draw[-, dashed] (36,24) -- (48,24);
 \node at (-13, 6) {\small $a_0$};
 \node at (25, 6) {\small $a'_0$};
  \node at (11, 18) {\small $a_1$};
  \node at (49, 18) {\small $a'_1$};
   \node at (-38, -6) {\small $a_{-1}$};
   \node at (2, -6) {\small $a'_{-1}$};
 \node at (-6, 13) {\small $b_0$};
 \node at (-6, -13) {\small $b'_0$};
  \node at (18, 25) {\small $b_1$};
  \node at (18, -1) {\small $b'_1$};
   \node at (-28, 1) {\small $b_{-1}$};
 %  \node at (-28, -25) {\small $b'_{-1}$};
\node at (6, 13) {\small $c_0$};
 \node at (6, -1) {\small $c'_0$};
  \node at (30, 25) {\small $c_1$};
  \node at (30, 11) {\small $c'_1$};
   \node at (-18, 1) {\small $c_{-1}$};
   \node at (-18, -13) {\small $c'_{-1}$};
\node at (18, 12) {\small $d_0$};
  \node at (42, 24) {\small $d_1$};
   \node at (-6, 0) {\small $d_{-1}$};
   \node at (-30, -12) {\small $d_{-2}$};
  \node at (45, 27) {\small \scalebox{-1}[1]{$\ddots$}};
   \node at (-34, -15) {\small \scalebox{-1}[1]{$\ddots$}};
\node at (-12,0) {$\bullet$}; %\draw[->] (-11, 0) arc (0:90:1); \node at (-11, 1) {\small $1$};
\node at (0,0) {$\bullet$}; %\draw[->] (1,0) arc (0:180:1); \node at (0, 1.5) {\small $4$};
\node at (12,0) {$\bullet$}; %\draw[->] (13,0) arc (0:180:1); \node at (12, 1.5) {\small $7$};
\node at (24,0) {$\bullet$}; %\draw[->] (24, 1) arc (90:180:1); \node at (23, 1) {\small $2$};
\node at (-12,12) {$\bullet$}; %\draw[->] (-12,11) arc (270:360:1); \node at (-11, 11) {\small $6$};
\node at (0,12) {$\bullet$}; %\draw[->] (-1,12) arc (180:360:1); \node at (0, 10.5) {\small $3$};
\node at (12,12) {$\bullet$}; %\draw[->] (11,12) arc (180:360:1); \node at (12, 10.5) {\small $8$};
\node at (24,12) {$\bullet$}; %\draw[->] (23,12) arc (180:270:1); \node at (23, 11) {\small $5$};
%%%
\node at (12,24) {$\bullet$};
\node at (24,24) {$\bullet$};
\node at (36,24) {$\bullet$};
\node at (48,24) {$\bullet$};
\node at (36,12) {$\bullet$};
\node at (48,12) {$\bullet$};
\node at (0,-12) {$\bullet$};
\node at (-12,-12) {$\bullet$};
\node at (-24,-12) {$\bullet$};
\node at (-36,-12) {$\bullet$};
\node at (-36,0) {$\bullet$};
\node at (-24,0) {$\bullet$};
% \draw[-] (5.8,11.7) -- (6.0,12.3);\draw[-] (5.8,-0.3) -- (6,0.3);
%   \node at (-6,0.1) {\small $\uparrow^{+1}$}; \node at (18,12.1) {\small $\uparrow^{+1}$};
% \draw[-] (6,11.7) -- (6.2,12.3);\draw[-] (6,-0.3) -- (6.2,0.3);
%   \node at (-6,12.1) {\small $\Uparrow^{-1}$};
%    \node at (18,0.1) {\small $\Uparrow^{-1}$};
\end{tikzpicture}
\end{minipage}
\caption{The $(3,1)$-rectangle (left) and the $(3,1)$-staircase (right).
On the left, $\Sigma$ consists of a single point with cone angle $6\pi$; it splits into countably many singularities on the staircase, all with cone angle $6\pi$.}
\label{fig:31staircase}
\end{center}
\end{figure}

A pseudo-Anosov diffeomorphism $\psicomp:\Xcomp \to \Xcomp$ is
a bijection of $\Xcomp$ that is a diffeomorphism on $\Xcomp \setminus \Sigma$, and admits a continuous splitting of the tangent bundle $T(\Xcomp \setminus \Sigma)$ into a stable foliation $\{ E^s(x) \}_{x \in \Xcomp \setminus \Sigma}$ and an unstable foliation $\{ E^u(x) \}_{x \in \Xcomp \setminus \Sigma}$ such that for all $x \in \Xcomp \setminus \Sigma$:
\begin{itemize}
 \item the angle $\angle(E^s(x) , E^u(x)) \geq \alpha$;
 \item $D\psicomp(x)(E^s(x)) = E^s(\psicomp(x))$ and
 $\| D\psicomp^n(x) v \| \leq C e^{-cn} \| v \|$ for all $v \in E^s(x)$;
 \item $D\psicomp^{-1}(x)(E^u(x)) = E^u(\psicomp^{-1}(x))$ and
 $\| D\psicomp^{-n}(x) v \| \leq C e^{-cn} \| v \|$ for all $v \in E^u(x)$.
\end{itemize}
Here the constants $C, c, \alpha > 0$ are independent of $x \in \Xcomp \setminus \Sigma$.
The set $\Sigma$ is discrete, and consists of common endpoints of so-called {\em prongs} of the
stable and unstable foliations.
In our case, $\Sigma$ coincides with the set of conical singularities of the translation surface.

The lift $\psicover$ of a pseudo-Anosov diffeomorphism $\psicomp$
is a bijection of the $\Z^d$-cover $\Xcover$ such that
$\psicomp \circ p = p \circ \psicover$.
It isn't fully a pseudo-Anosov diffeomorphism in its own right,
partially because the singularities of $\Xcover$ can be more complicated
than conical singularities, e.g., they could have infinite cone angle, or no proper two-dimensional neighbourhood.
However, $\psicover$ has enough hyperbolicity for the purpose in this paper.

%The pseudo-Anosov diffeomorphisms on the translation surfaces we are considering are piecewise affine, generated by a single hyperbolic matrix $A$.

\begin{figure}[ht]
\begin{center}
\begin{minipage}[l]{0.45\textwidth}
\begin{tikzpicture}[scale=0.04]
\filldraw[-, pink]  (0,0) -- (42,84) -- (84,182)  -- (42,98) -- (0,0);
\draw[-, red]  (0,0) -- (42,84) -- (84,182)  -- (42,98) -- (0,0);
\draw[-] (0,0) -- (42,0) -- (42,14) -- (0,14) -- (0,0);
\draw[-] (21,14.7) -- (21.2,14.3);\draw[-] (21,-0.3) -- (21.2,0.3);
\draw[-, dashed] (28,0) -- (28,14);
\draw[-, dashed] (14,0) -- (14,14);
\draw[-] (28,14) -- (70,14) -- (70,28) -- (28,28) -- (28,14);
\draw[-, dashed] (56,14) -- (56,28);
\draw[-, dashed] (42,14) -- (42,28);
\draw[-] (-28,-14) -- (14,-14) -- (14,0) -- (-28,0) -- (-28,-14);
\draw[-, dashed] (0,-14) -- (0,0);
\draw[-, dashed] (-14,-12) -- (-14,0);
%%%
\draw[-, blue]  (18,0) -- (24,14);
\filldraw[-, cyan]  (0,-14) -- (7,-14) -- (14,0) -- (28,0) -- (42,28) -- (35,28) -- (28,14) -- (14,14) -- (0,-14);
\draw[-, blue]  (0,-14) -- (14,14);
\draw[-, blue]  (6,-14) -- (18,14);
\draw[-, blue]  (7,-14) -- (21,14);
\draw[-, blue]  (21,0) -- (35,28);
\draw[-, blue]  (24,0) -- (36,28);
\draw[-, blue]  (28,0) -- (42,28);
\filldraw[-, cyan]  (35,0) -- (42,14) -- (42,0) -- (35,0);
\draw[-, blue]  (35,0) -- (42,14) -- (36,0) -- (35,0);
\draw[-] (0,0) -- (42,0) -- (42,14) -- (0,14) -- (0,0);
\draw[-] (21,14.7) -- (21.2,14.3);\draw[-] (21,-0.3) -- (21.2,0.3);
\draw[-, dashed] (28,0) -- (28,14);
\draw[-, dashed] (14,0) -- (14,14);
\draw[-] (28,14) -- (70,14) -- (70,28) -- (28,28) -- (28,14);
\draw[-, dashed] (56,14) -- (56,28);
\draw[-, dashed] (42,14) -- (42,28);
\draw[-] (-28,-14) -- (14,-14) -- (14,0) -- (-28,0) -- (-28,-14);
\draw[-, dashed] (0,-14) -- (0,0);
\draw[-, dashed] (-14,-12) -- (-14,0);
\end{tikzpicture}
\end{minipage}
\begin{minipage}[l]{0.45\textwidth}
\begin{tikzpicture}[scale=0.07]
\filldraw[-, cyan]  (0,0) -- (0,14) -- (7,14) -- (0,0);
\draw[-, blue]  (0,0) -- (6,14) -- (7,14) -- (0,0);
\filldraw[-, cyan]  (0,-14) -- (7,-14) -- (14,0) -- (28,0) -- (42,28) -- (35,28) -- (28,14) -- (14,14) -- (0,-14);
\draw[-, blue]  (0,-14) -- (14,14);
\draw[-, blue]  (6,-14) -- (18,14);
\draw[-, blue]  (7,-14) -- (21,14);
\draw[-, blue]  (18,0) -- (24,14);
\draw[-, blue]  (21,0) -- (35,28);
\draw[-, blue]  (24,0) -- (36,28);
\draw[-, blue]  (28,0) -- (42,28);
\filldraw[-, cyan]  (35,0) -- (42,14) -- (42,0) -- (35,0);
\draw[-, blue]  (35,0) -- (42,14) -- (36,0) -- (35,0);
%%%
\draw[-] (0,0) -- (42,0) -- (42,14) -- (0,14) -- (0,0);
   \node at (35,0.1) {\small $\uparrow^{+1}$};
   \node at (35,28.1) {\small $\uparrow^{+1}$};
\draw[-] (21,14.7) -- (21.2,14.3);\draw[-] (21,-0.3) -- (21.2,0.3);
   \node at (7,14.1) {\small $\Uparrow^{-1}$};
    \node at (7, -13.9) {\small $\Uparrow^{-1}$};
    \node at (0, 0) {\small $\bullet$};
\draw[-, dashed] (28,0) -- (28,14);
\draw[-, dashed] (14,0) -- (14,14);
\draw[-] (28,14) -- (70,14) -- (70,28) -- (28,28) -- (28,14);
\draw[-, dashed] (56,14) -- (56,28);
\draw[-, dashed] (42,14) -- (42,28);
\draw[-] (-28,-14) -- (14,-14) -- (14,0) -- (-28,0) -- (-28,-14);
\draw[-, dashed] (0,-14) -- (0,0);
\draw[-, dashed] (-14,-12) -- (-14,0);
\end{tikzpicture}
\end{minipage}

\caption{The $0$th step and its image under $\psicover$ with matrix $\binom{1 \ 3}{2 \ 7}$.}\label{fig:7}
\end{center}
\end{figure}

\subsection{Homological properties of $\Z^d$-covers}\label{sec:homological}

Let $\Xcomp$ be a compact translation surface of genus $g \geq 1$, and let $\Sigma \subset \Xcomp$ be the finite set of singularities and marked points of $\Xcomp$, whose cardinality we denote by $\kappa \geq 1$.
The relative homology $H_1(\Xcomp, \Sigma, \Z)$ is a free abelian group of rank $2g+\kappa-1$.
The intersection form
\[
\langle \cdot, \cdot \rangle \colon H_1(\Xcomp, \Sigma, \Z) \times H_1(\Xcomp \setminus \Sigma, \Z) \to \Z
\]
is non-degenerate.
Let us consider a homomorphism $\tilde\zeta : \pi_1(\Xcomp \setminus \Sigma) \to \Z^d$.
Using the action of the fundamental group $\pi_1$ on the universal cover of $\Xcomp$, denoted by $\widetilde \Xcomp$, one can associate to the kernel of $\tilde\zeta$ a $\Z^d$-cover $\ker \tilde \zeta \setminus \widetilde \Xcomp$.
Notice that, as $\Z^d$ is abelian, one can factor $\tilde\zeta$ by a morphism
\[
	\coverhomo : H_1(\Xcomp \setminus \Sigma, \Z) \to \Z^d.
\]
The projection of $\coverhomo$ on each coordinate of $\Z^d$ in the canonical basis defines linear forms on $H_1(\Xcomp \setminus \Sigma, \Z)$.
As the intersection form is non-degenerate, there exists a collection of independent primitive loops $\Gamma = \left\{\gamma_1, \dots, \gamma_d\right\} \subset H_1(\Xcomp, \Sigma, \Z)$ such that
\[
	\coverhomo(\gamma) = \left( \left\langle\gamma_1, \gamma\right\rangle, \dots, \left\langle\gamma_d, \gamma\right\rangle \right).
\]
Conversely, such a collection defines a $\Z^d$-cover $\Xcover$ of $\Xcomp$ with projection $\coverproj : \Xcover \to \Xcomp$.
We denote its group of deck transformation by $\Deck$, which is isomorphic to $\Z^d$.
Thus we can label each element of $\Deck$ by $\deck_{\nb}$ with $\nb \in \Z^d$.\\

As the intersection form is non-degenerate, there exists smooth 1-forms $\omega_i$ on $\Xcomp$ that vanish on $\Sigma$ such that for all $\gamma \in H_1(\Xcomp \setminus \Sigma, \Z)$,
\[
	\int_\gamma \omega_i = \left\langle \gamma_i, \gamma \right\rangle.
\]
Notice that the $\omega_i$s form a free family of vectors in $H^1(\Xcomp, \Sigma, \Z)$.
We denote by $\cH(\Z)$ the $\Z$-module of $H^1(\Xcomp, \Sigma, \Z)$ generated by these vectors and $\cH(\R)$ the corresponding real sub-bundle of $H^1(\Xcomp, \Sigma, \R)$.
Then the quotient $\cH(\R) / \cH(\Z)$ is isomorphic to the $d$-dimensional torus $\mathbb T^d$.\\

Let $\hol : H_1(\Xcomp, \Sigma, \Z) \to \C$ denote the holonomy map.
As observed in \cite{HW}, it is a necessary condition for the linear flow to be recurrent in almost every direction on the cover to have
\[
\hol(\gamma_i) = 0 \qquad \text{ for all $i=1,\dots, d$.}
\]
This is commonly called a {\em no drift condition} and we assume it for the covers we consider.
For $\Z$-covers, by \cite[Proposition 10]{HW}, the no drift condition as defined here is equivalent to the corresponding condition in (H3), namely that $\int_{\Xcomp} F \diff m = 0$. \\

\begin{prop}\label{prop:def_xi}
Let $x_0 \in \Xcover$ be fixed. For any $\omega \in \cH(\R)$ and for all $x \in \Xcover$, the integral
\[
\xi_\omega(x) = \int_x^{x_0} \omega \circ  \coverproj
\]
does not depend on the path chosen and is hence a well-defined smooth function on $\Xcover$.
\end{prop}
\begin{proof}
It suffices to show that for any loop $\gamma$ in the cover $\Xcover$, we have
$\int_{[\gamma]} \coverproj^{\ast} \omega = 0$.
By definition of the cover, $\coverproj_{\ast} [\gamma] = [p \circ \gamma] \in \ker \coverhomo$, so that
\[
\int_{[\gamma]} \coverproj^{\ast} \omega = \int_{\coverproj_{\ast} [\gamma]} \omega =0,
\]
which proves the claim.
\end{proof}

Let $\fund \subset \Xcover$ be a fundamental domain for the cover, namely a compact connected subset of
$\Xcover$, whose boundary $\partial \fund$ has measure zero, such that the restriction of the projection $p \colon \Xcover \to \Xcomp$ to the interior of $\fund$ is injective and $p(\fund) = \Xcomp$.

Note that %, for any fixed base point $x_0 \in \fund$, we have
$|\xi_\omega(x)-\xi_\omega(y)| \leq 1$ for any $x,y \in \fund$, and equality can hold only if $x,y \in \partial \fund$.

The orbit of $\fund$ under deck transformations tessellates $\Xcover$, namely, for all $\mb, \nb \in \Z^d$,
\[
	\Leb[\deck_{\mb}(\fund) \cap \deck_{\nb}(\fund)] = \delta_{\mb \, \nb}, \qquad \text{and} \qquad \bigcup_{\mathbf n \in \Z^d} \deck_{\mathbf n}(\fund) = \Xcover.
\]
For a given fundamental domain, one can associate a $\Z^d$-coordinate function from $\Xcover \setminus \bigcup_{\nb \in \Z^d}\partial \fund$ to $\Z^d$, which associates to a point $x \in \Xcover$ the unique $\nb \in \Z^d$ such that $\deck_{-\nb}(\fund) \in \fund$.\\

The following proposition shows that the functions $\xi_{\omega_i}$ form a smooth version of $\Z^d$-coordinates.
Its proof is left to the reader.
\begin{prop}\label{lem:coordinate}
	Let $x_0 \in \Xcover$ and $\omega \in \cH(\R)$, they induce a simply connected fundamental domain
	\[\mathcal F = \{x : |\xi_{\omega_i}(x)| \le 1 \text{ for }  1 \le i \le d\}\]
	and for any $x \in \Xcover$ such that $x$ is not in the orbit by the deck transformation of  $\partial \fund$,
\[
	x \in \deck_\nb(\mathcal{F}) \qquad \text{if and only if} \qquad \lfloor \xi_{\omega_i} (x) \rfloor = n_i \qquad \text{ for all $i=1\dots, d$}.
\]
\end{prop}

An automorphism of $\Xcover$ is called \textit{homogeneous} if it commutes with all deck transformations.
For such a homogeneous automorphism $\psicover$,  we can associate its \textit{Frobenius} function
\[
	F(x) := \xi(\psicover(x')) - \xi(x')
\]
where $x' \in p^{-1}(x)$ and $\xi$ is a $\Z^d$-coordinate for the cover.
This is well defined since for all $\nb \in \Z^d$, $\psicover \circ \deck_\nb = \deck_\nb \circ \psicover$ and $\xi \circ \deck_\nb = \xi + \nb$.
%and a change in $\Z^d$-coordinates changes the Frobenius function by a $\psicomp$-coboundary.\\

We then can define the \textit{average drift} of such an automorphism by
\[
	\delta(\psicover) := \int_{\Xcomp} F\, d\Leb
\]
where $\Leb$ is normalized Lebesgue measure on the surface.
It is independent of the choice of $\Z^d$-coordinates and if $\phi_\Gamma$ is another homogeneous automorphism, we have
\begin{equation}\label{drift_compat}
	\delta(\psicover \circ \phi_\Gamma) = \delta(\psicover) + \delta(\phi_\Gamma),
\end{equation}
see \cite[Lemma 2.2]{ADDS} for details.

\subsection{Homogeneous pseudo-Anosov lifts}\label{sec:pseudo-Anosov}

The next proposition determines when we can lift an automorphism on the base to an automorphism on the $\Z^d$-cover.

\begin{prop}\label{lem:psi_lift}
Let $\psicomp : \Xcomp \to \Xcomp$ be a linear automorphism which preserves $\Sigma$ and $\psicomp_*$ its induced map on $H_1(X \setminus \Sigma, \Z)$.
If $\zeta \circ \psicomp_* = \zeta$ then $\psi$ can be lifted to $\Xcover$ and its lift is homogeneous.
\end{prop}
\begin{proof}
	Recall that, by definition, automorphisms are induced on the universal cover by acting on the set of paths up to homotopy.
	The condition of the proposition then implies that this induced map can be factored through a quotient by $\ker \widetilde{\coverhomo}$, leading to the creation of a lifted automorphism on the cover.

	Since the cover inherits its flat structure from the pulled-back structure on $\Xcomp$, the lift is also linear.

	Moreover, not only does $\psicomp_*$ preserve the kernel of $\coverhomo$, but it also preserves all homology classes modulo the kernel of $\coverhomo$.
	As every deck transformation is determined by the action of an element $\gamma$ in $\pi_1(\Xcomp \setminus \Sigma) \bmod \ker \widetilde{\coverhomo}$,
	one simply has to notice that the condition implies $\psicomp (\gamma) \equiv \gamma \bmod \ker \widetilde{\coverhomo}$.
\end{proof}

In a translation surface, a closed orbit in a given direction is contained in a cylinder formed by a union of closed orbits and whose boundary is a union of saddle connections in the surface.
The length of those orbits is called the \emph{width} of the cylinder and the distance across the cylinder in the orthogonal direction if called the \emph{height}.
The modulus of a cylinder if given by the ratio of width over height.

A direction of a translation surface is called \emph{completely periodic}\footnote{The name comes from periodicity of the Teichm\"uller flow in this direction.} if it can be decomposed as a union of cylinders in that direction and those cylinders have commensurable moduli (i.e., their pairwise ratios are rational).
Notice that the actual period length of cylinders may be different and even not commensurable.

One feature of commensurable cylinder is that they have a common parabolic matrix acting trivially on them by a Dehn twist.

\begin{figure}[ht]
\begin{center}
\begin{tikzpicture}[scale=1.]
  % square
  \begin{scope}[xshift=4cm,yshift=1cm]
    \draw[postaction={decorate},decoration={ markings, mark=at position .5 with {\arrow{latex}}, } ] (0,0) -- (0,1);
    \draw[postaction={decorate},decoration={ markings, mark=at position .5 with {\arrow{latex}}, }] (3,0) -- (3,1);
    \draw (0,0) -- (3,0);
    \draw (0,1) -- (3,1);
    \node[label=right:$h$] at (3,.5) {};
    \node[label=below:$w$] at (1.5,0) {};

    \draw[->] (4,.5) -- (6,.5);
    \node[label=above:{$\begin{pmatrix} 1 & \frac{w}{h}\\ 0 & 1 \end{pmatrix}$}] at (5,.5) {};
  \end{scope}
\end{tikzpicture}
\begin{tikzpicture}[scale=1.]
\end{tikzpicture}
\begin{tikzpicture}[scale=1.]
  % square
  \begin{scope}[xshift=4cm,yshift=1cm]
    \draw[postaction={decorate},decoration={ markings, mark=at position .5 with {\arrow{latex}}, } ] (0,0) -- (3,1);
    \draw[postaction={decorate},decoration={ markings, mark=at position .5 with {\arrow{latex}}, }] (3,0) -- (6,1);
    \draw (0,0) -- (3,0);
    \draw (3,1) -- (6,1);
    \draw[->] (4.5,1) to [bend right] (1.5,1);
    \draw[dashed] (3,0) -- (3,1);
    \node[label={[rotate=+90]below:\ding{36}}] at (2.75,-.1) {};
  \end{scope}
\end{tikzpicture}
\caption{Dehn twist on an horizontal cylinder}
\label{fig:Dehn}
\end{center}
\end{figure}
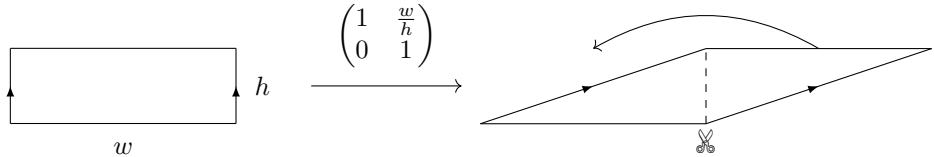

We apply the classical Thurston--Veech construction of pseudo-Anosov by product of two transverse Dehn twists.
We refer to \cite{lanneau} for an nice survey introduction of this construction or \cite[Section 3.4]{H} for an extensive study with generalization to infinite translation surfaces.
\begin{prop}
	\label{prop:twist}
	Consider a completely periodic direction of a translation surfaces $\Xcomp$ and the associated decomposition in a finite family of cylinders; let $\mu_1, \dots, \mu_n$ be their moduli and let $\eta_1, \dots, \eta_n$ be the homology classes of their core curves.
	Let $k_1, \dots, k_n \in \Z$ such that $k_1 \cdot \mu_1 = \dots = k_n \cdot \mu_n$.\\
	There is a linear parabolic automorphism $\psi$ on $\Xcomp$, which acts on homology by the map defined for all $\gamma \in H_1(\Xcomp \setminus \Sigma,~\R)$ by
	\begin{equation}
		\label{twist_action}
		\psi(\gamma) = \gamma + k_1 \cdot \langle \eta_1, \gamma \rangle \cdot \eta_1 + \dots + k_n \cdot \langle \eta_n, \gamma \rangle \cdot \eta_n.
	\end{equation}
\end{prop}

\begin{proof}
	Assume that the flow is in the horizontal direction.
	For a given cylinder let us denote by $w$, $h$ and $\mu = w/h$ its width, height and modulus.
	The action of the parabolic matrix $\binom{1 \ \mu}{0 \ 1}$
	defines an automorphism on this cylinder which topologically acts as a Dehn twist, meaning that it adds the homology class of the core curve of the cylinder each time a representing loop crosses the core curve positively and subtracts when it crosses it negatively.\\

	By assumption, the moduli $\mu_1, \dots, \mu_n$ of the cylinders %$\mathcal C_1, \dots, \mathcal C_n$
	decomposing the surface are commensurable.
	Thus there exist integers $k_1, \dots, k_n$ such that
	the parabolic matrix
	\[
	\begin{pmatrix}
		1 & \mu_1\\
		0 & 1
	\end{pmatrix}
	^{k_1}
	= \ \dots \ =
	\begin{pmatrix}
		1 & \mu_n\\
		0 & 1
	\end{pmatrix}
	^{k_n}
	\]
	defines an automorphism on the whole surface.
\end{proof}

In this work, we will be interested in products of Dehn twists in transverse directions in order to produce lifted linear automorphisms which associated matrix is hyperbolic \textit{i.e.} has trace larger than 2.
Such automorphisms are extensively studied on compact translation surfaces and are called linear \textit{pseudo-Anosov} maps.
They preserve two transverse foliations in their contracting and expanding directions.

\subsection{Staircases}\label{sec:staircases}

Let us consider a partition $\mathcal P_I$ of the interval $I$ into subintervals $I_1, I_2, \dots, I_n$ and a permutation
$\sigma : \{ 1, \dots, n\} \to \{ 1, \dots, n\}$.
Let us identify the vertical sides of the rectangle $I \times [0, 1]$ and, for all $1 \le i \le n$, interval $I_i$ at the top with interval $I_{\sigma(i)}$ at the bottom of the rectangle.

This defines a genus $g$ surface $\Xcomp$ with $k$ singularities and marked points at the boundaries of the intervals in the partition.
The numbers $g$ and $k$ depend on the permutation $\sigma$, but always satisfy
%$2g-1+k = n+1$ and
$\dim H_1(\Xcomp,\Sigma, \Z) = 2g+k-1=n+1$.
For $1 \le i \le n$, let $\eta_i \in H_1(\Xcomp, \Sigma, \Z)$ be a relative homology class corresponding to the path along the interval $I_i$ from left to right.
We define the classes $\gamma_i = \eta_{i} - \eta_{\sigma(i)}$;
they form a set $\Gamma$ and let $\Xcover$ be the corresponding $\Z^d$-cover.

An important example of compact translation surface cover in \cite{ADDS} and \cite{FU14}
is given by a rectangle with horizontal and vertical side length $s \in \N$ and $1$ respectively, and identifications
$(0,x) \sim (s,x)$, $(x,0) \sim (x+s-1,1)$, $(x,1) \sim (x+s-1,0)$,
for $x \in [0,1]$ and $(x,0) \sim (x,1)$ for $x \in [1,s-1]$.
It is a particular case of the above setting, with $n =3$, $|I_1| = |I_3| = 1$, $|I_2| = s-2$, $\sigma = (1 \, 3)$ and $\Gamma = \{ \gamma_3 - \gamma_1\}$.
The corresponding $\Z$-cover is called the $(s,1)$-{\em staircase}.

We formulate the next lemma for this set of examples, although it should be possible to adapt this argument for general cases.

\begin{lemma}\label{lem:rectanglecover}
	Let $\Xcover$ be the $\Z$-cover associated to the $(s,1)$-staircase.
	There exists an homogeneous linear pseudo-Anosov automorphism $\psicomp$ on $\Xcomp$ with zero average drift.
\end{lemma}

\begin{proof}
	By \Cref{prop:twist}, the action (from the right on row vectors) of the matrices
	\[
		D_h = \begin{pmatrix}
			1 & s \\
			0 & 1
		\end{pmatrix}
		\quad \text{and} \quad
		D_v = \begin{pmatrix}
			1 	& 0\\
			2 	& 1
		\end{pmatrix}
	\]
	on the surface define parabolic automorphisms (in fact, Dehn twists) in the horizontal and vertical directions.
Indeed, we note that the surface $\Xcomp$ can be decomposed into the union of two vertical cylinders (of widths 2 and 1 and heights 1 and $s-2$ respectively), or of a single horizontal cylinder (of width $s$ and height $1$).
	The homology class defining the cover can be represented by linear combination of horizontal paths
	(here $\gamma_3 - \gamma_1$), so the first automorphism preserves them.
	If $\eta$ is the homology class corresponding to the vertical side, the second automorphism maps $\gamma_k$ to $\gamma_k + 2 \cdot \eta$, for $k = 1$ or $3$.
	Therefore, it maps $\gamma_3 - \gamma_1$ to itself and both $D_h$ and $D_v$ preserve $\gamma_3 - \gamma_1$ and so any hyperbolic composition of these two linear maps defines a linear pseudo-Anosov lift which preserves the cover morphism.\\

	Notice moreover that their is a linear involution $\sigma$ with derivative $- Id$ on $\Xcomp$ which commutes with Dehn twists and such that $\sigma_* \Gamma = - \Gamma$.
	Hence, if we consider the Frobenius function corresponding to one of these Dehn twists, it satisfies $F \circ \sigma = - F$ and has thus zero average.
	This implies, by \Cref{drift_compat}, that the Frobenius function of a pseudo-Anosov automorphisms obtained as product of such Dehn twists also has zero average.
\end{proof}
%
% No claim is made that these are all the possible pseudo-Anosov automorphisms with these properties (up to homotopy); for example, there is an extra automorphism that turns the $(s,1)$-staircase by $\pi$.

\subsection{Wind-tree billiards}\label{sec:wind-tree}

In a cover, and more generally in any infinite translation surface,  a direction is said to have finite horizon if there is no infinite line in that direction.
One can generalize the Lemma~\ref{lem:rectanglecover} to $\Z^d$-covers with finite horizon.

\begin{lemma}
	If a $\Z^d$-cover $\Xcover$ has finite horizon and is completely periodic in two distinct directions then there exists a pseudo-Anosov map $\psicomp$ on $\Xcomp$ which preserves $\coverhomo$.
\end{lemma}

\begin{proof}
	Let us a consider a cylinder in $\Xcover$, its core curve $\gamma$ is a loop, thus the image of the curve by the projection $\coverproj$ must be in $\ker \coverhomo$.
	The projection of the cylinder to the base surface is again a cylinder whose core curve homology is an integer multiple of the class $[p_* \gamma]$.
	By \eqref{twist_action}, the Dehn twist acts trivially on $H_1(\Xcomp \setminus \Sigma, \Z)$ quotiented by $\ker \coverhomo$.
	Hence the Dehn twist on this cylinder preserves $\coverhomo$.

	If the cylinders have commensurable moduli, with common multiple $\mu$, this implies that there exists a linear automorphism $\binom{1 \ \mu}{0 \ 1}$
	in the corresponding direction on $\Xcomp$ which can be lifted to $\Xcover$.
	Thus if it is finite horizon and completely periodic in two distinct directions, the product of the corresponding two matrices is hyperbolic and the composition of the two Dehn twists produces the pseudo-Anosov automorphism we were looking for.
\end{proof}

\subsubsection{The wind-tree model with plus-shaped obstacles}
For illustration, we present a variation of the Ehrenfest wind-tree model as an example of a $\Z^2$-cover with finite horizon in vertical and horizontal directions.
Whereas such directions for rectangular wind-tree is less straightforward to illustrate.\\

\begin{figure}[ht]
\begin{center}
\begin{minipage}[t]{0.4\textwidth}
\begin{tikzpicture}[scale=0.45]
\filldraw[-, color=gray]  (1,6) -- (4,6) -- (4,7)--(1,7)--(1,6);
\filldraw[-, color=gray]  (2,5) -- (2,8) -- (3,8)--(3,5)--(2,5);
\draw[-]  (3,6) -- (4,6) -- (4,7)--(3,7)--(3,8)--(2,8)--(2,7)--(1,7)--(1,6)--(2,6)--(2,5)--(3,5)--(3,6);
\filldraw[-, color=gray]  (7,6) -- (10,6) -- (10,7)--(7,7)--(7,6);
\filldraw[-, color=gray]  (8,5) -- (8,8) -- (9,8)--(9,5)--(8,5);
\draw[-]  (9,6) -- (10,6) -- (10,7)--(9,7)--(9,8)--(8,8)--(8,7)--(7,7)--(7,6)--(8,6)--(8,5)--(9,5)--(9,6);
\filldraw[-, color=gray]  (4,8) -- (7,8) -- (7,9)--(4,9)--(4,8);
\filldraw[-, color=gray]  (5,7) -- (5,10) -- (6,10)--(6,7)--(5,7);
\draw[-]  (6,8) -- (7,8) -- (7,9)--(6,9)--(6,10)--(5,10)--(5,9)--(4,9)--(4,8)--(5,8)--(5,7)--(6,7)--(6,8);
\filldraw[-, color=gray]  (4,4) -- (7,4) -- (7,5)--(4,5)--(4,4);
\filldraw[-, color=gray]  (5,3) -- (5,6) -- (6,6)--(6,3)--(5,3);
\draw[-]  (6,4) -- (7,4) -- (7,5)--(6,5)--(6,6)--(5,6)--(5,5)--(4,5)--(4,4)--(5,4)--(5,3)--(6,3)--(6,4);
\filldraw[-, color=gray]  (10,8) -- (13,8) -- (13,9)--(10,9)--(10,8);
\filldraw[-, color=gray]  (11,7) -- (11,10) -- (12,10)--(12,7)--(11,7);
\draw[-]  (12,8) -- (13,8) -- (13,9)--(12,9)--(12,10)--(11,10)--(11,9)--(10,9)--(10,8)--(11,8)--(11,7)--(12,7)--(12,8);
\filldraw[-, color=gray]  (10,4) -- (13,4) -- (13,5)--(10,5)--(10,4);
\filldraw[-, color=gray]  (11,3) -- (11,6) -- (12,6)--(12,3)--(11,3);
\draw[-]  (12,4) -- (13,4) -- (13,5)--(12,5)--(12,6)--(11,6)--(11,5)--(10,5)--(10,4)--(11,4)--(11,3)--(12,3)--(12,4);
\filldraw[-, color=gray]  (7,2) -- (10,2) -- (10,3)--(7,3)--(7,2);
\filldraw[-, color=gray]  (8,1) -- (8,4) -- (9,4)--(9,1)--(8,1);
\draw[-]  (9,2) -- (10,2) -- (10,3)--(9,3)--(9,4)--(8,4)--(8,3)--(7,3)--(7,2)--(8,2)--(8,1)--(9,1)--(9,2);
\filldraw[-, color=gray]  (-2,4) -- (1,4) -- (1,5)--(-2,5)--(-2,4);
\filldraw[-, color=gray]  (-1,3) -- (-1,6) -- (0,6)--(0,3)--(-1,3);
\draw[-]  (0,4) -- (1,4) -- (1,5)--(0,5)--(0,6)--(-1,6)--(-1,5)--(-2,5)--(-2,4)--(-1,4)--(-1,3)--(0,3)--(0,4);
\filldraw[-, color=gray]  (-2,8) -- (1,8) -- (1,9)--(-2,9)--(-2,8);
\filldraw[-, color=gray]  (-1,7) -- (-1,10) -- (0,10)--(0,7)--(-1,7);
\draw[-]  (0,8) -- (1,8) -- (1,9)--(0,9)--(0,10)--(-1,10)--(-1,9)--(-2,9)--(-2,8)--(-1,8)--(-1,7)--(0,7)--(0,8);

\draw[-, dotted, line width=1.5]  (-.5,4.5)--(5.5,4.5)--(5.5,8.5)--(-.5,8.5)--(-.5,4.5);
\end{tikzpicture}
\end{minipage}
\qquad \quad
\begin{minipage}[t]{0.4\textwidth}
	\vspace{-5.cm}
\begin{tikzpicture}[scale=0.4]
	\begin{scope}[decoration={ markings, mark=at position 0.5 with {\arrow{stealth}}}]
\filldraw[-, color=gray]  (1,6) -- (4,6) -- (4,7)--(1,7)--(1,6);
\filldraw[-, color=gray]  (2,5) -- (2,8) -- (3,8)--(3,5)--(2,5);
\draw[-]  (3,6) -- (4,6) -- (4,7)--(3,7)--(3,8)--(2,8)--(2,7)--(1,7)--(1,6)--(2,6)--(2,5)--(3,5)--(3,6);
\draw[-]  (4,8.5)--(4,8)--(5,8)--(5,7)--(5.5,7);
\draw[-]  (5.5,6)--(5,6)--(5,5)--(4,5)--(4,4.5);
\draw[-]  (1,4.5) -- (1,5)--(0,5)--(0,6)--(-.5,6);
\draw[-]  (-.5,7)--(0,7)--(0,8) -- (1,8) -- (1,8.5);
\draw[dotted, color=blue, line width=2, postaction={decorate}]  (-.5,8.5)--(5.5,8.5);
\draw[dotted, color=blue, line width=2, postaction={decorate}]  (-.5,4.5)--(5.5,4.5);
\draw[dotted, color=red, line width=2, postaction={decorate}]  (-.5,4.5)--(-.5,8.5);
\draw[dotted, color=red, line width=2, postaction={decorate}]  (5.5,4.5)--(5.5,8.5);

\begin{scope}[shift={(10,0)}]
\filldraw[-, color=gray]  (1,6) -- (4,6) -- (4,7)--(1,7)--(1,6);
\filldraw[-, color=gray]  (2,5) -- (2,8) -- (3,8)--(3,5)--(2,5);
\draw[-]  (3,6) -- (4,6) -- (4,7)--(3,7)--(3,8)--(2,8)--(2,7)--(1,7)--(1,6)--(2,6)--(2,5)--(3,5)--(3,6);
\draw[-]  (4,8.5)--(4,8)--(5,8)--(5,7)--(5.5,7);
\draw[-]  (5.5,6)--(5,6)--(5,5)--(4,5)--(4,4.5);
\draw[-]  (1,4.5) -- (1,5)--(0,5)--(0,6)--(-.5,6);
\draw[-]  (-.5,7)--(0,7)--(0,8) -- (1,8) -- (1,8.5);
\draw[dotted, color=blue, line width=2, postaction={decorate}]  (-.5,8.5)--(5.5,8.5);
\draw[dotted, color=blue, line width=2, postaction={decorate}]  (-.5,4.5)--(5.5,4.5);
\draw[dotted, color=red, line width=2, postaction={decorate}]  (-.5,8.5)--(-.5,4.5);
\draw[dotted, color=red, line width=2, postaction={decorate}]  (5.5,8.5)--(5.5,4.5);
\end{scope}

\begin{scope}[shift={(0,-7)}]
\filldraw[-, color=gray]  (1,6) -- (4,6) -- (4,7)--(1,7)--(1,6);
\filldraw[-, color=gray]  (2,5) -- (2,8) -- (3,8)--(3,5)--(2,5);
\draw[-]  (3,6) -- (4,6) -- (4,7)--(3,7)--(3,8)--(2,8)--(2,7)--(1,7)--(1,6)--(2,6)--(2,5)--(3,5)--(3,6);
\draw[-]  (4,8.5)--(4,8)--(5,8)--(5,7)--(5.5,7);
\draw[-]  (5.5,6)--(5,6)--(5,5)--(4,5)--(4,4.5);
\draw[-]  (1,4.5) -- (1,5)--(0,5)--(0,6)--(-.5,6);
\draw[-]  (-.5,7)--(0,7)--(0,8) -- (1,8) -- (1,8.5);
\draw[dotted, color=blue, line width=2, postaction={decorate}]  (5.5,8.5)--(-.5,8.5);
\draw[dotted, color=blue, line width=2, postaction={decorate}]  (5.5,4.5)--(-.5,4.5);
\draw[dotted, color=red, line width=2, postaction={decorate}]  (-.5,4.5)--(-.5,8.5);
\draw[dotted, color=red, line width=2, postaction={decorate}]  (5.5,4.5)--(5.5,8.5);
\end{scope}

\begin{scope}[shift={(10,-7)}]
\filldraw[-, color=gray]  (1,6) -- (4,6) -- (4,7)--(1,7)--(1,6);
\filldraw[-, color=gray]  (2,5) -- (2,8) -- (3,8)--(3,5)--(2,5);
\draw[-]  (3,6) -- (4,6) -- (4,7)--(3,7)--(3,8)--(2,8)--(2,7)--(1,7)--(1,6)--(2,6)--(2,5)--(3,5)--(3,6);
\draw[-]  (4,8.5)--(4,8)--(5,8)--(5,7)--(5.5,7);
\draw[-]  (5.5,6)--(5,6)--(5,5)--(4,5)--(4,4.5);
\draw[-]  (1,4.5) -- (1,5)--(0,5)--(0,6)--(-.5,6);
\draw[-]  (-.5,7)--(0,7)--(0,8) -- (1,8) -- (1,8.5);
\draw[dotted, color=blue, line width=2, postaction={decorate}]  (5.5,8.5)--(-.5,8.5);
\draw[dotted, color=blue, line width=2, postaction={decorate}]  (5.5,4.5)--(-.5,4.5);
\draw[dotted, color=red, line width=2, postaction={decorate}]  (-.5,8.5)--(-.5,4.5);
\draw[dotted, color=red, line width=2, postaction={decorate}]  (5.5,8.5)--(5.5,4.5);
\end{scope}

\draw[dotted, line width = 1.5]  (7.5,10)--(7.5,-5);
\draw[dotted, line width = 1.5]  (-1.5,3.)--(16.5,3.);
\end{scope}

\end{tikzpicture}
\end{minipage}
\caption{Left: Wind-tree model with plus-shapes wind-tree model (and finite horizon in horizontal and vertical directions).
	A fundamental domain in dotted lines.
Right: Four copies forming a fundamental domain of the unfolded  translation surface of the plus-shaped wind-tree model.
}
\label{fig:pluswind}
\end{center}
\end{figure}
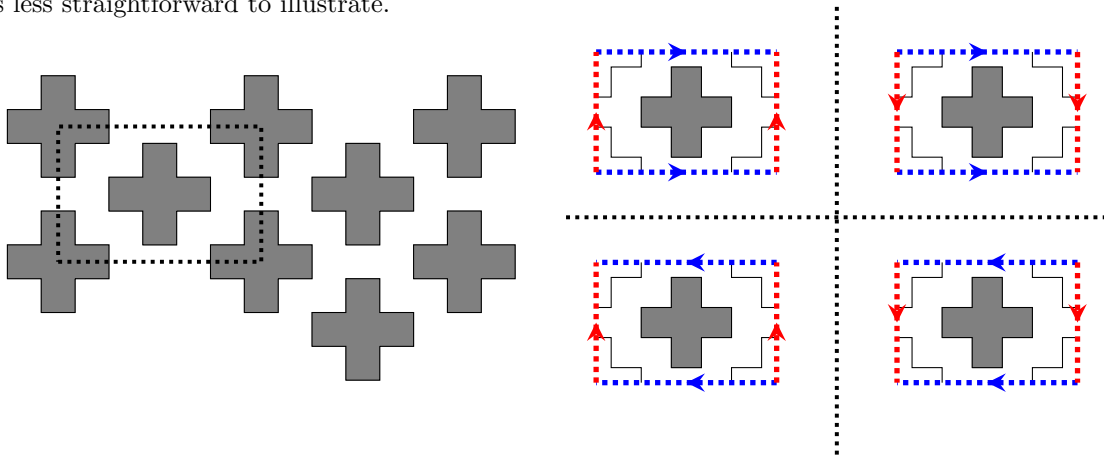

In the wind-tree model represented in \Cref{fig:pluswind}, the matrices
\[
D_h = 	\begin{pmatrix}
		1 & 12\\
		0 & 1
	\end{pmatrix}
	\qquad \text{and} \qquad
D_v =	\begin{pmatrix}
		1 & 0\\
		6 & 1
	\end{pmatrix}
\]
act as Dehn twists that commute with deck transformations.
These two maps generate a group of automorphisms that contain infinitely many pseudo-Anosov which also commute with deck transformations.\\

The classical construction to study the flow of wind-tree models is to \textit{unfold} the billiard in the torus into 4 copies.
So that each time the flow is bouncing on a side it is translated to the symmetric surface with respect to this flow.
The $\Z^2$-cover is thus defined by the homology classes $\gamma_v$ in red
(in \Cref{fig:Dehn}) for the first coordinate and $\gamma_h$ in blue for the second.
These correspond to $\gamma_1$ and $\gamma_2$ in the first section. \\

Notice that for symmetric obstacles, one has two automorphisms
$\tau_h$ and $\tau_v$  of the surface which exchange the two copies horizontally or vertically.
Notice that these automorphisms commute with the vertical and horizontal Dehn twists and
\[ (\tau_h)_* \gamma_h = \gamma_h, \quad (\tau_v)_* \gamma_v = \gamma_v, \quad (\tau_h)_* \gamma_v = -\gamma_v, \quad (\tau_v)_* \gamma_h = -\gamma_h.\]

\begin{lemma}\label{lem:FrobDehn1}
	The Frobenius function for the Dehn twists on the plus-shaped wind-tree models have zero average drift.
\end{lemma}

\begin{proof}
	Consider the horizontal Dehn twist $D_h$.
	It preserves the top two copies of the unfolded surface in \Cref{fig:pluswind} (right) and also the bottom two.
	But as the $\tau_v$ automorphism sends $\gamma_h$ to $-\gamma_h$, the Frobenius function for $D_h$ satisfies $F \circ \tau_h = - F$.
	Thus its integral on the surface is zero. The argument for the vertical Dehn twist $D_v$ is the same; it preserves the left two copies of the unfolded surface and also the right two.
\end{proof}

\subsubsection{The classical wind-tree model}

In the classical Ehrenfest wind-tree model, the obstacles are $a \times b$-rectangles centered and aligned with the lattice $\Z^2$.
Let us take $a = b = \frac12$, see \Cref{fig:fd1}.
All lines in north-east and north-west direction have finite horizon.
We can take cylinders in those direction, which (when lifted to the unfolding) have width $\frac14 \sqrt{2}$ and length $3\sqrt2$. Two of them cover the unfolded fundamental domain.
When lifted to the $\Z^2$-cover, they are still cylinders (i.e., they don't break up into strips, see \Cref{fig:pluswind}, and therefore the appropriate Dehn twists lift to the cover.

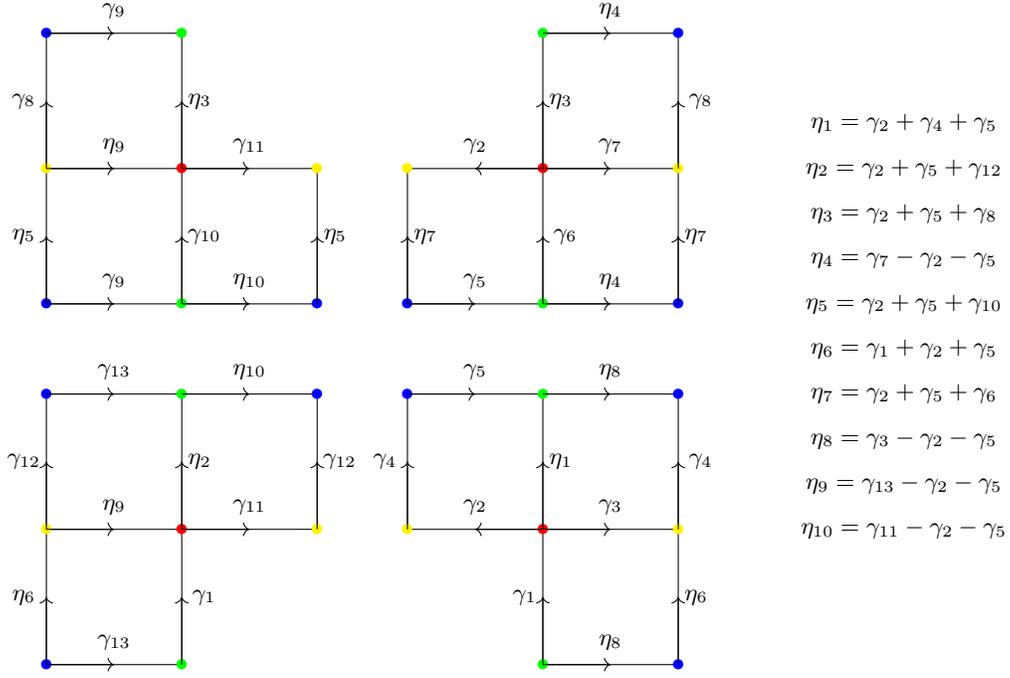
\begin{figure}[h]
\begin{center}
\begin{tikzpicture}[scale=0.6]
	\draw[-] (0,0)--(0,6)--(6,6)--(6,3)--(3,3)--(3,0)--(0,0);
	\draw[-] (8,6)--(14,6)--(14,0)--(11,0)--(11,3)--(8,3)--(8,6);
	\draw[-] (0,8)--(0,14)--(3,14)--(3,11)--(6,11)--(6,8)--(0,8);
	\draw[-] (8,8)--(8,11)--(11,11)--(11,14)--(14,14)--(14,8)--(8,8);
	\draw[-] (0,3)--(3,3)--(3,6);
	\draw[-] (0,11)--(3,11)--(3,8);
	\draw[-] (11,8)--(11,11)--(14,11);
	\draw[-] (11,6)--(11,3)--(14,3);
	\node at (3,3) {{\color{red}$\bullet$}};
	\node at (11,3) {{\color{red}$\bullet$}};
	\node at (3,11) {{\color{red}$\bullet$}};
	\node at (11,11) {{\color{red}$\bullet$}};
	\node at (3,0) {{\color{green}$\bullet$}};
	\node at (3,6) {{\color{green}$\bullet$}};
	\node at (3,8) {{\color{green}$\bullet$}};
	\node at (3,14) {{\color{green}$\bullet$}};
	\node at (11,0) {{\color{green}$\bullet$}};
	\node at (11,6) {{\color{green}$\bullet$}};
	\node at (11,8) {{\color{green}$\bullet$}};
	\node at (11,14) {{\color{green}$\bullet$}};
	\node at (0,14) {{\color{blue}$\bullet$}};
	\node at (0,8) {{\color{blue}$\bullet$}};
	\node at (0,6) {{\color{blue}$\bullet$}};
	\node at (0,0) {{\color{blue}$\bullet$}};
	\node at (6,8) {{\color{blue}$\bullet$}};
	\node at (6,6) {{\color{blue}$\bullet$}};
	\node at (8,8) {{\color{blue}$\bullet$}};
	\node at (8,6) {{\color{blue}$\bullet$}};
	\node at (14,14) {{\color{blue}$\bullet$}};
	\node at (14,8) {{\color{blue}$\bullet$}};
	\node at (14,6) {{\color{blue}$\bullet$}};
	\node at (14,0) {{\color{blue}$\bullet$}};
	\node at (0,3) {{\color{yellow}$\bullet$}};
	\node at (0,11) {{\color{yellow}$\bullet$}};
	\node at (6,3) {{\color{yellow}$\bullet$}};
	\node at (6,11) {{\color{yellow}$\bullet$}};
	\node at (8,3) {{\color{yellow}$\bullet$}};
	\node at (8,11) {{\color{yellow}$\bullet$}};
	\node at (14,3) {{\color{yellow}$\bullet$}};
	\node at (14,11) {{\color{yellow}$\bullet$}};
	\draw[->] (11,0)--(11,1.5); \node at (10.6,1.5) {\small $\gamma_1$};
	\draw[->] (3,0)--(3,1.5); \node at (3.5,1.5) {\small $\gamma_1$};
	\draw[->] (11,3)--(9.5,3); \node at (9.5,3.5) {\small $\gamma_2$};
	\draw[->] (11,11)--(9.5,11); \node at (9.5,11.5) {\small $\gamma_2$};
	\draw[->] (11,3)--(12.5,3); \node at (12.5,3.5) {\small $\gamma_3$};
	\draw[->] (14,3)--(14,4.5); \node at (14.5,4.5) {\small $\gamma_4$};
	\draw[->] (8,3)--(8,4.5); \node at (7.5,4.5) {\small $\gamma_4$};
	\draw[->] (8,6)--(9.5,6); \node at (9.5,6.5) {\small $\gamma_5$};
	\draw[->] (8,8)--(9.5,8); \node at (9.5,8.5) {\small $\gamma_5$};
	\draw[->] (11,8)--(11,9.5); \node at (11.5,9.5) {\small $\gamma_6$};
	\draw[->] (11,11)--(12.5,11); \node at (12.5,11.5) {\small $\gamma_7$};
	\draw[->] (14,11)--(14,12.5); \node at (14.5,12.5) {\small $\gamma_8$};
	\draw[->] (0,11)--(0,12.5); \node at (-0.5,12.5) {\small $\gamma_8$};
	\draw[->] (0,14)--(1.5,14); \node at (1.5,14.5) {\small $\gamma_9$};
	\draw[->] (0,8)--(1.5,8); \node at (1.5,8.5) {\small $\gamma_9$};
	\draw[->] (3,8)--(3,9.5); \node at (3.5,9.5) {\small $\gamma_{10}$};
	\draw[->] (3,11)--(4.5,11); \node at (4.5,11.5) {\small $\gamma_{11}$};
	\draw[->] (3,3)--(4.5,3); \node at (4.5,3.5) {\small $\gamma_{11}$};
	\draw[->] (6,3)--(6,4.5); \node at (6.5,4.5) {\small $\gamma_{12}$};
	\draw[->] (0,3)--(0,4.5); \node at (-0.5,4.5) {\small $\gamma_{12}$};
	\draw[->] (0,6)--(1.5,6); \node at (1.5,6.5) {\small $\gamma_{13}$};
	\draw[->] (0,0)--(1.5,0); \node at (1.5,0.5) {\small $\gamma_{13}$};
	\draw[->] (11,3)--(11,4.5); \node at (11.4,4.5) {\small $\eta_1$};
    \draw[->] (3,3)--(3,4.5); \node at (3.4,4.5) {\small $\eta_2$};
    \draw[->] (11,11)--(11,12.5); \node at (11.4,12.5) {\small $\eta_3$};
    \draw[->] (3,11)--(3,12.5); \node at (3.4,12.5) {\small $\eta_3$};
    \draw[->] (14,8)--(14,9.5); \node at (14.4,9.5) {\small $\eta_7$};
    \draw[->] (0,8)--(0,9.5); \node at (-0.5,9.5) {\small $\eta_5$};
      \draw[->] (14,0)--(14,1.5); \node at (14.4,1.5) {\small $\eta_6$};
    \draw[->] (0,0)--(0,1.5); \node at (-0.5,1.5) {\small $\eta_6$};
    \draw[->] (8,8)--(8,9.5); \node at (8.4,9.5) {\small $\eta_7$};
    \draw[->] (6,8)--(6,9.5); \node at (6.4,9.5) {\small $\eta_5$};
    \draw[->] (11,0)--(12.5,0); \node at (12.5,0.5) {\small $\eta_8$};
    \draw[->] (11,14)--(12.5,14); \node at (12.5,14.5) {\small $\eta_{4}$};
    \draw[->] (0,3)--(1.5,3); \node at (1.5,3.5) {\small $\eta_9$};
    \draw[->] (0,11)--(1.5,11); \node at (1.5,11.5) {\small $\eta_9$};
    \draw[->] (3,6)--(4.5,6); \node at (4.5,6.5) {\small $\eta_{10}$};
    \draw[->] (3,8)--(4.5,8); \node at (4.5,8.5) {\small $\eta_{10}$};
    \draw[->] (11,6)--(12.5,6); \node at (12.5,6.5) {\small $\eta_8$};     \draw[->] (11,8)--(12.5,8); \node at (12.5,8.5) {\small $\eta_4$};
%%%%%%%
\node at (19,12) {\small $\eta_1 = \gamma_2+\gamma_4+\gamma_5$};
\node at (19,11) {\small $\eta_2 = \gamma_2+\gamma_5+\gamma_{12}$};
\node at (19,10) {\small $\eta_3 = \gamma_2+\gamma_5+\gamma_8$};
\node at (19,9) {\small $\eta_4 = \gamma_7-\gamma_2-\gamma_5$};
\node at (19,8) {\small $\eta_5 = \gamma_2+\gamma_5+\gamma_{10}$};
\node at (19,7) {\small $\eta_6 = \gamma_1+\gamma_2+\gamma_5$};
\node at (19,6) {\small $\eta_7 = \gamma_2+\gamma_5+\gamma_6$};
\node at (19,5) {\small $\eta_8 = \gamma_3-\gamma_2-\gamma_5$};
\node at (19,4) {\small $\eta_9 =  \gamma_{13}-\gamma_2-\gamma_5$};
\node at (19,3) {\small $\eta_{10} = \gamma_{11}-\gamma_2-\gamma_5$};
\end{tikzpicture}
\end{center}
\caption{A fundamental domain for the surface $X$. The four colored dots represents the four singularities in $\Sigma$, each of angle $6\pi$. The elements $\gamma_1, \dots, \gamma_{13} \in H(X,\Sigma,\Z)$ form a basis of the relative homology: the homology class of any oriented path connecting two elements of $\Sigma$ can be written as a linear combination with integer coefficients of $\gamma_1, \dots, \gamma_{13}$ (for example, the path marked as $\eta_1$ corresponds to the homology class $\gamma_2+\gamma_4+\gamma_5$, since the concatenation of the paths $\gamma_2, \gamma_4,\gamma_5$ and $-\eta_1$ bounds a disk and hence is trivial in homology).}
\label{fig:fd1}
\end{figure}

\begin{figure}[t]
\begin{center}
	\begin{tikzpicture}[scale=0.6]
		\filldraw[dashed,pink] (0,0)--(0,2)--(2,4)--(4,4);
		\filldraw[dashed,pink] (2,5)--(4,5)--(4,7);
		\filldraw[dashed,pink] (0,5)--(0,7)--(2,9)--(2,7);
		\filldraw[dashed,pink] (5,5)--(9,9)--(7,9)--(7,7)--(5,7);
		\filldraw[dashed,pink] (7,5)--(9,5)--(9,7);
		\filldraw[dashed,pink] (5,2)--(7,4)--(9,4)--(7,2);
		\filldraw[dashed,pink] (7,0)--(9,0)--(9,2);
		\filldraw[dashed,green!20] (0,5)--(2,5)--(4,7)--(2,7);
		\filldraw[dashed,green!20] (0,7)--(0,9)--(2,9);
		\filldraw[dashed,green!20] (0,2)--(0,4)--(2,4);
		\filldraw[dashed,green!20] (0,0)--(4,4)--(4,2)--(2,2)--(2,0);
		\filldraw[dashed,green!20] (5,2)--(5,4)--(7,4);
		\filldraw[dashed,green!20] (7,0)--(7,2)--(9,4)--(9,2);
		\filldraw[dashed,green!20] (5,5)--(9,9)--(9,7)--(7,5);
		\filldraw[dashed,blue!20] (18,2)--(16,2)--(14,4)--(16,4);
		\filldraw[dashed,blue!20] (14,0)--(14,2)--(16,0);
		\filldraw[dashed,blue!20] (19,4)--(23,0)--(23,2)--(21,4);
		\filldraw[dashed,blue!20] (18,5)--(14,9)--(16,9)--(16,7)--(18,7);
		\filldraw[dashed,blue!20] (14,5)--(14,7)--(16,5);
		\filldraw[dashed,blue!20] (23,5)--(23,7)--(21,9)--(21,7);
		\filldraw[dashed,blue!20] (19,5)--(19,7)--(21,5);
		\filldraw[dashed,orange!20] (14,2)--(14,4)--(16,2)--(16,0);
		\filldraw[dashed,orange!20] (14,9)--(14,7)--(16,5)--(18,5);
		\filldraw[dashed,orange!20] (16,4)--(18,4)--(18,2);
		\filldraw[dashed,orange!20] (23,0)--(19,4)--(19,2)--(21,2)--(21,0);
		\filldraw[dashed,orange!20] (23,9)--(21,9)--(23,7);
		\filldraw[dashed,orange!20] (19,7)--(21,7)--(23,5)--(21,5);
		\filldraw[dashed,orange!20] (23,4)--(21,4)--(23,2);
		\draw[-] (0,0)--(0,4)--(4,4)--(4,2)--(2,2)--(2,0)--(0,0);
		\draw[-] (5,4)--(9,4)--(9,0)--(7,0)--(7,2)--(5,2)--(5,4);
		\draw[-] (0,5)--(0,9)--(2,9)--(2,7)--(4,7)--(4,5)--(0,5);
		\draw[-] (5,5)--(5,7)--(7,7)--(7,9)--(9,9)--(9,5)--(5,5);
		\draw[-] (0,2)--(2,2)--(2,4);
		\draw[-] (0,7)--(2,7)--(2,5);
		\draw[-] (7,5)--(7,7)--(9,7);
		\draw[-] (7,4)--(7,2)--(9,2);

		\draw[-] (0+14,0)--(0+14,4)--(4+14,4)--(4+14,2)--(2+14,2)--(2+14,0)--(0+14,0);
		\draw[-] (5+14,4)--(9+14,4)--(9+14,0)--(7+14,0)--(7+14,2)--(5+14,2)--(5+14,4);
		\draw[-] (0+14,5)--(0+14,9)--(2+14,9)--(2+14,7)--(4+14,7)--(4+14,5)--(0+14,5);
		\draw[-] (5+14,5)--(5+14,7)--(7+14,7)--(7+14,9)--(9+14,9)--(9+14,5)--(5+14,5);
		\draw[-] (0+14,2)--(2+14,2)--(2+14,4);
		\draw[-] (0+14,7)--(2+14,7)--(2+14,5);
		\draw[-] (7+14,5)--(7+14,7)--(9+14,7);
		\draw[-] (7+14,4)--(7+14,2)--(9+14,2);
		\draw[-,red, thick] (0,2)--(2,4);
		\draw[-,red, thick] (0,7)--(2,9);
		\draw[-,red, thick] (2,5)--(4,7);
		\draw[-,red, thick] (7,5)--(9,7);
		\draw[-,red, thick] (5,2)--(7,4);
		\draw[-,red, thick] (7,0)--(9,2);
		\draw[-,green, thick] (0,0)--(4,4);
		\draw[-,green, thick] (0,5)--(2,7);
		\draw[-,green, thick] (5,5)--(9,9);
		\draw[-,green, thick] (7,2)--(9,4);
         \draw[-,red, thick] (0.1,2)--(2.1,4);
		\draw[-,red, thick] (0.1,7)--(2.1,9);
		\draw[-,red, thick] (2.1,5)--(4.1,7);
		\draw[-,red, thick] (7.1,5)--(9.1,7);
		\draw[-,red, thick] (5.1,2)--(7.1,4);
		\draw[-,red, thick] (7.1,0)--(9.1,2);
		\draw[-,green, thick] (0.1,0)--(4.1,4);
		\draw[-,green, thick] (0.1,5)--(2.1,7);
		\draw[-,green, thick] (5.1,5)--(9.1,9);
		\draw[-,green, thick] (7.1,2)--(9.1,4);
		\draw[-,blue, thick] (16,0)--(14,2);
		\draw[-,blue, thick] (18,2)--(16,4);
		\draw[-,blue, thick] (16,5)--(14,7);
		\draw[-,blue, thick] (21,5)--(19,7);
		\draw[-,blue, thick] (23,7)--(21,9);
		\draw[-,blue, thick] (23,2)--(21,4);
		\draw[-,orange, thick] (16,2)--(14,4);
		\draw[-,orange, thick] (18,5)--(14,9);
		\draw[-,orange, thick] (23,5)--(21,7);
		\draw[-,orange, thick] (23,0)--(19,4);
         \draw[-,blue, thick] (16.1,0)--(14.1,2);
		\draw[-,blue, thick] (18.1,2)--(16.1,4);
		\draw[-,blue, thick] (16.1,5)--(14.1,7);
		\draw[-,blue, thick] (21.1,5)--(19.1,7);
		\draw[-,blue, thick] (23.1,7)--(21.1,9);
		\draw[-,blue, thick] (23.1,2)--(21.1,4);
		\draw[-,orange, thick] (16.1,2)--(14.1,4);
		\draw[-,orange, thick] (18.1,5)--(14.1,9);
		\draw[-,orange, thick] (23.1,5)--(21.1,7);
		\draw[-,orange, thick] (23.1,0)--(19.1,4);
		\end{tikzpicture}
	\end{center}
\caption{Two copies of the fundamental domain for the surface $X$ with a cylinder decomposition for the north-east Dehn twist $D_{ne}$ (left) and for the north-west Dehn twist $D_{nw}$ (right). The closed paths in red and green on the left correspond to the same element $\gamma_{ne} \in H_1(X,\Sigma,\Z)$; similarly, the closed paths in blue and orange on the right correspond to the same element $\gamma_{nw} \in H_1(X,\Sigma,\Z)$.}
\label{fig:fd2}
\end{figure}
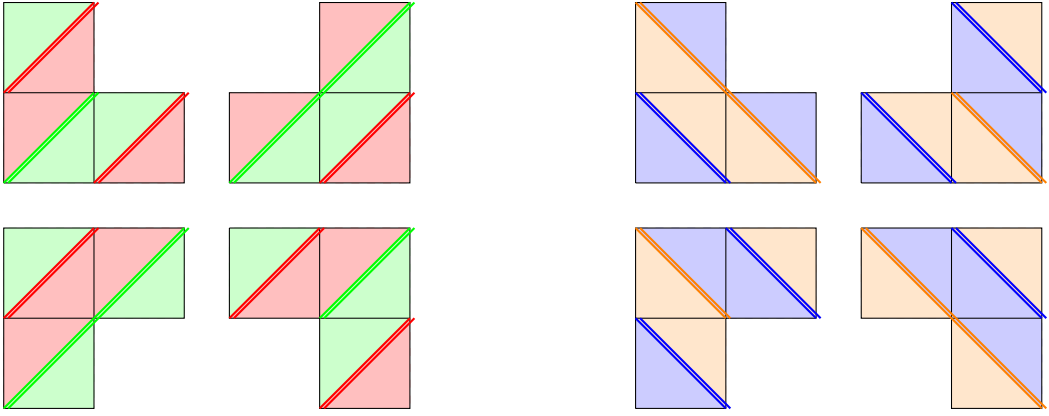

As can be seen from \Cref{fig:fd1}, the unfolded fundamental domain has four singularities, with cone angle $6\pi$.
A Dehn twist that shears by twelve units, maps these two cylinders back to themselves in a way that extends continuously over both cylinders. The affine parts of these Dehn twists, in the north-east and north-west directions,
are represented by the matrices
$$
\begin{pmatrix}
          1+\frac{t}{2} & -\frac{t}{2} \\[2mm] \frac{t}{2} &  1-\frac{t}{2}
         \end{pmatrix}
\quad \text{ and } \quad
\begin{pmatrix}
          1+\frac{t}{2} & \frac{t}{2} \\[2mm] -\frac{t}{2} &  1-\frac{t}{2}
         \end{pmatrix},
         \qquad t = 12.
$$
We now describe the action of the north-east and of the north-west Dehn twists $D_{ne}$ and $D_{nw}$ on the relative homology $H_1(X,\Sigma,\Z)$ as follows: let us fix the basis $\gamma_1, \dots, \gamma_{13}$ as in \Cref{fig:fd1} (note that the rank of $H_1(X,\Sigma,\Z)$ is $2\cdot 5 + 4-1 = 13$, where $5$ is the genus of $X$ and $4$ is the number of singularities).
With respect to this basis, the cover is generated by the pair $\Gamma = \{-\gamma_3+\gamma_7+\gamma_9-\gamma_{13}, -\gamma_4 -\gamma_6+\gamma_{10}+\gamma_{12}\}$.
We define
$$
\gamma_{ne} := \gamma_1+\gamma_3+\gamma_4+\gamma_5+\gamma_6+\gamma_7+\gamma_8+\gamma_9+\gamma_{10}+\gamma_{11}+\gamma_{12}+\gamma_{13},
$$
and
$$
\gamma_{nw} := \gamma_1+6\gamma_2-\gamma_3+\gamma_4+5\gamma_5+\gamma_6-\gamma_7+\gamma_8-\gamma_9+\gamma_{10}-\gamma_{11}+\gamma_{12}-\gamma_{13}.
$$
The action of the north-east Dehn-twist on $H_1(X,\Sigma,\Z)$ is given by
(see as in \Cref{fig:fd2})
$$
\begin{cases}
	\gamma_i \mapsto \gamma_i + \gamma_{ne} & \text{for } i\in \{1,2,4,6,8,10,12\}, \\
	\gamma_j \mapsto \gamma_j -  \gamma_{ne} & \text{for } j\in \{3,5,7,9,11,13\},
\end{cases}
$$
so that
$$
D_{ne} =
\begin{pmatrix}
2 & 1 & -1 & 1 & -1 & 1 & -1 & 1 & -1 & 1 & -1 & 1 & -1 \\
0 & 1 & 0 & 0 & 0 & 0 & 0 & 0 & 0 & 0 & 0 & 0 & 0 \\
1 & 1 & 0 & 1 & -1 & 1 & -1 & 1 & -1 & 1 & -1 & 1 & -1 \\
1 & 1 & -1 & 2 & -1 & 1 & -1 & 1 & -1 & 1 & -1 & 1 & -1 \\
1 & 1 & -1 & 1 & 0 & 1 & -1 & 1 & -1 & 1 & -1 & 1 & -1 \\
1 & 1 & -1 & 1 & -1 & 2 & -1 & 1 & -1 & 1 & -1 & 1 & -1 \\
1 & 1 & -1 & 1 & -1 & 1 & 0 & 1 & -1 & 1 & -1 & 1 & -1 \\
1 & 1 & -1 & 1 & -1 & 1 & -1 & 2 & -1 & 1 & -1 & 1 & -1 \\
1 & 1 & -1 & 1 & -1 & 1 & -1 & 1 & 0 & 1 & -1 & 1 & -1 \\
1 & 1 & -1 & 1 & -1 & 1 & -1 & 1 & -1 & 2 & -1 & 1 & -1 \\
1 & 1 & -1 & 1 & -1 & 1 & -1 & 1 & -1 & 1 & 0 & 1 & -1 \\
1 & 1 & -1 & 1 & -1 & 1 & -1 & 1 & -1 & 1 & -1 & 2 & -1 \\
1 & 1 & -1 & 1 & -1 & 1 & -1 & 1 & -1 & 1 & -1 & 1 & 0 \\
\end{pmatrix}
$$
The north-west Dehn-twist acts on $H_1(X,\Sigma,\Z)$ as
$$
\begin{cases}
 \gamma_i \mapsto \gamma_i + \gamma_{nw} & \text{for } i \neq 2, \\
 \gamma_2 \mapsto \gamma_2 - \gamma_{nw},
\end{cases}
$$
so that
$$
D_{nw} =
\begin{pmatrix}
2 & -1 & 1 & 	1 & 1 & 1 & 	1 & 1 & 1 & 		1 & 1 & 1 & 1 \\
6 & -5 & 6 & 	6 & 6 & 6 & 	6 & 6 & 6 & 		6 & 6 & 6 & 6 \\
-1 & 1 & 0 & 	-1 & -1 & -1 & 	-1 & -1 & -1 & 		-1 & -1 & -1 & -1 \\
1 & -1 & 1 & 	2 & 1 & 1 & 	1 & 1 & 1 & 		1 & 1 & 1 & 1 \\
5 & -5 & 5 & 	5 & 6 & 5 & 	5 & 5 & 5 & 		5 & 5 & 5 & 5 \\
1 & -1 & 1 & 	1 & 1 & 2 & 	1 & 1 & 1 & 		1 & 1 & 1 & 1 \\
-1 & 1 & -1 & 	-1 & -1 & -1 & 	0 & 1 & -1 & 		-1 & -1 & -1 & -1 \\
1 & -1 & 1 & 	1 & 1 & 1 & 	1 & 2 & 1 & 		1 & 1 & 1 & 1 \\
-1 & 1 & -1 & 	-1 & -1 & -1 & 	-1 & -1 & 0 & 		-1 & -1 & -1 & -1 \\
1 & -1 & 1 & 	1 & 1 & 1 & 	1 & 1 & 1 & 		2 & 1 & 1 & 1 \\
-1 & 1 & -1 & 	-1 & -1 & -1 & 	-1 & -1 & -1 & 		-1 & 0 & -1 & -1 \\
1 & -1 & 1 & 	1 & 1 & 1 & 	1 & 1 & 1 & 		1 & 1 & 2 & 1 \\
-1 & 1 & -1 & 	-1 & -1 & -1 & 	-1 & -1 & -1 & 		-1 & -1 & -1 & 0 \\
\end{pmatrix}
$$

\medskip

\begin{lemma}\label{lem:Fwind-tree}
 The Frobenius functions for the Dehn twists $D_{ne}$ and $D_{nw}$ have zero average drift.
\end{lemma}

\begin{proof}
	The Dehn twist $D_{ne}$ preserves the two north-east cylinders in the unfolded surface in \Cref{fig:fd2} (left), and each of these cylinders cover half of the area of each quarter of the unfolded surface.
The horizontal and vertical automorphism exchanging these quarter act on the horizontal and vertical component of the Frobenius function as:
$$
\begin{cases}
    (F \circ \tau_h)_h = -F_h \\
     (F \circ \tau_h)_v = F_v
\end{cases}
\quad \text{ and } \quad
\begin{cases}
    (F \circ \tau_v)_h = F_h \\
     (F \circ \tau_v)_v = -F_v
\end{cases}
.
$$
Thus its integral on the surface is zero. The argument for the north-west Dehn twist $D_v$ is the same; it preserves the two north-west cylinders in the unfolded surface in \Cref{fig:fd2} (right).
\end{proof}

\iffalse
\begin{proof}
 Let $\fund \subset \Xcover$ be an unfolded fundamental domain, as in \Cref{fig:fd} and $E \subset Q$ be its projection
 unto the billiard table $Q$ of the wind-tree model.
 This $E$ has a eight-shape as in \Cref{fig:wind}, and the projection
 is not injective: the triangle adjacent to the obtacles are covered twice,
 and when integrating over $E$, we need to take that into account.
 The (projection of the) Dehn twist $D_{ne}$ preserves $E$, and therefore
 \begin{eqnarray*}
  \int_{\fund} F \diff\Leb &=& \int_E (D_{nw}(z) - 1) \diff\Leb(z)
  = \Leb(D_{ne}(E)) - \Leb(E) = 0,
 \end{eqnarray*}
 as required. A similar proof works for the other Dehn twist $D_{nw}$, except that we need another eight-shaped projection of $\fund$ inside $Q$ that is preserved by te projection of $D_{nw}$, and $E$ rotated over $90^\circ$ has indeed this property.
\end{proof}
\fi

\subsection{Ergodic properties}

Although the main result of this paper gives rational ergodicity of $\phi_t$ (with rates), we can use the more classical method of establishing the essential value $1$ (for $\Z$-extensions, and for $\Z^d$ extensions we need a basis of $\Z^d$ as essential values) for a first return map of this flow.

A key obstruction for ergodicity is when the Frobenius fonction is a $\psi$-coboundary, \textit{i.e.} when there exists a measurable function $g: X \to \mathbb Z$ such that $F = g \circ \psi - g$.
In particular, we show in this section that the Frobenius function is not a coboundary.

To simplify the exposition, rotating the coordinate axes, we can consider the vertical flow on a surface endowed with a pseudo-Anosov automorphism contracting the vertical direction and expanding the horizontal.

Let us choose a horizontal segment $I$ in the surface $\Xcomp$ and $\widetilde I$ the union of lifts of $I$ in $\Xcover$.
If $T:I \to I$ and $\widetilde T: \widetilde I \to \widetilde I$ are the first return maps of the vertical linear flow respectively on $I$ and $\widetilde I$,
then there exists a map $f:I \to \Z^d$
such that we can express $\widetilde T$
as a skew-product
\[
	\widetilde T(x, \nb) = \left(T(x), \nb + f(x)\right)
\]
where $\widetilde I$ is identified with $I \times \Z$.
Notice that $T$ and $\widetilde T$ are ergodic if and only if the linear flow respectively on $\Xcomp$ and $\Xcover$ are ergodic.

The relevant object to study the orbits by $\widetilde T$ is then the induced cocycle for $f$ defined for $k \in \Z$ by
\[
	f_k(x) = f(x) + \dots + f(T^{k-1}x)
.\]
The ergodicity of these two maps can be linked with the following concept.
\begin{definition}\label{def:ess}
We call $e \in \bZ^d$ an essential value of (the induced cocycle by) $f$ if for every measurable set $K$ of positive measure, there exists $k \in \Z$ such that
$K \cap T^{-k}(K) \cap \{ x \in \Xcomp :  f_k(x) = e\}$ has positive measure.
\end{definition}

Note that $0 \in \Z^d$ is always an essential value.
The set of all finite essential values associated to $f$ is denoted by $\Ess_f$;
it forms a subgroup of  $\bZ^d$ and it follows from \cite{S77} that the skew product $\widetilde T$ is ergodic if and only if $\bZ^d = \Ess_f$.
Also, the map is recurrent if $0$ can be obtained as essential value using elements $k \in \Z \setminus \{0\} $ (Lebesgue measure $\Leb$ is infinite on $\Xcover$, so recurrence does not immediately follow from the invariance of $\Leb$.)\\

We represent the given translation surface as zippered rectangles, defining the linear flow as a suspension flow of an interval exchange transformation (see for an introduction \cite{Y} or \cite{Zorich} from which \Cref{zippered} is taken).

\begin{figure}[htbp]
    \centering
    \includegraphics[width=0.7\textwidth, height=0.25\textheight]{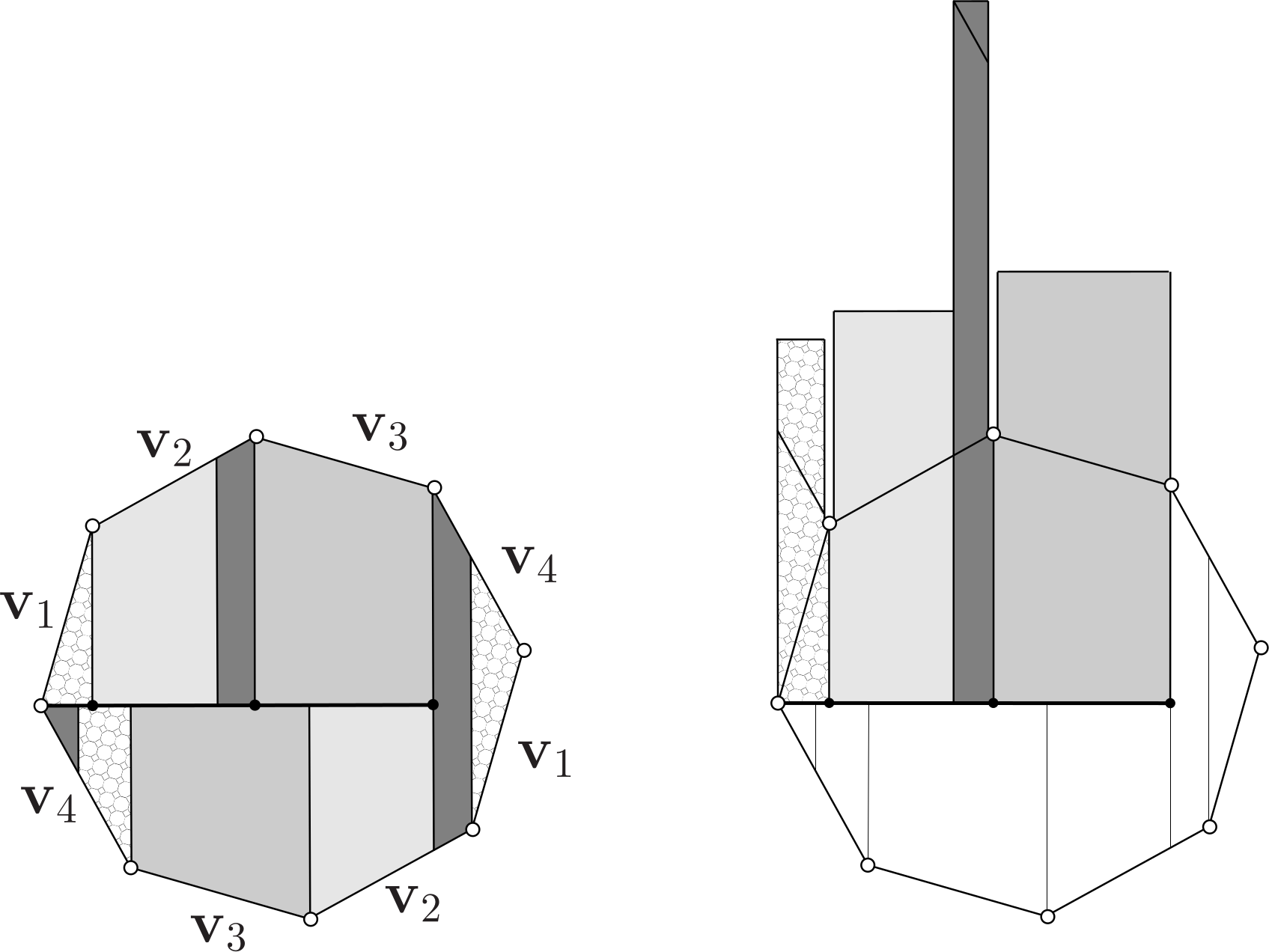}
    \caption{From a suspension flow over an interval exchange transformation to a zippered rectangles representation}
    \label{zippered}
\end{figure}

This representation defines a fundamental domain which is a union of vertical rectangles and such that the singularity of the surface are in their vertical sides.
We choose $\xi$ to be constant in the interior of this domain.
This domain has the nice feature that if there is an embedded rectangle in the surfaces $\Xcomp$ with vertices  $x, y$ at the bottom and $\phi_t(x), \phi_t(y)$ at the top then $\xi(x)-\xi(y) = \xi\left(\phi_t(x)\right) - \xi\left(\phi_t(y)\right)$.

\begin{thm}\label{thm:auto-ergo}
If $\Xcover$ is endowed with a lifted linear pseudo-Anosov automorphism $\psicover$ then the linear flow in the stable and unstable directions of the corresponding matrix is ergodic.
\end{thm}

\begin{proof}
	To any rectangle in this representation, on can associate a homology class by closing the curve going from bottom to top with a piece of the interval $I$.
	These homology classes form a basis of $H_1(X \setminus \Sigma, \Z)$ which is dual for the intersection form to the basis of $H_1(X, \Sigma, \Z)$ associated to the polygonal representation.
	Let $e \in \Z^d$ be the associated shift in $\xi$ for the flow from bottom to top of a rectangle; it is equal to the value of $\zeta$ on the homology class.
	Thus such values $e$ generate $\Z^d$ since the homology class of rectangles form a basis of $H_1(X \setminus \Sigma, \Z)$ and $\zeta$ is surjective, see \cite[Section 4.5]{Y}.
	We will now prove that $e$ is an essential value.

	It suffices to show that there exists $\delta >0$ such that for any arbitrary rectangle in $\Xcomp$, there exists a $\delta$ proportion of points in this rectangle such that $\phi_T(x)$ is in the rectangle and $\xi(\phi_T(x)) - \xi(x) = e$.\\

	When applying the pseudo-Anosov automorphism (or its inverse) dilating in the vertical direction by $\lambda^{-1}$, one gets an interval exchange on the contracted interval $I' \subset I$ by a factor $\lambda$.
	The heights of the rectangles are multiplied by $\lambda$ and the flow visits several of the initial rectangle of the zippered rectangle decomposition over $I$ before returning to $I'$.
	The rectangles over $I'$ can then be decomposed along their height by the rectangles they visit, this is usually called Rokhlin towers.
	This decomposition enables us to follow the intersection number defining $\xi$.\\

	Assume that, after we apply the automorphism $n_0$ times, the contracted interval $I^{(n_0)}$ is contained a single interval $I_0$ of $I$.
	Let $T_0$ be the height of the rectangle above $I_0$.
	Assume $x \in I^{(n_0)}$ is in the image of the rectangle whose shift is $e$, if $T_1 = \lambda^{-n_0} T_e$ is the return time of the rectangle above $x$ in $I^{(n_0)}$.
	We have $\xi(\phi_{T_1+t}(x)) - \xi(\phi_t(x)) = \xi(\phi_{T_1}(x)) - \xi(x)$ for all $0 \le t \le T_0$.

	Consider $\mathcal R$ the image of the rectangle over $I$ which associated shift is $e$ after $n+n_0$ iterations.
	Then by the previous remark, if we consider the subrectangle $\mathcal R_0$ of $\mathcal R$ with the same base but with height $\lambda^{-n} T_0$, all points $x \in \mathcal R_0$ satisfy $\xi(\phi_{T_1}(x)) - \xi(x) = e$.
	Then, for an arbitrary rectangle, if $\delta$ is half of its measure, by unique ergodicity of the flow, if we take $n$ large enough $\mathcal R_0$ intersects at least a proportion $\delta$ of the rectangle, and for all these points, we have $\xi(\phi_{T_1} \left( x \right)) - \xi(x) = e$.
	Thus $e$ is an essential value.
\end{proof}

\begin{corollary}\label{cor:noco}
	The Frobenius function $F$ is not a coboundary.
\end{corollary}

\begin{proof}
Assume by contradiction that $F = g \circ \psicomp - g$
	for some measurable function $g$. As $F$ is constant on a finite partition into cylinder sets w.r.t.\ the Markov partition, $F$ is H\"older continuous in the symbolic metric.
	By Livshits regularity (see \cite[Theorem 19.2.1]{KH}, which is stated for manifolds, but works in metric space too), $g$ is also H\"older continuous in symbolic metric, and therefore
	bounded.
%
% 	Assume by contradiction that $F$ is a coboundary.
	Take $\cR$ a rectangle of positive measure in $\Xcover$. For each $x' \in \cR$ and $n \in \Z$
	$$
	\xi(\psicover^n x') - \xi(x') = \sum_{j=0}^{n-1} F \circ \psicomp^j(x)  =
	\sum_{j=0}^{n-1} g \circ \psicomp^{j+1}(x)  - g \circ \psicomp^j(x)
	= g \circ \psicomp^n(x)-g(x)
	\leq 2\| g \|_\infty.
	$$
	The images of $\psicover$ contains larger and larger sections of the flow that remains in a compact set.
	%, moreover $\psicomp$ is minimal.
	Thus $\bigcup_{n \in  \N} \psicover^n \cR$ is a compact set of positive measure which is invariant by the flow.
	Hence the flow is not ergodic.
\end{proof}

\section{Ergodic integrals and (weak) rational ergodicity via Local Limit Laws}\label{sec:limitlaws}

We are interested in (weak) rational ergodicity (with optimal rates) of a translation flow $\phi_t$ defined on a $\Z^d$-cover $\Xcover$, $d\in\{1,2\}$, of the compact translation surface $\Xcomp$ that satisfies certain abstract assumptions. Throughout the entire paper, we let $d\in\{1,2\}$.
The main results of this section, Theorems~\ref{thm:f} and \ref{thm:rateswrd},
are a generalization of~\cite[Theorems 3.2 and 4.3]{ADDS}. The precise error rates for $\Z^d$-covers, $d=1,2$, are new. As far as we are aware, the treatment of $\Z^2$ covers is completely new.
The setting we require is as follows:

\begin{itemize}
\item[(H1)] Let $\Xcomp$ be a compact, two dimensional, surface and let $\Xcover$ be a $\Z^d$-cover with projection $p:\Xcover \to \Xcomp$ that is invariant
under deck-transformations $p \circ \deck_{\nb} = p$.
We assume that there exists a linear pseudo-Anosov automorphism $\psicover:\Xcover\to\Xcover$
on the $\Z^d$-cover $\Xcover$ that renormalizes the translation flow $\phi_t$ in the stable direction, that is $\psicover \circ \phi_t=\phi_{\lambda\,t}\circ\psicover$
for some $\lambda\in (0,1)$.\footnote{If the stable eigenvalue is negative, we take $\psicover^2$ instead.}
%Hence, the translation flow $\phi_t$ is in the stable direction of $\psicover$.

\item[(H2)] The linear pseudo-Anosov automorphism $\psicover$ commutes with the deck transformations,
%$\deck_\mb(x,\nb) = (x,\nb+\mb)$,
i.e., $\psicover \circ \deck_\nb = \deck_\nb \circ \psicover$ for all $\nb \in \Z^d$ and
$\psicomp = p \circ \psicover \circ p^{-1}:\Xcomp \to \Xcomp$ is well-defined.

\item[(H3)]
Upon the choice of a bounded fundamental domain $\fund$ (i.e.,
$\Xcover$ is the disjoint union $\bigsqcup_{\nb \in \Z^d} \deck_{\nb}(\fund)$),
we define $\xi:\Xcover \to \Z^d$ to be the $\Z^d$ component of $x \in \Xcover$,
via $\xi(x) = \nb$  if $x \in \deck_{\nb}(\fund)$.
We consider
$\psicover$ as the lift of a pseudo-Anosov automorphism $\psicomp: \Xcomp\to\Xcomp$
defined via
\[
\psicover(x,\nb)=(\psicomp(x),\nb+F(x)),\qquad x\in\Xcomp,\nb \in \Z^d,
\]
where $F(x) = \xi \circ \psicover(x')-\xi(x')$, defined independently of a choice of $x' \in p^{-1}(x)$, is called the {\em Frobenius function}.
We assume that $\int_{\Xcomp} F \diff\Leb = 0$ ({\em no drift condition}) and that $F:\Xcomp \to \Z^d$ is {\bf not} a coboundary, i.e.,
$F \neq g \circ \psicomp - g$ for any $g:\Xcomp \to \Z^d$.
\end{itemize}

\noindent
The Lebesgue measure $\Leb$ is invariant, for both the finite and the infinite measure preserving automorphism, $\psicomp$ and $\psicover$.

 \begin{rmk}
 The requirement that $\psicomp$ is (a  two dimensional)  linear automorphism can be  relaxed.  The reason this simplification is that it allows us to work with simpler anisotropic Banach spaces as described in \Cref{sec:banach} below.
 \end{rmk}

\begin{rmk}
In Sections~\ref{sec:staircases}
and~\ref{sec:wind-tree}, we give examples  where these  hypotheses apply for $d = 1,2$, respectively.
That the Frobenius function
has zero integral, but is not a coboundary was shown in \Cref{cor:noco}.
\end{rmk}

As expected (and clarified in Subsection~\ref{subsec:opllt}),
the Central Limit Theorem (CLT) for the ergodic sum $F_{\KK} := \sum_{j=0}^{\KK-1} F \circ \psi^j$ holds. That is,
\begin{align}\label{eq:clt}
\frac{F_{\KK}}{\sqrt{\KK}}\implies \chi,\text{ as } \KK\to\infty,
\end{align}
where $\implies$ stands for convergence in distribution and $\chi$ is a Gaussian random variable with mean $0$ (here we use the no drift condition in (H3) and covariance matrix
$\Sigma^2$. Here $\Sigma^2$ is a symmetric,
non-degenerate, $d\times d$ matrix
$\Sigma^2 = \sum_{j\in\Z} \int_{\Xcomp} F^T\otimes F\circ \psicomp^j\, \diff\Leb$ (with $F^T$ the transpose of $F$),
and $\Sigma$ will be its unique symmetric positive-definite square root.
For $d=1$, the matrix $\Sigma^2 = \sigma^2$ is a scalar.
The speed of mixing of $\psicomp$ ensures that the above sum converges.
The non-degeneracy of $\Sigma^2$ is ensured because $F$ is not a coboundary, see (H3).
%
%
% We are interested in a simple expression of the ergodic integral $\int_0^T G\circ \phi_t\, dt$ for $G\in C^1(\Xcomp)$. This extends to
%  expression of $\int_0^T \Gcover \circ \phi_t\, dt$
% where $\Gcover$ is compactly supported, because then $\Gcover$
% is the finite sum of $G( \cdot ,r) := G \cdot \1_{\deck_r(\fund)}$ and we can write the ergodic integral as the sum
% of finitely many integrals $\int_0^T (G \cdot \1_{\deck_r(\fund)}) \circ \phi_t\, dt$, accordingly.

We are interested in a simple expression of the ergodic integral $\int_0^T \Gcover \circ \phi_t\, \diff t$ where $\Gcover \in C^1(\Xcover)$ is compactly supported.
Hence $G( \cdot ,r) := \Gcover \cdot \1_{\deck_r(\fund)}$ is non-zero for at most
finitely many $r \in \Z^d$ (regardless of the exact choice of the fundamental domain $\fund$), and we can write the ergodic integral as the sum
of finitely many integrals $\int_0^T (\Gcover \cdot \1_{\deck_r(\fund)}) \circ \phi_t\, \diff t$, accordingly.

\subsection{Main results}\label{subsecc:mainres}

Let $\Sigma^2$ be the covariance matrix in~\eqref{eq:clt}.
For $d=1$, we write $\Sigma^2=\sigma^2$.  Recall that $\lambda \in (0,1)$ is the stable eigenvalue of $\psicover$.

\begin{thm}\label{thm:f}
Assume (H1)--(H3).
Let $\Gcover \in C^1(\Xcover)$ be a compactly supported real function.
% let $K$ be minimal such that $\psicover^{\KK}(\bigcup_{t=0}^T \phi_t(x))$ is contained in a single element of $\cP \vee \psicover^{-1}\cP$.
Choose $K \in \N$ so large that $\psicover^{\KK}$ maps the flow-line between $x' \in p^{-1}(x)$ and $\phi_T(x')$ into a single copy of the fundamental domain (this amounts to $K \approx  \log^* T := \lceil \log_{\lambda^{-1}} T \rceil$).

\begin{enumerate}
	\item[I.]
	Suppose $d = 1$. Then there exist real bounded functions $g_{k,j}$ so that for all $N\ge 1$,
	\begin{align*}
	\int_0^{T} \Gcover \circ \phi_t(x)\, \diff t &=\frac{\int_{\Xcover} \Gcover\, \diff\Leb}{\sigma \sqrt{2\pi}}
	\cdot e^{-\frac{\xi\left(\psicover^{\KK}(x)\right)^2}{ 2 \sigma^2 \KK } } \, \frac{T}{\sqrt{\KK}}  \\
	&\quad \times \left( 1 +
	\, \sum_{k=1}^N \frac{1}{\KK^k}  \sum_{j=0}^{2k}
	g_{k,j}(x)\xi(\psicover^{\KK}(x))^{2k-j} + \ O\left(\frac{1}{\KK^{N+1}} \right) \right)
	\quad \text{ as } T \to \infty.
	\end{align*}

	\item[II.]
	Suppose $d = 2$. Then there exist real bounded functions $g_1:X \to \R^2$, $g_2:X \to GL(\R,2)$ and a real constant $g'_2$, so that as $T \to \infty$,
	\begin{align*}
	  \int_0^{T} \Gcover \circ \phi_t(x)\, \diff t &=
	  \frac{\int_{\Xcover} \Gcover\, \diff\Leb}{2\pi\sqrt{\det \Sigma^2}}
e^{- \frac12 \langle \Sigma^{-1} \frac{\xi(\psicover^{\KK}(x))}{\sqrt{\KK}}, \Sigma^{-1} \frac{\xi(\psicover^{\KK}(x))}{\sqrt{\KK}} \rangle  }\frac{T}{\KK} \\
& \quad \times \left(1+\frac{\langle g_1(x), \xi(\psicover^{\KK}(x')) \rangle }{\KK}+\frac{\langle \xi(\psicover^{\KK}(x)), g_2(x) \xi(\psicover^{\KK}(x)) \rangle + g_2'}{\KK^2}
+O\left(\frac{1}{\KK^3}\right)\right).
       \end{align*}
\end{enumerate}
\end{thm}

\begin{rmk} The functions $g_{k,j}$ in the case $d=1$, and $g_1, g_2$
in the case $d=2$ are described precisely inside the proof.
For $d=2$, we can also go higher in the expansion, but since the calculations are tedious, we omit this.

The assumption that $G$ is real valued (compactly supported) can be relaxed to $G$ (compactly supported) taking values in $\C$ or even $\C^d$. We would still need to separate the (real or imaginary) components, and for $\C^d$, the
 vector valued functions $g_{k,j}$ need to be adjusted, which is a tedious exercise, even for $d=1$.
 \end{rmk}

 Using \Cref{thm:f} we obtain  \emph{expansion in weak rational ergodicity for `good' functions}.

  \begin{thm}\label{thm:rateswrd}
    Assume the setup of \Cref{thm:f}, and let $\fund \subset \Xcover$ be a fundamental domain.
    Let $\chi \simeq {\mathcal N}(0,\Sigma^2)$ be a $d=1,2$-dimensional Gaussian random variable.
  \begin{enumerate}[(i)]
  \item \label{thm:itema}
  Suppose $d=1$.
  Then, there exist real constants $d_{k,j}$ so that for all $N\ge 1$,
	\begin{align*}
	\int_{\fund} \int_0^{T} \Gcover \circ \phi_t(x)\, \diff t \, \diff\Leb &=\frac{\int_{\Xcover}\Gcover\, \diff\Leb}{\sigma\sqrt{2\pi}}\frac{T}{\sqrt{\KK}} \left( \E(e^{-\frac{\chi^2}{2}})
	+
	\, \sum_{k=1}^N \sum_{j=0}^{2k}
	\frac{d_{k,j}}{\sqrt{\KK}^{2k+j}}  + \ O\left(\frac{1}{\KK^{N+1}}\right) \right).
	\end{align*}

  \item \label{thm:itemb}
  Suppose $d=2$.
  Then, there are real constants $d_1, d_2$ so that
	\begin{align*}
	 \int_{\fund} \int_0^{T} \Gcover \circ \phi_t(x)\, \diff t \, \diff\Leb &=   \frac{\int_{\Xcover}\Gcover\, \diff\Leb}{\sqrt{\det\Sigma^2}}\frac{T}{\KK}
	 \left(\E(e^{-\frac{\chi^2}{2}}) + \frac{d_1}{\KK}+\frac{d_2}{\KK^2}
+O\left(\frac{1}{\KK^3}\right)\right).
	\end{align*}
  \end{enumerate}
 \end{thm}

 \begin{rmk}\label{rmk:genwr}
  Weak rational ergodicity without rates follows immediately since
  convergence for all $L^1(\Xcover)$-functions is an immediate consequence of \Cref{thm:f}
  and the ratio ergodic theorem, see~\cite{Aaronson}.
 \end{rmk}

 \subsection{Strategy of the proof}\label{subsec:strat}

Following the approach in~\cite{GL}, in particular~\cite[Equation (2.4) and (2.6)]{GL}, (see also~\cite[Equation (4)]{BuSi} for the same expression),
we first exploit the commutation relation \eqref{eq:renorm}, which is part of our assumption (H1).
For any $\vcover \in \mathcal{C}^0(\Xcover)$, for every $x \in \Xcover$, for all $T>0$ and all integers $K \geq 0$, we compute that
\begin{align*}
\int_0^T \vcover  \circ \phi_t(x)\, \diff t &= \int_0^T \vcover  \circ \psicover^{-K} \circ \psicover^{\KK} \circ \phi_t(x) \, \diff t = \frac{1}{\lambda^{k} }\int_0^T \vcover  \circ \psicover^{-K} \circ \phi_{\lambda^{\KK}r} \circ \psicover^{\KK} (x) \, \lambda^{\KK}\, \diff r \\
&=\frac{1}{\lambda^{\KK}}\int_0^{\lambda^{\KK}T} \vcover  \circ \psicover^{-K} \circ \phi_{r} \circ \psicover^{\KK} (x) \, \diff  r.
\end{align*}
Let $\Lcover : L^1(\Xcover)\to L^1(\Xcover)$ be the transfer operator associated with $\psicover$ defined via
 $\int_{\Xcover} \Lcover \vcover \, \wcover \, \diff\Leb=\int_{\Xcover} \vcover \, \wcover \circ \psicover \, \diff\Leb$ with $\vcover \in L^1(\Xcover)$ and $\wcover \in L^\infty(\Xcover)$.
 Since the map $\psicover$ is invertible and preserves $\Leb$,
 we also have $\Lcover \vcover  = \vcover  \circ \psicover^{-1}$.
 Thus,
 \begin{align}\label{eq:obs}
  \int_0^T \vcover  \circ \phi_t(x)\, \diff t &=\frac{1}{\lambda^{\KK}}\int_0^{\lambda^{k}T}  \Lcover^{\KK} \vcover \circ \phi_{r} \circ \psicover^{\KK} (x) \, \diff r.
 \end{align}

 The strategy is to relate the behaviour of  $\Lcover^{\KK}$ with an operator (or conditional) local limit theorem in terms of the transfer operator $\Lcomp:L^1(\Xcomp)\to L^1(\Xcomp)$ for the automorphism $\psicomp$ (defined via $\int_{\Xcomp} \Lcomp v\, w\, \diff\Leb =\int_{\Xcomp} v\, w\circ \psicomp\,  \diff\Leb$ with $v\in L^1(\Xcomp)$ and $w\in L^\infty(\Xcomp)$).
 Also we define the {\em twisted transfer operator} as
 $$
 \Lcomp_u v = \Lcomp(e^{iu F} v), \qquad u \in (-\pi,\pi]^d,
$$
where $uF$ indicates the scalar product if $u$ and $F$ are vectors.
%where we recall our abuse of notation $F(x) = F(x')$ for any $x' \in \Xcover$ such that $p(x') = x \in \Xcomp$.
The operator local limit theorem we are after is in the sense of~\cite[Section 6]{AD01}.

The first lemma below makes the relation between $\Lcover^{\KK} \vcover$, for compactly supported functions $\vcover$, and $\Lcomp_u ^{\KK} v$ precise.
Recall (from (H1)) that $p:\Xcover\to X$.

 \begin{lemma}\label{lem:rel}
 Let $v\in L^1(\Xcomp)$ and let $\vcover(\cdot)= v( \cdot , r) =v\circ p\in L^1(\Xcover)$ be the lifted version supported on $\{\xi=r\}$, $r\in\Z^d$.
 For all $\ell, r\in\Z^d$, for all $K\ge 1$ and for all $x\in X$,
 \[
 \Lcover^{\KK}  v (x,r)\, \1_{\{\xi=\ell\}} = \Lcomp^{\KK} v(x) \1_{\{F_{\KK}(x)=\ell-r\}}
= \frac{1}{(2\pi)^d} \int_{[-\pi,\pi]^d} e^{-iu(\ell-r)} \Lcomp_u^{\KK} v(x) \, \diff u.
  \]
  for ergodic sums $F_{\KK} := \sum_{j=0}^{K-1} F \circ \psicover^j$.
 \end{lemma}

 \begin{proof} Let $v\in L^1(\Xcomp)$, $w\in L^\infty(\Xcomp)$ and $v(\cdot ,r) = v\circ p$,
 $w( \cdot,\ell) = w \circ p$ be the versions supported on $\{ \xi = r\}$
 and $\{ \xi = \ell \}$, respectively.
 Compute that
\begin{align*}
\int_{\Xcover} \Lcover^{\KK}  v (x,r) \1_{\{\xi=\ell\}} w(x,\ell) \, \diff\Leb(x)
&=\int_{\Xcover} \Lcover^{\KK} (\1_{\{\Xcomp\times \{r\}\}} v)\,(\1_{\{\Xcomp\times \{\ell\}\}} w) \, \diff\Leb \\
&=\int_{\Xcover} (\1_{\{\Xcomp \times \{r\}\}} v(x,r))\,  (\1_{\{\Xcomp\times \{\ell\}\}} w(x,\ell))\circ \psicover^{\KK}(x)\, \diff\Leb \\
&= \int_{\Xcomp} v \, w \circ \psi^{\KK} \, \1_{\{F_{\KK}=\ell-r\}}\, \diff\Leb \\
 &=\int_{\Xcomp} \Lcomp^{\KK}( v \1_{\{F_{\KK}=\ell-r\}})\, w(x) \, \diff\Leb,
 \end{align*}
which gives the first equality in the statement.
 We can write the indicator function $\1_{\{F_{\KK}=\ell-r\}} =  \frac{1}{(2\pi)^d} \int_{[-\pi,\pi]^d} e^{iu (F_{\KK} - (\ell-r))} \, du$, so
\begin{eqnarray*}
\int_{\Xcomp} \Lcomp^{\KK}( v(x) \1_{\{F_{\KK}=\ell-r\}})\, w(x) \, \diff\Leb
&=&  \frac{1}{(2\pi)^d} \int_{\Xcomp}  \int_{[-\pi,\pi]^d} \Lcomp^{\KK}\left( v e^{iu \left( F_{\KK}-(\ell-r) \right)}\right)(x) \, du \ w(x) \,  \diff\Leb  \\
&=&  \frac{1}{(2\pi)^d} \int_{[-\pi,\pi]^d} \int_{\Xcomp} e^{-iu(\ell-r)} \Lcomp_u^{\KK}v \, w \, \diff\Leb \, \diff u.
\end{eqnarray*}
 ~\end{proof}

 Given Lemma~\ref{lem:rel}, our task comes down to obtain a precise expansion of $\Lcomp^{\KK} v(x) \1_{\{F_{\KK}(x)=\ell-r\}}$ in powers of $K$ and combined it with~\eqref{eq:obs}. As explained in subsection~\ref{subsec:opllt} below, $\ell-r$ in $\Lcomp^{\KK} v(x) \1_{\{F_{\KK}(x)=\ell-r\}}$ will be replaced by $\xi\left(\psicover^{\KK}(x')\right)$, which in the end, will give the form of Theorem~\ref{thm:f}.

 \subsection{An (operator) local limit theorem (LLT) for $F_{\KK}$}\label{subsec:opllt}% along with some consequences

 \Cref{prop:opllt} below is an asymptotic expansion operator LLT (in the sense of~\cite[Section 6]{AD01}) for the ergodic sums $F_{\KK}$.
 The expansion in \Cref{prop:opllt} is a key ingredient in the proof of our main results Theorems~\ref{thm:f} and \ref{thm:rateswrd}.
 We recall that \Cref{thm:f} is a version of \Cref{thm:main1}, including a precise statement for $d=2$, while Theorem~\ref{thm:rateswrd} gives optimal rates in a form of weak rational ergodicity for $C^1$ functions.

 We first recall some facts on the spectral properties of $\Lcomp$ and its twisted version $\Lcomp_uf =\Lcomp(e^{iuF} f)$, $u\in\R^d$.
 Since $\psicomp:\Xcomp\to\Xcomp$ is an invertible map, we need adequate, anisotropic Banach spaces on which the corresponding transfer operator $\Lcomp$ can act.  The details on Banach spaces we shall use are deferred to \Cref{sec:banach}. The first proposition summarizes
 all that we need to use in terms of Banach spaces (regardless the particular of these spaces) to prove the main results.

\begin{prop}\label{prop:op}
  \begin{itemize}
  \item[(a)]There exist anisotropic Banach spaces $\cB,\cB_w$
  so that $C^1(\Xcomp)\subset\cB\subset \cB_w\subset C^1(\Xcomp)^*$
  where $C^1(\Xcomp)^*$ is the (topological) dual of $C^1(X)$.
  The transfer operator $\Lcomp$ acts continuously  on $\cB$ and $\cB_w$.
  Moreover, $\Lcomp$ is quasicompact\footnote{the precise terminology is recalled and specified in \Cref{sec:banach}}
 when viewed as operator from $\cB$ to  $\cB$. In particular, $1$ is an isolated, simple eigenvalue in the spectrum of $\Lcomp$.

  \item[(b)]
  The derivatives $\frac{d^k}{du^k}\Lcomp_u f$
  are linear operators on $\cB$  with operator norm of $O(\|F\|_\infty^k)$.

  \item[(c)] There exist $\delta>0$ and a family of simple eigenvalues $\lambda_u$ that is analytic in $u$ for
  all $|u|<\delta$.
  Also, for all $|u|<\delta$ and $n\ge 1$,
\[ \Lcomp_u^n=\lambda_u^n\Pi_u+Q_u^n, \]
  where $\Pi_u$ is the family of spectral projections associated with $\lambda_u$ with $\Pi_0 v= \int_{\Xcomp} v\, \diff\Leb$,
  $\Pi_u, Q_u$ are analytic when regarded as (family of)  operators acting on $\cB$,
  $\Pi_uQ_u=Q_u\Pi_u$ and $\|Q_u^n\|_{\cB}\le \delta_0^n$ for some $\delta_0<1$.

  \item[(d)] There exists $\delta_1\in (0,1)$ so that $\|\Lcomp_u^n\|_{\cB}\le \delta_1^n$ for all $n\ge 1$.
  \end{itemize}
\end{prop}

The proof of Proposition~\ref{prop:op}  is provided in Section~\ref{sec:banach} (the headers of the subsections indicate which item of the proposition is proved). Proposition~\ref{prop:op} is known in various settings similar to the one here (see, for instance, the survey paper~\cite{Demers18} and references therein).

We recall that throughout, $d\in\{1,2\}$.
Throughout this section we let $\Pi_0^{(j)}, \lambda_0^{(j)}$ denote the $j$-th derivative in $u$
of $\Pi_u,\lambda_u$ evaluated at $u=0$. From here onward,
given $u\in\R^d$
we write $u^{\otimes j}:=u\otimes \cdot \otimes u$ for the $j$-fold tensor product of $u$ with itself. We define the $*$-product $u * v$ on column vectors $u \in \C^d$ and $v \in \C^{d'}$,
where we assume that $d'$ is a multiple of $d$, or vice versa.
The meaning of these type of products is clarified in Appendix~\ref{sec:tesnot}.
\begin{rmk}\label{rmk:samed}
 If $d=d'$ then $u * v= uv = \sum_{i=1}^du_iv_i$ is the usual scalar product.
\end{rmk}

By Proposition~\ref{prop:op}(c), for $|u|<\delta$ and for any $v\in\cB$,
\begin{align}\label{eq:lll}
  \Lcomp_u^n v=\lambda_u^n\Pi_u v+Q_u^n v=\left(\sum_{m=0}^\infty\lambda_0^{(m)} * u^{\otimes m}\right)\times \left(\sum_{j=0}^\infty \Pi_0^{(j)} v * u^{\otimes m}\right)+Q_u^n v.
\end{align}
\begin{rmk}\label{rmk:meaning}
 Throughout we restrict to $v$ taking real values. In this case we note that
when $d=1$  $\Pi_0^{(j)}v, \lambda_0^{(j)}$ are scalars, and when $d=2$,
are column vectors with $2^j$ entries. A similar statement holds for $\lambda_0^{(j)}, j\ge 0$. This is in the sense of the terminology clarified in Appendix~\ref{sec:tesnot}.
Clearly, when $j=0$, $\Pi_0^{(0)}v=\Pi_0 v=\int_{\Xcomp} v\, \diff\Leb$, $\lambda_0^{(0)}=\lambda_0=1$ are scalars.
\end{rmk}

Recall $u\in \R^d$, $d\in\{1,2\}$. A classical (not necessarily short) argument which dates back to~\cite{Nag, GH} (see also~\cite{AD01, Gousurv}), shows that provided that
$\lambda_u$ is twice differentiable at $0$ (so much weaker than analyticity of $\lambda$ ensured by Proposition~\ref{prop:op}(c)),
then
\begin{align}\label{eq:lambda-1}
1- \lambda_u =\frac 12\Sigma^2 * u^{\otimes 2}(1+o(1))
 =\frac 12\langle\Sigma u,\Sigma u\rangle(1+o(1)) \text{
 or equivalently }
 \lambda_u= e^{-\frac 12\Sigma^2 * u^{\otimes 2} (1+o(1)) }
\end{align}
where $\langle\cdot,\cdot\rangle$ is the usual scalar product and $\Sigma$ is the unique positive definite symmetric
square root of the
non-degenerate $d\times d$ covariance  matrix introduced in~\eqref{eq:clt}.

An immediate
consequence of~\eqref{eq:lambda-1} and~\eqref{eq:lll} is that $\E\left(e^{iu  F_{\KK}}\right)=e^{-\frac 12\Sigma^2 * u^{\otimes 2}(1+o(1))}$, as $u\to 0$. A classical argument  based on the Levy continuity theorem (see, for instance, the survey~\cite{Gousurv}) shows
that CLT stated in~\eqref{eq:clt} holds.

In the setup of the current section, a refined version of the CLT~\eqref{eq:clt} holds, namely a Local Limit Theorem (LLT).
This means that for $M\in\Z^d$,
\begin{align}\label{eq:llt1}
 \Leb\left(F_{\KK}(x)=M\right)
 &=\frac{1}{\left(2\pi\sqrt{\KK}\right)^d}
 \Phi\left(\frac{M}{\sqrt{\KK}}\right) (1+o(1)),\text{ as } K\to\infty,
\end{align}
 where $\Phi$ is the density of the Gaussian random variable $\chi$ in~\eqref{eq:clt}.

In the sequel we shall exploit, and prove, a stronger version of~\eqref{eq:llt1},
namely an operator LLT with precise expansion, as in Proposition~\ref{prop:opllt} below. This type of expansion for LLT is, essentially, contained inside~\cite[Proof of Theorem 3.2]{Pene18}, where different Banach spaces are used.

Before the statement, recall Remark~\ref{rmk:meaning} on the meaning of
 $\Pi_0^{(j)}v$ and $\lambda_0^{(j)}$.
 With the conventions on tensors, see~\Cref{sec:tesnot}
 and specifically \eqref{eq:Sigmau},
 we have
 $A_u := \frac{1}{2}\Sigma^2 * u^{\otimes 2}=\frac12 \langle \Sigma u,\Sigma u\rangle$.
 Recalling~\eqref{eq:lambda-1},
 and using the analyticity of $\lambda_u$,
 we can write
 \begin{align}\label{eq:li}
  \lambda_u^n= e^{-\frac{n}{2}\langle\Sigma u, \Sigma u\rangle}\left(1+\sum_{m=1}^\infty\frac{1}{m!}\left(\frac{\lambda^n}{A^n}\right)_0^{(m)}* u^{\otimes m}\right),
\end{align}
where $\left(\frac{\lambda^n}{A^n}\right)_0^{(m)}$ is the $m$-th derivative of
$\frac{\lambda_u^n}{A_u^n}$ evaluated at $0$, which is a column vector with $d^m$ entries, $d=1,2$.

Recall that $F$ takes values in $\Z^d$, $d=1,2$.
 \begin{prop}\label{prop:opllt}
Let $v\in C^1(X)$.
 Then

 \begin{enumerate}[(a)]
 \item \label{prop:itema}
	 If $v$ is a \textbf{real} function then
 $\Pi_0^{(j)}v$ is a column vector with $d^j$ entries which are real if $j$ is even and purely imaginary entries if $j$ is odd.

 Moreover, $\left(\frac{\lambda}{A}\right)_0^{(m)}$ is a column vector with $d^m$ real entries if $m$ is even and $\left(\frac{\lambda}{A}\right)_0^{(m)}$ is is a column vector with $d^m$  purely imaginary entries if $m$ is odd.

 \item \label{prop:itemb}
	  Let $\delta$, $\delta_0$ and $\delta_1$ be as in \Cref{prop:op}(c). Set $\delta_2=\max\{\delta_0,\delta_1\}$.
  Let $\Sigma^2$ be the covariance matrix in~\eqref{eq:clt}.
  Then for all $x\in X$ and for all $\ell, r\in\Z^d$,
\begin{align*}
& \Lcomp^{\KK} v(x) \1_{\{F_{\KK}(x)=\ell-r\}}+E_K v(x)\\
&=\frac{1}{\left(2\pi\sqrt{\KK}\right)^d} \int_{\left[-\delta\sqrt{\KK}, \delta\sqrt{\KK}\right]^d} e^{-iu\frac{\ell-r}{\sqrt{\KK}}}e^{-\frac{\langle \Sigma u, \Sigma u\rangle}{2}}\left(1+\sum_{m=1}^\infty\frac{1}{m!}\left(\frac{\lambda^{\KK}}{A^{\KK}}\right)_0^{(m)}* \frac{u^{\otimes m}}{\KK^{m/2}}\right)\\
&\quad\quad\quad\quad\quad\quad\quad\times \left(\int_{\Xcomp} v\, \diff\Leb+\sum_{j=1}^\infty \frac{1}{j!}\Pi_0^{(j)}v(x) * \frac{u^{\otimes j}} {\KK^{j/2}}\right)\, du,
 \end{align*}
 where $E_K$ is an operator acting on $\cB$ so that $ \|E_K v\|_{\cB}\le C\delta_2^{\KK} \|v\|_{C^1}$
 and so that $\left|\int_{\Xcomp} E_K v\, \diff\Leb \right| \le C' \delta_2^{\KK} \|v\|_{C^1}$ for some $C, C'>0$.
\end{enumerate}
 \end{prop}

 The proof of \Cref{prop:opllt}(a) is deferred to \Cref{sec:remainingproofs}. Here we provide the argument for \Cref{prop:opllt}(b).

 \begin{proof}[Proof of \Cref{prop:opllt}(b)] Recall $d=1,2$.
 By \Cref{lem:rel} (second equality there) and \Cref{prop:op}(c) and (d),
\begin{align}\label{eq:bbbb}
\Lcomp^{\KK} v(x) \1_{\{F_{\KK}(x)=\ell-r\} }
&=  \frac{1}{(2\pi)^d}\int_{[-\pi,\pi]^d} e^{-iu(\ell-r)} \Lcomp_u^{\KK} v(x) \, du \nonumber \\
\nonumber
&=  \frac{1}{(2\pi)^d} \int_{[-\delta,\delta]^d} e^{-iu(\ell-r)} (\lambda^{\KK}_u \Pi_u + Q^{\KK}_u) v(x) \, du +O(\delta_1^{\KK})\\
&=  \frac{1}{(2\pi)^d} \int_{[-\delta,\delta]^d} e^{-iu(\ell-r)} \lambda^{\KK}_u \Pi_u v(x) \, du
+O(\delta_2^{\KK}).
\end{align}
 By equation~\eqref{eq:li}, $\lambda_u^{\KK}= e^{-\frac{\KK}{2}\langle\Sigma u, \Sigma u\rangle}\left(1+\sum_{j=1}^\infty\frac{1}{j!}\left(\frac{\lambda^{\KK}}{A^{\KK}}\right)_0^{(j)}* u^{\otimes j}\right)$.
 We already know that
 $\Pi_u v=\int_{\Xcomp} v\, \diff\Leb+\sum_{j=1}^\infty \Pi_0^{(j)} v * u^{\otimes j}$.
 Putting these two expressions together and using a change of coordinates $u\to \frac{u}{\sqrt{\KK}^d}$ in~\eqref{eq:bbbb} gives the conclusion.
 \end{proof}

 To clarify that integral in \Cref{prop:opllt}(b) leads to a real scalar (when $v$ takes real values), we rewrite it in a more transparent way and
 record this as a lemma.

 \begin{lemma}\label{lem:consopllt} Assume the setup  of \Cref{prop:opllt}.
  Let $\Phi$ is the density of the Gaussian random variable $\chi$ in~\eqref{eq:clt}.  Let $d=1,2$ and set $I_j(\Sigma,L)=\int_{\R^d} e^{-iu L}e^{-\frac{\langle \Sigma u, \Sigma u\rangle}{2}} u^{\otimes j}\, du$ for $j \geq 1$ and $L \in \R^d$.

  Then for any $n\ge 1$,
  \begin{align*}
    \Lcomp^{\KK} v(x) &\1_{\{F_{\KK}(x)=\ell-r\}}=\frac{\int_{\Xcomp} v\, \diff\Leb}{\left(2\pi\sqrt{\KK}\right)^d}
 \Phi\left(\frac{M}{\sqrt{\KK}}\right) +\sum_{j=1}^{N} \frac{1}{j!} \frac{C_j(v)} {\KK^{(j+d)/2}}+E_{K,N} v(x,r)
  \end{align*}
for real bounded functions
  $C_1(v)=I_1\left(\Sigma,\frac{\ell-r}{\sqrt{\KK}}\right)*\Pi_0^{(j)}v(x)
   +I_1\left(\Sigma, \frac{\ell-r}{\sqrt{\KK}}\right)*\left(\frac{\lambda}{A}\right)_0^{(j)}$,
  \begin{align*}
  C_j(v)= & I_j\left(\Sigma,\frac{\ell-r}{\sqrt{\KK}}\right)*\Pi_0^{(j)}v(x)
   +I_j\left(\Sigma, \frac{\ell-r}{\sqrt{\KK}}\right)*\left(\frac{\lambda}{A}\right)_0^{(j)}\\
   &+I_j\left(\Sigma,\frac{\ell-r}{\sqrt{\KK}}\right)*\left(
  \sum_{r_1+r_2=j}\frac{1}{r_1! r_2!}\Pi_0^{(r_1)}v(x)\otimes \left(\frac{\lambda}{A}\right)_0^{(r_2)}\right) \quad  \text{ for } 2 \le j\le N,
  \end{align*}
 and $E_{K,N}$ is an operator acting on $\cB$ so that $\|E_{K, N} v\|_{\cB} = o\left(\KK^{-(N+d)/2}  \right) \| v \|_{C^1}$
 and $\left|\int_{\Xcomp} E_{K,N}v\diff\Leb \right| = o\left(\KK^{-(N+d)/2}  \right)$.
\end{lemma}

\begin{proof} Truncating each sum inside the integral in \Cref{prop:opllt}(b)  at $N\ge 1$ and using the information on the operator $E_{K,N}$, we obtain

\begin{align*}
\Lcomp^{\KK} v(x) \1_{\{F_{\KK}(x)=\ell-r\}}
&=\frac{1}{\left(2\pi\sqrt{\KK}\right)^d} \int_{\left[-\delta\sqrt{\KK}, \delta\sqrt{\KK}\right]^d} e^{-iu\frac{\ell-r}{\sqrt{\KK}}}e^{-\frac{\langle \Sigma u, \Sigma u\rangle}{2}}\left(1+\sum_{m=1}^N\frac{1}{m!}\left(\frac{\lambda^{\KK}}{A^{\KK}}\right)_0^{(m)}* \frac{u^{\otimes m}}{\KK^{m/2}}\right)\\
&\quad\quad\quad\quad\quad\quad\quad\times \left(\int_{\Xcomp} v\, \diff\Leb+\sum_{j=1}^N \frac{1}{j!}\Pi_0^{(j)}v(x) * \frac{u^{\otimes j}} {\KK^{j/2}}\right)\, du +E_{K,N} v(x),
 \end{align*}
where $E_{K,N}$  is an operator as in the statement of the corollary.

The density of the Gaussian can be written as $\Phi(L)=\frac{1}{(2\pi)^d}\int_{\R^d} e^{-iu L}e^{-\frac{\langle \Sigma u, \Sigma u\rangle}{2}}\, du$, for a vector $L \in \R^d$,  due to the Fourier inversion formula (i.e., the inverse Fourier transform of the characteristic function).
Note that
\[
\Phi^{(j)}(L)=(-i)^{j}\int_{\R^d} e^{-iu L}e^{-\frac{\langle \Sigma u, \Sigma u\rangle}{2}} u^{\otimes j}\, du.
\]
So, $i^j \Phi^{(j)}(L)=I_j(\Sigma, L)$.

Applying the formula for $\Phi^{(j)}$ inside the integral above with $L=\frac{\ell-r}{\sqrt{\KK}}\in\R^d$, $d=1,2$, we obtain
   \begin{align*}
\Lcomp^{\KK} & v(x) \1_{\{F_{\KK}(x)=\ell-r\}}
=\frac{\int_{\Xcomp} v\, \diff\Leb}{\left(2\pi\sqrt{\KK}\right)^d}
 \Phi\left(\frac{\ell-r}{\sqrt{\KK}}\right) +\frac{\int_{\Xcomp} v\, \diff\Leb}{\left(2\pi\sqrt{\KK}\right)^d}\sum_{m=1}^N\frac{1}{m!}\frac{i^m}{\KK^{m/2}} \Phi^{(m)}\left(\frac{\ell-r}{\sqrt{\KK}}\right)* \left(\frac{\lambda^{\KK}}{A^{\KK}}\right)_0^{(m)}\\
& +\frac{1}{\left(2\pi\sqrt{\KK}\right)^d}\sum_{j=1}^N \frac{1}{j!}\frac{i^j} {\KK^{j/2}}\Phi^{(m)}\left(\frac{\ell-r}{\sqrt{\KK}}\right)*\Pi_0^{(j)}v(x)\\
&+\frac{1}{\left(2\pi\sqrt{\KK}\right)^d}\sum_{m=1}^N\sum_{j=1}^N \frac{1}{m!j!}\frac{i^{m+j}} {\KK^{(m+j)/2}}\Phi^{(m+j)}\left(\frac{\ell-r}{\sqrt{\KK}}\right)*\left(\Pi_0^{(j)}v(x)\otimes \left(\frac{\lambda^{\KK}}{A^{\KK}}\right)_0^{(m)}\right)
+E_{K,N} v(x).
 \end{align*}
 Note that
 \begin{align*}
  \sum_{m=1}^N&\sum_{j=1}^N \frac{1}{m!j!}\frac{i^{m+j}} {\KK^{(m+j)/2}}\Phi^{(m+j)}\left(\frac{\ell-r}{\sqrt{\KK}}\right)*\Pi_0^{(j)}v(x)\otimes \left(\frac{\lambda^{\KK}}{A^{\KK}}\right)_0^{(m)}\\
  &=\sum_{j=2}^{2N} \frac{1}{j!}\frac{i^{j}} {\KK^{j/2}}\Phi^{(j)}\left(\frac{\ell-r}{\sqrt{\KK}}\right)*\left(
  \sum_{r_1+r_2=j}\frac{1}{r_1! r_2!}\Pi_0^{(r_1)}v(x)\otimes \left(\frac{\lambda^{\KK}}{A^{\KK}}\right)_0^{(r_2)}\right)\\
  &=\sum_{j=2}^{N} \frac{1}{j!}\frac{i^{j}} {\KK^{j/2}}\Phi^{(j)}\left(\frac{\ell-r}{\sqrt{\KK}}\right)*\left(
  \sum_{r_1+r_2=j}\frac{1}{r_1! r_2!}\Pi_0^{(r_1)}v(x)\otimes \left(\frac{\lambda^{\KK}}{A^{\KK}}\right)_0^{(r_2)}\right)+O\left(\frac{1}{\KK^{(N+1)/2}}\right).
 \end{align*}
 Recall that $i^j \Phi^{(j)}\left(\left(\frac{\ell-r}{\sqrt{\KK}}\right)\right)=I_j\left(\Sigma, \frac{\ell-r}{\sqrt{\KK}}\right)$.
Writing the above three sums in a single sum gives the expression of $C_j(v)$, as in the statement of the lemma (after multiplication with $(2\pi\sqrt{\KK})^{-d}$).

It remains to justify that $C_j(v)$ are real bounded functions.

\textbf{$C_j(v)$ are real or complex functions.} Recall $d=1,2$.
Recall from Remark~\ref{rmk:samed}
that when the dimension of (tensor) vectors are the same, the operation $*$
gives a scalar. By \Cref{lem:integrals} with $L=\frac{\ell-r}{\sqrt{\KK}}\in\R^d$, $d=1,2$ (in \Cref{sec:integrals}),  the integrals $I_j$ are column vectors with $d^j$ entries. We already know from Remark~\ref{rmk:meaning} that $\Pi_0^{(j)}$ and $\left(\frac{\lambda^{\KK}}{A^{\KK}}\right)_0^{(j)}$ are column vectors with $d^j$ entries. So the operation $*$ between $I_j$
and $\Pi_0^{(j)}$, and $\left(\frac{\lambda^{\KK}}{A^{\KK}}\right)_0^{(j)}$ gives a bounded real or complex function. The same applies to $I_j\left(\frac{\ell-r}{\sqrt{\KK}}\right)*\left(\Pi_0^{(r_1)}v(x)\otimes \left(\frac{\lambda^{\KK}}{A^{\KK}}\right)_0^{(r_2)}\right)$, since $r_1+r_2=j$.

\textbf{$C_j(v)$ are real bounded functions.}
By \Cref{prop:opllt}(a), $\Pi_0^{(j)}$ is a column vector with $d^j$ real entries if $j$ is even and with purely imaginary if $j$ is odd. The same applies to $\left(\frac{\lambda^{\KK}}{A^{\KK}}\right)_0^{(j)}$.
Finally, by \Cref{lem:integrals} in \Cref{sec:integrals},
$I_j$ have real entries if $j$ is even and purely imaginary entries if $j$ is odd. The boundedness of these functions follows from the analyticity of $\Pi_t$ and $\lambda_t$, ensured by Proposition~\ref{prop:op} (b), (c).
\end{proof}

The asymptotic expansion in the usual LLT follows immediately.
 That is, taking $\ell-r=M\in\Z^d$, $v\equiv 1$ in \Cref{lem:consopllt},
 and integrating over the space, we obtain that for all $N \ge 1$,
 \begin{align}\label{eq:clasLLT}
 \Leb\left(F_{\KK}(x)=M\right)
 &=\frac{1}{\left(2\pi\sqrt{\KK}\right)^d}
 \Phi\left(\frac{M}{\sqrt{\KK}}\right) +\sum_{j=1}^{N} \frac{C_j} {\KK^{(j+d)/2}}+o\left( \frac{C_{N+1}}{\KK^{(N+d)/2} } \right), \quad \text{ as } K\to\infty,
 \end{align}
 where $C_j$ are real constants.

The expansion in~\eqref{eq:clasLLT} allows us to record a technical lemma that will play an important role in the proof of \Cref{thm:rateswrd} below. Before the statement we recall that $F_{\KK}(x)=\xi \circ \psicover^{\KK}(x')-\xi(x')$
 for $x' \in p^{-1}(x)$.

 \begin{lemma}\label{lemma:crtech}

\begin{enumerate}[(i)]
\item \label{lem:crtech_i}
Let $q \ge 0$ be an integer and $f_q:\Xcomp \times \R^d \to \R$,
$f_q(x,w)=e^{-\frac12\langle w,w\rangle} g(x) * w^{\otimes q}$
for a bounded function $g:\Xcomp \to \R^{d^q}$.
Then there exist real constants $C_{d,q}$ and $d_{j,q}$, so that for any $N \ge 1$,
\begin{align*}
\int_{\Xcover} f_q\left(p(x'), \frac{\xi\left(\psicover^{\KK}(x')\right)}{\sqrt{K}}\right)\, \diff\Leb
=C_{d,q} +\sum_{j=1}^N \frac{d_{j,q}}{\KK^{j/2}} +o(\KK^{-N/2}),\qquad \text{ as } K\to\infty.
\end{align*}

\item \label{lem:crtech_ii}
Let $f:\R^d\to\R$, $f(w)= e^{-\frac12\langle w,w\rangle}$
and let $\chi$ be the $d$-dimensional Gaussian introduced in~\eqref{eq:clt}.
Then, there exist real constants $d_j$, so that for any $N \ge 1$,
\begin{align*}
\int_{\Xcover} f\left(\frac{\xi\left(\psicover^{\KK}(x')\right)}{\sqrt{\KK}}\right)\, \diff\Leb
=\E(f(\chi))+\sum_{j=1}^N \frac{d_j}{\KK^{j / 2}} +o(\KK^{-N / 2}),\quad \text{ as } K\to\infty,
\end{align*}
 \end{enumerate}
 \end{lemma}

 \begin{proof} \textbf{\Cref{lem:crtech_i}}.
 Since $F_{\KK}(x)=\xi \circ \psicover^{\KK}(x')-\xi(x')$
 for $x' \in p^{-1}(x)$,
 \begin{align*}
 \int_{\Xcover} f_q\left(p(x'), \frac{\xi\left(\psicover^{\KK}(x')\right)}{  \sqrt{\KK}}\right)\, \diff\Leb
 &=\sum_{M\in\Z^d} \int_{\{F_{\KK}\circ p =M \}}   f_q\left(p(x'), \frac{\xi\left(\psicover^{\KK}(x')\right)}{ \sqrt{\KK}}\right)\, \diff\Leb \\
 &= \sum_{M\in\Z^d} a_M e^{-\frac{1}{2\KK} \langle M,M \rangle} \ \Leb\left(F_{\KK} = M\right),
 \end{align*}
 where $a_M$ is chosen by the intermediate value theorem for integrals.
 Due to the boundedness of $g$ and exponential factor,
 $a_M = O(M^q)$, and hence finite  for each $M$, even though
 $\Leb(\{ x' \in \Xcover : F_{\KK} \circ p = M\}) = \infty$.

 This together with~\eqref{eq:clasLLT} gives
\begin{align}\label{eq:use1}
 \int_{\Xcover} f_q\left(p(x'), \frac{\xi\left(\psicover^{\KK}(x')\right)}{\sqrt{\KK}}\right)\, \diff\Leb
 =   \frac{1}{\KK^{d/2}}
\sum_{M\in\Z^d} a_M e^{-\frac{1}{2\KK} \langle M,M \rangle} \  H\left(\frac{M}{\sqrt{\KK}}\right)
\quad \text{ as } K \to \infty,
 \end{align}
 for
 $$
 H\left(\frac{M}{\sqrt{\KK}}\right) =
  (2\pi)^{-d} \Phi\left(\frac{M}{\sqrt{\KK}}\right)+\sum_{j=1}^{n}
 \frac{C_j}{\KK^{j/2}} + o\left(\frac{C_{n+1}}{\KK^{n/2}} \right),
 \qquad \text{ for } C_j \in \R.
 $$
Next, for each $M \in \Z^d$, define functions  $a_M:Q(M) \to \R$ on
the unit cube $Q(M)$ centered at $M \in \Z^2$ in such a way that $\int_{Q(M)}
a_M(w)e^{-\frac{1}{2\KK} \langle w, w \rangle} \  H\left(\frac{w}{\sqrt{\KK}}\right) dw = a_M e^{-\frac{1}{2\KK} \langle M,M \rangle} \  H\left(\frac{M}{\sqrt{\KK}}\right)$, and set $a = \sum_{M \in \Z^d} \tilde a_M \cdot \1_{Q(M)}$.
Then
 \begin{eqnarray*}
 \sum_{M\in\Z^d} a_M e^{-\frac{1}{2\KK} \langle M,M \rangle} \  H\left(\frac{M}{\sqrt{\KK}}\right)
 &=& \int_{\R^d} a(w) e^{-\frac{1}{2\KK} \langle w,w \rangle} H\left(\frac{w}{\sqrt{\KK}}\right)\, dw \\
 &=& K^{d/2} \int_{\R^d} a(v\sqrt{\KK}) e^{-\frac{1}{2} \langle v,v \rangle} H(v)\, dv,
 \end{eqnarray*}
where we used the change of coordinates $v = w/\sqrt{\KK}$. Since
 $\int_{\R^d} a(w\sqrt{\KK}) e^{-\frac{1}{2} \langle v,v \rangle} H(v)\, dv < \infty$ due to the exponential factor, the sum scales as
 $\KK^{d/2}$. Insert this estimate into~\eqref{eq:use1} to find constants $C_{q,d}$ and $d_{j,d} \in \R$ such that~\cref{lem:crtech_i} holds.

\textbf{\Cref{lem:crtech_ii}}
We just need to argue that the first term, that is $C_{d,0}$ in \cref{lem:crtech_i}, is exactly $\E( f_0(\chi))$. Apart from this constant the statement is as in \cref{lem:crtech_i} for $f_q$  with $q=0$ and $g\equiv 1$.
One could proceed via an exact calculation (using, for instance, the Euler-Maclaurin formula),
but a quicker way is to recall~\eqref{eq:clasLLT} and note that by the Portmanteau Theorem,
\begin{align*}
 \int_{\Xcover} f\left(\frac{\xi\left(\psicover^{\KK}(x')\right)}{\sqrt{\KK}}\right)\, \diff\Leb \to \E(f(\chi)),\quad \text{ as } K\to\infty.
\end{align*}
~\end{proof}

We record an immediate consequence of \Cref{lem:consopllt} that will be instrumental in the proof of \Cref{thm:f}.  Recall that $F_{\KK}(x)=\xi \circ \psicover^{\KK}(x')-\xi(x')\in\Z^d$, $d=1,2$,
 for $x' \in p^{-1}(x)$.

 \begin{corollary}\label{cor:opllt}
 Adopt the notation of \Cref{lem:consopllt}.
% Set $I_j(\Sigma,L)=\int_{\R^d} e^{-iu L}e^{-\frac{\langle \Sigma u, \Sigma u\rangle}{2}} u^{\otimes j}\, du$ for $j\ge 0$ and $L \in \R^d$.
Let $G\in C^1(X)$ and $x' \in p^{-1}(x)$.
 Then
\[
 \Lcomp^{\KK} G(x) \1_{\{F_{\KK}(x)=\xi(\psicover^{\KK}(x'))\}}
 =\frac{\int_{\Xcomp} G\, \diff\Leb}{\left(2\pi\sqrt{\KK}\right)^d}
 I_0\left(\frac{\xi(\psicover^{\KK}(x'))}{\sqrt{\KK}}\right) +\sum_{j=1}^{N} \frac{1}{j!}\frac{C_j(G, \xi(\psicover^{\KK}(x'))} {\KK^{(j+d)/2}}+E_{K,N} G(x),
 \]
  for real bounded functions
  $C_1(v)=I_1\left(\Sigma,\frac{\ell-r}{\sqrt{\KK}}\right)*\Pi_0^{(j)}G
   +I_1\left(\Sigma, \frac{\ell-r}{\sqrt{\KK}}\right)*\left(\frac{\lambda}{A}\right)_0^{(j)}$,
  \begin{align*}
   C_j(v)=&I_j\left(\Sigma,\frac{\ell-r}{\sqrt{\KK}}\right)*\Pi_0^{(j)}G
   +I_j\left(\Sigma, \frac{\ell-r}{\sqrt{\KK}}\right)*\left(\frac{\lambda^{\KK}}{A^{\KK}}\right)_0^{(j)}\\
   &+I_j\left(\Sigma,\frac{\ell-r}{\sqrt{\KK}}\right)*\left(
  \sum_{r_1+r_2=j}\frac{1}{r_1! r_2!}\Pi_0^{(r_1)}v(x,r)\otimes \left(\frac{\lambda^{\KK}}{A^{\KK}}\right)_0^{(r_2)}\right) \quad \text{ for } 2 \le j \le N,
  \end{align*}
and
the operator norm $\|E_{K, N} G\|_{\cB}= o\left(\KK^{-(N+d)/2} \|G\|_{C^1}\right)$
 and $\left|\int_{\Xcomp} E_{K,N}G\diff\Leb \right| = o\left(\KK^{-(N+d)/2}  \right)$.
 \end{corollary}

    \begin{proof} We want to apply \Cref{lem:consopllt} with $v=G$
    and suitable choice of $\ell,r\in\Z^d$.

    Take $\ell=0$ and $x' \in p^{-1}(x)$ so that $r=-\xi(\psicover^{\KK}(x'))$. To justify this choice, just recall that
$F_{\KK}(x)=\xi \circ \psicover^{\KK}(x')-\xi(x')$. The conclusion follows from \Cref{lem:consopllt} with $G$ instead of $v$ and $\ell-r=\xi(\psicover^{\KK}(x'))$.\end{proof}

\subsection{Proof of the main results}\label{subsec:prm}

Before proceeding to the proof of \Cref{thm:f} we record one more technical lemma. Recall the notation of equation~\eqref{eq:obs}.
The following lemma will be proved in \Cref{sec:EKN}.
\begin{lemma}\label{lemma:EKN}
 Consider the operator $E_{K, N}$ defined in Corollary~\ref{cor:opllt}.
 Then
 \[
  \left|\frac{1}{\lambda^{\KK}}\int_0^{\lambda^{\KK}T} E_{K,N} G\circ \phi_{r} (x_K) \, \diff  r\right|=o\left(\frac{T}{\KK^{(N+d)/2}} \|G\|_{C^1}\right).
 \]
\end{lemma}

We can now proceed to

\begin{proof}[Proof of \Cref{thm:f}]
Let us assume without loss of generality that $\Gcover$ is supported on $\{ \xi = 0\}$. To emphasize that, we write $\Gcover(x) = G(x,0)$, so
$G( \cdot , 0) = G \circ p$ for a unique $G \in C^1(\Xcomp)$ and $\int_{\Xcomp} G \, \diff\Leb = \int_{\Xcover} \Gcover \, \diff\Leb$.
By equation~\eqref{eq:obs} and \Cref{lem:rel} (first equality there with $\ell=0$ and $r=-\xi(\psicover^{\KK}(x'))$),
\begin{align}\label{eq:rn1}
\nonumber \int_0^T G \circ \phi_t(x)\, \diff t &=\frac{1}{\lambda^{\KK}}\int_0^{\lambda^{k}T}  \Lcover^{\KK}(G(x,0))\circ \phi_{r} \circ \psicover^{\KK} (x') \, \diff r\\
  &=\frac{1}{\lambda^{\KK}}\int_0^{\lambda^{k}T} \left(\Lcomp^{\KK} G(x) \1_{\{F_{\KK}(x)=\xi(\psicover^{\KK}(x'))\}}\right) \circ \phi_{r} \circ \psicover^{\KK} (x') \, \diff r.
\end{align}
% Take $K \geq 1$ minimal such that the
% $\psicover^{\KK}$-image of the segment $[x, \phi_T(x)]$ is contained in at most two adjacent elements $P \in \cP \vee \psicover^{-1} \cP$.
% We will argue as if only one $P$ is needed;
% the other case can be easily derived by cutting the segment into two pieces.
% We have $\KK \approx \log^* T$ as in the statement of the theorem.
%
% Take $\ell=0$ and $r=\xi \circ \psicover^{\KK}(x,0)$.
% By~\eqref{eq:obs} and \Cref{lem:rel} and the fact that $A_K = \{ \xi = \ell\}$,
% \begin{align*}
% \frac{1}{T}\int_0^{T} G \circ \phi_t(x)\, d t =\int_{\Xcomp} \Lcomp^{\KK} G(x) \1_{\{F_{\KK}=\xi \circ \psicomp^{\KK}\}}\, \diff\Leb +\, O(1/T).
% \end{align*}
By \Cref{cor:opllt},
\begin{align}\label{eq:pr}
 \Lcomp^{\KK} G(x) \1_{\{F_{\KK}(x)=\xi(\psicover^{\KK}(x'))\}}
 &=\frac{\int_{\Xcomp} G\, \diff\Leb}{\left(2\pi\sqrt{\KK}\right)^d}
 I_0\left(\frac{\xi(\psicover^{\KK}(x'))}{\sqrt{\KK}}\right)+
 \sum_{j=1}^{N} \frac{1}{j!}\frac{I_j\left(\Sigma, \frac{\xi(\psicover^{\KK}(x'))}{\sqrt{\KK}}\right)} {\KK^{(j+d)/2}}*e_{j,G}(x)
 \nonumber \\
 & \qquad\qquad\qquad +o\left(\KK^{-(N+d)/2}  \right),
 \end{align}
 where $e_{j,G}=\left(\Pi_0^{(j)}G+ \left(\frac{\lambda^{\KK}}{A^{\KK}}\right)_0^{(j)}\right)$ and
 \begin{align*}
 e_{j,G}=\left(\Pi_0^{(j)}G+ \left(\frac{\lambda^{\KK}}{A^{\KK}}\right)_0^{(j)}+\sum_{r_1+r_2=j}\frac{1}{r_1!r_2!}\Pi_0^{(r_1)}G\otimes \left(\frac{\lambda^{\KK}}{A^{\KK}}\right)_0^{(r_2)}\right), \quad j\ge 2.
 \end{align*}
are column vectors with $d^j$ entries which are real if $j$ is even and purely imaginary entries if $j$ is odd. Each such entry is bounded,
due to \Cref{prop:op} (b), (c). As in the proof of \Cref{lem:consopllt}, the operation $*$ between
$I_j$ and $e_{j,G}$ produces a bounded function.

 We are left with describing the integrals $I_j\left(\Sigma,\frac{\xi(\psicover^{\KK}(x'))}{\sqrt{\KK}}\right)$ and in the end combining with~\eqref{eq:rn1}.

Recall $d=1,2$ and write
$I_j(\Sigma, L)= \int_{\R^d} e^{ i u L  - \frac{\langle \Sigma u,\Sigma u\rangle}{2}} u^j\, du,
$
with $L=L(x'):=\frac{\xi(\psicover^{\KK}(x'))}{\sqrt{\KK}}$.\\

\textbf{Item I., $d=1$}. Recall that when $d=1$, we write $\Sigma=\sigma$.

\textbf{Describing the integrals $I_j$ and obtaining a close expression of $\Lcomp^{\KK} G(x) \1_{\{F_{\KK}(x)=\xi(\psicover^{\KK}(x'))\}}$.}

From \Cref{lem:integrals} in \Cref{sec:integrals}, we obtain
$I_0(\sigma,L) = \frac{\sqrt{2\pi}}{\sigma} e^{-\frac{ L^2}{2\sigma^2} }$ and
$I_j(\sigma,L) = \frac{1}{\sigma^2} (i L I_{j-1}(\sigma,L) + (j-1) I_{j-2}(\sigma,L))$ for $j \geq 1$.
%
% For $j=1,2, 3,4$,
% \begin{eqnarray*}
% I_1(\sigma,L) = i L
% \frac{\sqrt{2\pi}}{\sigma^3} e^{-\frac{ L^2}{2 \sigma^2} },
% &&
% I_2(\sigma,L) = \frac{\sqrt{2\pi}}{\sigma^3} e^{-\frac{ L^2}{2 \sigma^2} } \left( 1-\frac{ L^2}{\sigma^2} \right).
% \\
% I_3(\sigma, L) =  i L \frac{\sqrt{2\pi}}{\sigma^5} e^{-\frac{ L^2 }{2\sigma^2} } \left(3-\frac{L^2}{\sigma^2} \right),
% &&
% I_4(\sigma,L) = \frac{\sqrt{2\pi}}{\sigma^5} e^{-\frac{ L^2 }{2 \sigma^2} }
% \left( 3 - 6 \frac{L^2}{\sigma^2} +\frac{L^4}{\sigma^4}\right), \ldots
% \end{eqnarray*}
By \Cref{lem:integrals}, $I_j(\sigma, L)$  is real if $j$ is even and purely imaginary if $j$ is odd.
Thus, for real coefficients $c_{p,j}$ we can write
\begin{equation}\label{eq:Ipower}
 I_j\left(\sigma, \frac{\xi(\psicover^{\KK}(x'))}{\sqrt{\KK}}\right)  = i^{j \bmod 2}\cdot \sum_{p=0}^{\lfloor j/2\rfloor} c_{p,j}
  \frac{\xi(\psicover^{\KK}(x))^{2p+(j \bmod 2)}}{\KK^{(2p+(j \bmod 2))/2} } e^{-\frac{\xi\left(\psicover^{\KK}(x')\right)^2}{2\sigma^2\KK}  }.
\end{equation}

Recall from~\eqref{eq:pr} that the bounded functions $e_{j,G}$ are real if $j$ is even and purely imaginary if $j$ is odd.
This combines with $i^{j \bmod 2}$ to get a real coefficient.
Let $f_{k,j} = i^{j \bmod 2} e_{j,G} c_{(2k-j-(j \bmod 2))/2, j}$.
Combining~\eqref{eq:pr} and~\eqref{eq:Ipower},
\begin{eqnarray*}
 I_0 +\sum_{j=1}^N
\frac{i^{j \bmod 2} e_{j,G}(x)} {\KK^{j/2}} I_j &=&
\sum_{j=0}^N \sum_{p=0}^{\lfloor j/2 \rfloor}
c_{p,j} i^{j \bmod 2} e_{j,G}(x)
\frac{\xi(\psicover^{\KK}(x))^{2p+(j \bmod 2)}}{\KK^{(2p+j+(j \bmod 2))/2} } e^{-\frac{\xi\left(\psicover^{2\sigma^2 \KK}(x')\right)^2}{\KK}  }\\[2mm]
&=&
\sum_{k=0}^N \left( \sum_{j=0}^{2k}
f_{k,j}(x) \xi(\psicover^{\KK}(x))^{2k-j} \right)
e^{-\frac{\xi\left(\psicover^{\KK}(x')\right)^2}{2\sigma^2\KK}  }
\left(\frac{1}{\KK}\right)^k +  O\left( \left(\frac{1}{\KK}\right)^{N+1} \right),
\end{eqnarray*}
where we introduced a new summation index $2k = 2p+j+(j \bmod 2)$
and switched the order of the sums.
The terms in the inner brackets can all be computed explicitly.
We just give the first two as illustration:
\[
 I_0 + \frac{e_{1,G}(x)} {\KK^{1/2}} I_1 + \dots =
\left(
\frac{\sqrt{2\pi}}{\sigma}  + \frac{\sqrt{2\pi}}{\sigma^3}
\frac{ e_{2,G}(x) + i e_{1,G}(x) \xi(\psicover^{\KK}(x))}{\KK} + O\left( \frac{1}{\KK^2} \right)
\right)  e^{-\frac{\xi\left(\psicover^{\KK}(x)\right)^2}{2\sigma^2 \KK}  }.
\]
Putting all the above together gives, for any $x \in \Xcomp$ and $x' \in p^{-1}(x)$,
\begin{align*}
 \Lcomp^{\KK} G(x) \1_{\{F_{\KK}(x)=\xi(\psicover^{\KK}(x'))\}} &=\frac{\int_{\Xcomp}G\, \diff\Leb }{\sigma\sqrt{2\pi}}
e^{-\frac{\xi\left(\psicover^{\KK}(x')\right)^2}{ 2 \sigma^2 \KK} } \, \frac{1}{\sqrt{\KK}}\\
&\quad + \frac{ 1 }{\sigma^3 \sqrt{2\pi} }
\left( e_{2,G}(x) + i e_{1,G}(x) \xi(\psicover^{\KK}(x)) \right)
\, e^{- \frac{\xi\left(\psicover^{\KK}(x')\right)^2}{2\sigma^2\KK}  }
\, \frac{1}{(\sqrt{\KK})^3}\\
&\quad + \frac{1}{2\pi}
\, \sum_{k=2}^N \left( \sum_{j=0}^{2k}
f_{k,j}(x) \xi(\psicover^{\KK}(x'))^{2k-j} \right)
\, e^{-\frac{\xi\left(\psicover^{\KK}(x')\right)^2}{2\sigma^2 \KK}  }
\frac{1}{(\sqrt{\KK})^{2k+1}}   \\
&\quad + \ O\left(\frac{1}{(\sqrt{\KK})^{2N+3}}\right),
\end{align*}
for real bounded functions $f_{k,j}$.

\textbf{Concluding the argument in the case $d=1$, combining with~\eqref{eq:rn1}}.
We use the above displayed equation for expression of $\Lcomp^{\KK} G(x) \1_{\{F_{\KK}(x)=\xi(\psicover^{\KK}(x'))\}} $ inside~\eqref{eq:rn1} to get

\begin{align*}
 \int_0^T \Gcover \circ \phi_t(x)\, dt &=\frac{\int_{\Xcover} \Gcover \, \diff\Leb }{\sigma\sqrt{2\pi}}\frac{1}{\sqrt{\KK}}\frac{1}{\lambda^{\KK}}\int_0^{\lambda^{\KK}T}  e^{-\frac{\xi\left(\psicover^{\KK}(x')\right)^2}{ 2 \sigma^2 \KK} }  \circ \phi_{r} \circ \psicover^{\KK} (x') \, d r.\\
 +\ &\frac{1}{2\pi}\sum_{k=1}^N \frac{1}{(\sqrt{\KK})^{2k+1}} \left( \sum_{j=0}^{2k}
 \frac{1}{\lambda^{\KK}}\int_0^{\lambda^{\KK}T}
f_{k,j}(x) \xi(\psicover^{\KK}(x'))^{2k-j} e^{-\frac{\xi\left(\psicover^{\KK}(x')\right)^2}{2\sigma^2 \KK}  }\circ \phi_{r} \circ \psicover^{\KK} (x') \, d r\right)\\
 +\ &  \frac{1}{\lambda^{\KK}}\int_0^{\lambda^{\KK}T} E_{K,N} G\circ \phi_{r} \circ \psicover^{\KK} (x') \, d r ,
\end{align*}
where $ \|E_{K, N}\|=O\left(\frac{1}{(\sqrt{\KK})^{2N+3}}\right)$ and $x'\in p^{-1}(x)$.

Let $x_K=\psicover^{\KK} (x)$. Recall that one main assumption of the theorem we prove here
is that we choose $K \in \N$ so large that $\psicover^{\KK}$ maps the flow-line between $x'$ and $\phi_T(x')$  into a single copy of the fundamental domain.
Thus,
 $\xi\left(\psicover^{\KK}(\phi_r(x_K)))\right) = \xi\left(\psicover^{\KK}(x')\right)$ is constant for all $0 \le r\le \lambda^{\KK}T$. Thus,
\begin{align*}
 \frac{1}{\lambda^{\KK}}\int_0^{\lambda^{\KK}T}  e^{-\frac{\xi\left(\psicover^{\KK}(x')\right)^2}{ 2 \sigma^2 \KK} }  \cdot \phi_{r} \circ \psicover^{\KK} (x) \, d r=
  e^{-\frac{\xi\left(\psicover^{\KK}(x')\right)^2}{ 2 \sigma^2 \KK} }\cdot \left(\frac{1}{\lambda^{\KK}}\int_0^{\lambda^{\KK}T} 1\, dr\right)= e^{-\frac{\xi\left(\psicover^{\KK}(x')\right)^2}{ 2 \sigma^2 \KK} }\cdot T
\end{align*}
By a similar argument,
\begin{align*}
 \frac{1}{\lambda^{\KK}}\int_0^{\lambda^{\KK}T}
f_{k,j}(x) \xi(\psicover^{\KK}(x'))^{2k-j} & e^{-\frac{\xi\left(\psicover^{\KK}(x')\right)^2}{2\sigma^2 \KK}  }\circ \phi_{r} \circ \psicover^{\KK} (x') \, d r \\
= \ &  \xi(\psicover^{\KK}(x'))^{2k-j} e^{-\frac{\xi\left(\psicover^{\KK}(x')\right)^2}{2\sigma^2 \KK}  }\frac{1}{\lambda^{\KK}}\int_0^{\lambda^{\KK}T}
f_{k,j}(\phi_{r}(x_K)) \, d r.
\end{align*}
Since the functions $f_{k j}$ are bounded,
$f_{k,j}^{*}(x) := \frac{1}{\lambda^{\KK}}\int_0^{\lambda^{\KK}T}
f_{k,j}(\phi_{r}(x_K)) \, d r$ are bounded as well.

Thus,
\begin{align*}
 \int_0^T \Gcover \circ \phi_t(x)\, \diff t &=\frac{\int_{\Xcover}G\, \diff\Leb }{\sigma\sqrt{2\pi}}
e^{-\frac{\xi\left(\psicover^{\KK}(x')\right)^2}{ 2 \sigma^2 \KK} } \, \frac{T}{\sqrt{\KK}}\\
&\quad + \frac{1}{2\pi}
\, \sum_{k=2}^N \left( \sum_{j=0}^{2k}
f_{k,j}^{*}(x) \xi(\psicover^{\KK}(x'))^{2k-j} \right)
\, e^{-\frac{\xi\left(\psicover^{\KK}(x')\right)^2}{2\sigma^2 \KK}  }
\frac{T}{(\sqrt{\KK})^{2k+1}}   \\
&\quad + \frac{1}{\lambda^{\KK}}\int_0^{\lambda^{\KK}T} E_{K,N} G\circ \phi_{r} (x_K) \, \diff r.
\end{align*}
By \Cref{lemma:EKN}, $\left|\frac{1}{\lambda^{\KK}}\int_0^{\lambda^{\KK}T} E_{K,N} G\circ \phi_{r} (x_K) \, \diff  r\right|=O\left(\frac{T}{(\sqrt{\KK})^{2N+3}}\|G\|_{C^1}\right)$.
Item I. follows with $g_{k,j}=\frac{f_{k,j}\sigma}{ \int_{\sqrt{2\pi}\Xcomp}G\, \diff\Leb}$.\\

\textbf{Item II, $d=2$}. Similarly to the case $d=1$, we first
want a precise expression of $\Lcomp^{\KK} G(x) \1_{\{F_{\KK}(x)=\xi(\psicover^{\KK}(x'))\}}$.

In this case, both $I_j(\Sigma, L)$ and $e_{j,G}$ in~\eqref{eq:pr}
are column vectors with $2^j$ entries. We already know that the product gives a bounded function.

We diagonalize
$\Sigma = A J A^{-1}$ for $J = \binom{\sigma_1\ \\ 0}{\ 0\  \\ \ \sigma_2}$ for a unitary matrix $A$,
so $A^{-1} = A^*$ is the transpose of $A$. Note that $\sqrt{\det\Sigma^2}=\sigma_1\sigma_2$.

\Cref{lem:integrals} B computes the integrals $I_0, I_1, I_2$ precisely.
In particular, $I_0(\Sigma, L) = \frac{2\pi}{\sigma_1 \sigma_2}
e^{- \frac12 \langle \Sigma^{-1} L, \Sigma^{-1} L \rangle  }$,
so when inserted in \eqref{eq:pr} with $L = \xi(\psicover^{\KK}(x'))/\sqrt{\KK}$ and after the integration $\frac{1}{\lambda^{\KK}} \int_0^{\lambda^{\KK}T}$ to pick up an extra factor $T$, we obtain as constant term
$$
\frac{\int_{\Xcover} \Gcover \diff\Leb}{2\pi \sqrt{\det \Sigma^2}}
\cdot e^{- \frac12 \langle \Sigma^{-1} \frac{\xi(\psicover^{\KK}(x'))}{\sqrt{\KK}},
\Sigma^{-1}\frac{\xi(\psicover^{\KK}(x'))}{\sqrt{\KK}} \rangle }
\frac{T}{\KK}.
$$
Also,
\[
 I_1(\Sigma, L) =
\frac{2\pi i}{\sigma_1 \sigma_2} e^{- \frac12 \langle \Sigma^{-1} L, \Sigma^{-1} L \rangle }
A \begin{pmatrix}
 \frac{1}{\sigma_1^2} (A^*L)_1 \\[1mm]
 \frac{1}{\sigma_2^2} (A^*L)_2
     \end{pmatrix},
\]
and $*$-multiplication with the purely imaginary vector $e_{1,G}$ produces a scalar linear form which we can denote as $\langle g_1(x), \cdot \rangle$. Then applying integration $\frac{1}{\lambda^{\KK}} \int_0^{\lambda^{\KK} T}$ to \eqref{eq:pr} gives as linear term
$$
\frac{\int_{\Xcover} \Gcover \diff\Leb}{2\pi \sqrt{\det \Sigma^2}} \cdot
e^{- \frac12 \langle \Sigma^{-1} \frac{\xi(\psicover^{\KK}(x'))}{\sqrt{\KK}},
\Sigma^{-1}\frac{\xi(\psicover^{\KK}(x'))}{\sqrt{\KK}} \rangle }
\frac{T}{\KK} \frac{\langle g_1(x), \xi(\psicover^{\KK}(x')) \rangle}{\KK}.
$$
For the quadratic term, we have $I_2(\Sigma,L) = \frac{2\pi}{\sigma_1\sigma_2} e^{- \frac12 \langle  \Sigma^{-1} L,  \Sigma^{-1} L \rangle }
(Q(L \otimes L) + Q')$,
$*$-multiplication with the real vector $e_{2,G}$ produces a scalar quadratic form which we can denote as $\langle \cdot , g'_2(x) \cdot \rangle + g'_2(x)$. Replacing $L$ by $\xi(\psicover^{\KK}(x'))/\sqrt{\KK}$ and then applying integration $\frac{1}{\lambda^{\KK}} \int_0^{\lambda^{\KK} T}$ to \eqref{eq:pr} gives as quadratic term
$$
\frac{\int_{\Xcover} \Gcover \diff\Leb}{2\pi \sqrt{\det \Sigma^2}} \cdot
e^{- \frac12 \langle \Sigma^{-1} \frac{\xi(\psicover^{\KK}(x'))}{\sqrt{\KK}},
\Sigma^{-1}\frac{\xi(\psicover^{\KK}(x'))}{\sqrt{\KK}} \rangle }
\frac{T}{\KK} \frac{\langle \xi(\psicover^{\KK}(x')), g_2(x) \xi(\psicover^{\KK}(x')) \rangle + g'_2} {\KK^2},
$$
as required.
\end{proof}

Using \Cref{thm:f} and \Cref{lemma:crtech}, we complete

 \begin{proof}[Proof of~\Cref{thm:rateswrd}]
	 We prove \Cref{thm:itema}.  \Cref{thm:itemb} follows similarly.

  By \Cref{thm:f}, the following holds for bounded functions $g_{k,j}$.
  \begin{align}\label{eq:bl}
\nonumber \int_{\fund}& \int_0^{T}  G \circ \phi_t(x)\, \diff t\, \diff\Leb =\frac{\int_{\Xcover}G\, \diff\Leb}{\sigma \sqrt{2\pi}}\frac{T}{\sqrt{\KK}}
\cdot \int_{\Xcover} e^{-\frac{\xi\left(\psicover^{\KK}(x')\right)^2}{ 2 \sigma^2 \KK} }  \diff\Leb \\
&\quad +
\, \sum_{k=1}^N \frac{T}{(\sqrt{\KK})^{2k+1}} \left( \sum_{j=0}^{2k}
 \int_{\Xcover}g_{k,j}(x)\xi(\psicover^{\KK}(x'))^{2k-j}
\, e^{-\frac{\xi\left(\psicover^{\KK}(x)\right)^2}{2\sigma^2 \KK}  }\, \diff\Leb \right)
 %\\\nonumber &\quad
+ \ O\left(\frac{T}{(\sqrt{\KK})^{2N+3}}\right).
\end{align}

We want to apply
\Cref{lemma:crtech} with $f_q(x, w)=e^{-\frac{w^2}{2}} w^q g_{k,j, q}(x)$,  with $q\in\{0,\ldots, 2k-j\}$, $w$
replaced by $\frac{\xi\left(\psicover^{\KK}(x')\right)}{\sigma \sqrt{\KK}}$ and $g_{k,j, 0}=1$ and $g_{k,j, q}=g_{k,j}$ for $q\ne 0$.

By \Cref{lemma:crtech} (ii) (so, the case $q=0$),
\[
\int_{\fund}
\, e^{-\frac{\xi\left(\psicover^{\KK}(x')\right)^2}{2\sigma^2 \KK}  }\, \diff\Leb =
 \E(f_0(\chi)) + \sum_{j=1}^{n} \frac{d_j}{K^{j/2}}
 + o\left(\frac{1}{\sqrt{\KK}^n}\right),
\]
for real constants $d_j$.

To deal with the sum in~\eqref{eq:bl}, we apply
\Cref{lemma:crtech}(i) with $f_q(x, w)=e^{-\frac{w^2}{2}} w^q g_{k,j, q}(x)$,  with $q\in\{0,\ldots, 2k-j\}$ described in the last to previous paragraph.
This ensures that the sum $\sum_{j=0}^{2k}$ of the integrals in~\eqref{eq:bl}
convergence to $\sum_{j=0}^{2k}d_{k,j}$  for real constants  $d_{k,j}$.
 \end{proof}

\subsection{Banach spaces and Proof of Proposition~\ref{prop:op}}
\label{sec:banach}

There are several choices in the literature for the Banach spaces we can use,
 see the surveys~\cite[Section 2]{Baladi17} and~\cite{Demers18}.
For the automorphism $\psicomp$ it is convenient to work with a variant of the spaces introduced in~\cite{DemersLiverani08}
(see also~\cite{DZ14} and references therein for generalizations applicable to billiards) applicable to a class of
hyperbolic maps with singularities. A possible alternative choice would be the anisotropic Banach spaces considered in \cite{FGL}, which are a great tool for studying the Ruelle spectrum for general pseudo Anosov maps.
In the current setup we are only interested in the spectral gap of the transfer operator of the above mentioned simple automorphism (along properties of the twist).
This is why we use the spaces in \cite{Demers18}, which among others, allows us to use some facts already established for this class of automorphisms.

We find it convenient to work
with a slight modification of the Banach spaces considered in~\cite{Demers18} for the purpose of obtaining limit
theorem via spectral methods for a general class of baker maps\footnote{The baker map itself is defined as $b(x,y) = (2x \bmod 1, \frac12(y + \lfloor 2x \rfloor))$ on the unit square.}. For a similar (simplified) variation of the spaces in~\cite{DemersLiverani08} of the Banach spaces in~\cite{Demers18} we refer to~\cite{LT16},
which focused on some two-dimensional, non-uniformly hyperbolic versions of Pomeau-Manneville maps.

The automorphism $\psicomp$ resembles a baker map except for the existence of singular (and potentially
and marked) points, see \Cref{sec:homological}.
For a baker map the singularities are given by the set of discontinuity points.
In the setup of $\psicomp$,
we say that a point $s\in\Xcomp$ is singular if the cone angle at $s$ is not $2\pi$.
The difference in the type of singularities and potentially marked points introduces a difference in the class of admissible leaves.
In all other aspects the variant of the Banach spaces in~\cite{Demers18}
remain the same in the set of $\psicomp$. We summarize below the ingredients of these Banach spaces, using the notation of~\cite{Demers18},
as to emphasize that the case of the automorphism $\psicomp$ (regarding the spectral gap for $\Lcomp$) is
one of the easiest possible examples that the spaces introduced in~\cite{DemersLiverani08} can treat.

We remark that Proposition~\ref{prop:op}, of which proof we sketch below, is not new with us, and that our only tasks is collect the statements scattered throughout~\cite{Demers18} and similar papers mentioned below.

\subsubsection{Definitions of Banach spaces}\label{subsec:ban}

Although, the presence of a (natural) Markov partition is not a crucial element in the
construction in~\cite{Demers18} for baker type maps, it does simplify the writing.
The presence of this type of Markov structure
considerably simplifies  the description of admissible leaves.
In particular, it allows us to define admissible leaves
as full unstable segments. For the same reason, that of simplicity, we will take advantage
of the Markov partition $\cP$.
%and more precisely of $\cP^R := \bigvee_{j=0}^R \psicomp^j \cP$.

We define the set  $\cW^s$ of \emph{admissible leaves} as the set of stable segments $W$ that exactly stretch across an element $P \in \cP$ such that its (one dimensional) interior is contained in the interior of $P$.
Note that, for any such $W \in \cW^s$, the stable segment $\psicomp^{-1}W$ can be decomposed into a finite union of elements of $\cW^s$.

Any $W \in \cW^s$ has an affine parametrization $\{\chi_W(r) : r \in [0,l]\}$, where $l$ is the length of $W$. For any measurable function $h : W \to \C$, we write
\[
\int_W h \diff\Leb = \int_0^l h \circ \chi_W(r) \diff r.
\]
%It is usual to construct $\cP$ based on stable and unstable manifold of integers lattice points.
% The slopes of these manifold belong to $\Q[\lambda]$ and $\lambda^{-1}$ is a Pisot number, so
% by the Garsia separation principle \cite{G62}, there is a constant $C>0$, independent of $R$,
% such that the stable length of all the elements of $\cP^R$ is between $|\lambda|^R/C$ and
% $C |\lambda|^R$. Therefore we use $|\lambda|^{-R}$ to normalize the integrals over leaves $W \in \cW^s$.

Let $\alpha \in [0,1]$. For any $W \in \cW^s$, we let $C^\alpha(W,\C ) $ denote the Banach space of complex-valued functions $W$ with H\"older exponent $\alpha$, equipped with the norm
\[
|h|_{C^\alpha(W,\bC)}=\sup_{z\in W} |h(z)|+\sup_{z,w\in W}\frac{|h(z)-h(w)|}{|z-w|^\alpha}.
\]
Such a set is a collection of local unstable manifolds that do not contain a singularity point.
From here onward all required definitions are as in~\cite[Section 2.2]{Demers18}.

We say that $\vf\in C^\alpha(\Xcomp,\bC)$ if it is $C^\alpha(W,\bC)$ for all $W\in\cW^s$.
Given $h \in C^1(\Xcomp,\C)$, define the \emph{weak norm} by
\begin{equation*}\label{eq-weaknorm}
\|h\|_{\cB_w}:=\sup_{W\in\cW^s}\;\sup_{|\phi|_{ C^1(W,\bC)}\leq 1 }  \int_W h\phi\, \diff\Leb.
\end{equation*}
Given $\alpha\in [0,1)$, define the \emph{strong stable norm} by
\begin{equation*}\label{eq-strongnormst}
\|h\|_s:=\sup_{W\in\cW^s}\;\sup_{|\phi|_{ C^\alpha(W,\bC)}\leq 1 }
 \int_W h\phi\, \diff\Leb.
\end{equation*}
For any two {\em aligned}\footnote{i.e., the one is obtained from the other by a translation in the unstable direction.} admissible leaves $W_1, W_2 \in \cW^s$ \emph{in the same atom of $\cP$}, let $d(W_1,W_2)$ denote the distance \emph{in the unstable direction} between $W_1$ and $W_2$. In other words, if $W_i = \{\chi_i(r) : r \in [0,l]\}$, then $d(W_1,W_2)$ is the length of the segment in the unstable direction connecting $\chi_1(r)$ to $\chi_2(r)$.

With the same notation as above, for two functions $\vf_i\in C^1(W_i,\bC)$, with $i=1,2$, we also define
\[
d_0(\vf_1, \vf_2)=\sup_{r \in [0,l]} |\vf_1 \circ \chi_r(x_1)-\vf_2 \circ \chi_r(x_2)|.
\]
Next define the \emph{strong unstable norm} by
\begin{equation*}\label{eq-strongnormunst}
\| h \|_u := \sup_{W_1, W_2\in\cW^s}\,
\sup_{|\vf_i|_{C^1}\leq 1,\, d_0(\vf_1,\vf_2)=0}\frac{1}{d(W_1, W_2)^{1-\alpha}}\left|\int_{W_1} h\vf_1\,  \diff\Leb  -\int_{W_2} h\vf_2\, \diff\Leb \right|.
\end{equation*}
Finally, the \emph{strong norm} is defined by $\|\vf\|_{\cB}=\|\vf\|_s+\|\vf\|_u$.
These norms are exactly those of \cite[Section 2.3]{Demers18}.
% but normalized with $\lambda^{-R}$ because
% our use of the finer partition
% $\cP^R$.

Define the weak space $\cB_w$ to be the completion of $C^1(\Xcomp)$ in the weak norm and define $\cB$ to be the completion of $C^1(\Xcomp)$ in the strong norm.

 \begin{lemma}~\cite[Lemma 2.4]{Demers18} (see also~\cite[Lemma 7.2]{LT16})\label{lem:emb} We have the following sequence of continuous, injective embeddings: $C^1(\Xcomp)\subset\cB\subset\cB_w\subset (C^1(\Xcomp))^*$. Moreover, the unit ball of $\cB$ is
 relatively compact in $\cB_w$.\end{lemma}

 \subsubsection{Well-definedness and boundedness of $\Lcomp$ on $\cB$ and $\cB_w$. Proof of Proposition~\ref{prop:op}(a)}

 Recall that $\psicomp$ is piecewise affine (so $\psicomp$ is $C^1(W)$, for any $W\in\cW^s$) and note that for any
 $\alpha\in [0,1]$, for any $W\in\cW^s$ and for
any $\vf\in C^\alpha(W,\bC)$, $\vf\circ \psicomp\in C^\alpha(\Xcomp,\bC)$.
Moreover, for any $n\ge 1$, $\psicomp^{-n}(W)$ consists of a union of leaves in $\cW^s$ and
the transfer operator $\Lcomp$ of $\psicomp$ is defined as
 \begin{align}
 \Lcomp h(\vf)= h(\vf\circ\psicomp), \quad \text{ for all } \vf\in C^\alpha(\cW^s) \text{ and } h\in(C^\alpha(\cW^s))^*.
 \end{align}
Recall that Lebesgue measure $\Leb$ is invariant for $\psicomp$.
 We identify $h$ with the measure $\diff \mu = h\, \diff\Leb$. Then
 $h\in C^1(\cW^s)\subset (C^1(\cW^s))^*$ and $\Lcomp h$ is associated with the measure having density
 \begin{align}\label{eq:trop}
 \Lcomp h(x)=\frac{h \circ \psicomp^{-1}(x)}{J_\psicomp(\psicomp^{-1}(x))}=h \circ \psicomp^{-1}(x),
 \end{align}
 where $J_\psicomp$ is the Jacobian of $\psicomp$ with respect to $\Leb$,
 which is equal to $1$ (since the contraction and expansion are the same).

 In general, it is not true that for systems with discontinuities, $\Lcomp(C^1(\Xcomp))\subset C^1(\Xcomp)$ and hence it is not obvious that $\Lcomp$
 is well defined on $\cB$: see, for instance,~\cite[Footnote 13]{Demers18}. However, in the current
 setup of $\psicomp$, similar to the first line of the proof of~\cite[Lemma 4.1]{Demers18},
 %due to~\eqref{eq:trop},
 $\Lcomp h\in C^1(\cW^s)$ (since for any $W\in\cW^s$, $\psicomp^{-1}W$ is an exact union of leaves in $\cW^s$).
 Hence, $\Lcomp(C^1(\Xcomp))\subset C^1(\Xcomp)$ and \textbf{$\Lcomp$ is well defined on $\cB$}.

 Also, by~\cite[Lemma 4.1]{Demers18}, \textbf{$\Lcomp$ acts continuously on $\cB$ and $\cB_w$}
and the proof of~\cite[Theorem 2.5]{Demers18} (for baker type maps) goes word for word the same
in the setup of $\psicomp$. This yields

\begin{lemma}{~\cite[Theorem 2.5]{Demers18}}\label{lemma:spg} The operator $\Lcomp$ is quasi-compact as an operator on $\cB$.
That is, its spectral radius is $1$ and its essential
spectral radius is strictly less than $1$. Moreover,
$1$ is a simple eigenvalue, and all other eigenvalues have modulus strictly less than $1$.
\end{lemma}

% The computations on the relevant Lasota-Yorke inequalities are exactly the same as
% in~\cite[Proposition 4.2]{Demers18}.
% % the fact that we use a finer partition $\cP^R$
% % has no effect on these inequalities and hence
% % the bound on the essential spectral radius is independent of $R$.
% We use here that due to the uniform expansion factor $\lambda^{-1}$, the
% preimage $\psicomp^{-n}(W)$ is a union of admissible leaves of total length equal to $|\lambda|^{-n} |W|$.

 Lemmas~\ref{lem:emb} and~\ref{lemma:spg} is exactly the content of
 Proposition~\ref{prop:op}(a).

\subsubsection{Analyticity of the twisted transfer operator $\Lcomp_u f=\Lcomp(e^{iu F} f)$, $f\in\cB$.\\
Proof of Proposition~\ref{prop:op}(b)}
\label{sec:anal}

The Frobenius function $F:\Xcomp\to\Z$, $x \mapsto \xi \circ \psicover(x')-\xi(x')$
for $x' \in p^{-1}(x)$,
is not globally $C^1$, hence the simple argument of~\cite[Lemma 4.8]{Demers18}
cannot go through.
However, $F$ is constant on each element of the partition
$\cP^R \vee \psicomp^{-1} \cP^R$ (hence $C^\infty$ on each element of $\cP^R \vee \psicomp^{-1} \cP^R$). As a consequence,
the argument for the analyticity of the twisted transfer operator is a much simplified version of the argument
used in the proof
of~\cite[Lemma 3.9]{DPZ} (essentially a consequence of the arguments used in~\cite{DZ11,DZ13,  DZ14}).

\begin{lemma}\label{lemma:anal}
Let $u\in \R^d$, $f\in\cB$ and $m\ge 1$. Then $\frac{d^k}{du^k}\Lcomp_u f$ is a linear operator
on $\cB$  with operator norm of $O(\|F\|_\infty^k)$.
\end{lemma}

\begin{proof} Using~\eqref{eq:trop}, compute that
\begin{align*}
\frac{d^k}{du^k}\Lcomp_uf=i^k \Lcomp(F^k e^{i u F} f)= i^k(F^k e^{iu F} )\circ\psicomp^{-1}\Lcomp f.
\end{align*}
Since $F$ is locally constant and since each element of $\cP \vee \psicomp^{-1} \cP$ contains no singularities in its interior,
a simplified version\footnote{The (serious) simplification comes from the simple form of admissible
leaves and the fact that the Jacobian is constant.} of the argument used in~\cite[Lemma 3.7]{DZ13}
(see also~\cite[Lemma 3.3]{DPZ})
shows that for any $f\in\cB$, $f\, F\in\cB$ and that for some $C>0$,
\begin{align}\label{eq:inB}
\|f\, F\|_{\cB}\le C \|f\|_{\cB}\sup_{P_i\in\cP^R}\|F\|_{C^\alpha(P_i)}.
\end{align} Thus,
\begin{align*}
\left\|\frac{d^k}{du^k}\Lcomp_uf \right\|_{\cB}\le C\sup_{P_i\in\cP^R}\|(F^k e^{iu F} )\circ\psicomp^{-1}\|_{C^\alpha(P_i)}
\| f \|_{\cB}\le C\|F\|_\infty^k\|f\|_{\cB}.
\end{align*}
\end{proof}

\subsection{Spectrum of $\Lcomp_u$ and leading eigenvalue. Proof of \Cref{prop:op}(c)}

We already know (see Lemma~\ref{lemma:spg}) that $1$ is a simple isolated eigenvalue of $\Lcomp_0$.
Since $u\to\Lcomp_u$ is analytic (see Lemma~\ref{lemma:anal}),
there exists $\delta>0$ and a simple family of simple eigenvalues $\lambda_u$, analytic in $u\in (0,\delta)$ with $\lambda_0=1$.
Standard perturbation theory (see~\cite{ITM} and, for instance,~\cite[Section 2]{Gousurv}) ensures that for all $u\in(0,\delta)$,
\begin{align}\label{eq:decomp}
\Lcomp_u^n=\lambda_u^n\Pi_u+Q_u^n,
\end{align}
where $\Pi_u$ is the spectral projection onto the one-dimensional eigenspace associated to $\lambda_u$ with $\Pi_0 f=\int_{\Xcomp} f\, \diff\Leb$ and where $\|Q_u^n\|\le \theta^n$ for some $\theta\in (0,1)$,
and $Q_u \Pi_u = \Pi_u Q_u$.
By \Cref{lemma:anal} and standard perturbation theory (see~\cite{ITM}),
all the eigen-elements are again analytic. That is, $\Pi_u, Q_u$ are also analytic in $u\in (0,\delta)$.

% Equations~\eqref{eq:decomp} and \eqref{eq:expl}, together with standard arguments (via the so-called Aaronson-Denker-Nagaev-Guivarch's method: see~\cite{AD01, Gousurv})
% ensure CLT for $F_n$, that is equation~\eqref{eq:clt}.

\subsection{Spectrum of $\Lcomp_u$ for $u\in (\delta,\pi]$. Proof of \Cref{prop:op}(d)}

The following lemma gives the required control.

\begin{lemma}{~\cite[Lemma C.1]{DPZ}}\label{lem:aper}
Let $u\in (-\pi,\pi]$, $h\in\cB$ and $\eta\in\C$  be such that $\Lcomp_u h = \eta h$ in $\cB$ and $|\eta|\ge 1$. Then either
$h\equiv 0$ or $u\in 2\pi\Z$ and $h$ is $\Leb$-a.s.\ constant.
\end{lemma}

The proof of \Cref{lem:aper} goes word for word as~\cite[Proof of Lemma C.1]{DPZ}, except for differences in notations.
The differences in the definitions of the norms used in Subsection~\ref{sec:banach} are irrelevant for this argument.
\Cref{lem:aper}  ensures that there exists $\delta_1\in (0,1)$ so that
$\|\Lcomp_u^n\|_{\cB}\le \delta_1^n$ for all $|u|>\delta$ and all $n\ge 1$.

\subsection{Proof of \Cref{prop:opllt}(a)}\label{sec:remainingproofs}
We first record a technical lemma that will be instrumental in the proof of \cref{prop:itema} of \Cref{prop:opllt}.
For integers $m, n \geq 0$ and $r_1, \dots, r_n \in \N$, define the operator
\begin{equation}\label{eq:G}
 G(y,m,n,r_1, \dots, r_n)
 = (y-1)^{-m} (y-\Lcomp_0)^{-1}
 \Lcomp_0(F^{\otimes r_1}) \otimes \cdots \otimes (y-\Lcomp_0)^{-1} \Lcomp_0(F^{\otimes r_n}) \otimes
 (y-\Lcomp_0)^{-1}.
\end{equation}
As $F$ takes values in $\Z^d$, the tensor products $F^{\otimes r_j}$ are vectors with $d^{r_j}$ components.
\begin{lemma}\label{lem:G}
For all $m, n \geq 0$, $r_1,\dots, r_n \in \N$ and real-valued $v \in \cB$,
the contour integral
$$
\int_{|y-1| = \delta}  G(y,m,n,r_1, \dots, r_n) v \, dy
$$
is a vector with purely imaginary entries.
\end{lemma}

\begin{proof}
 We will use induction on $m,n$, starting with $n=0$.
 If $m = n = 0$, then
 \begin{equation*}\label{eq:vec_im}
 \int_{|y-1| = \delta}  G(y,0,0) v \, dy
 = \int_{|y-1| = \delta}  (y-\Lcomp_0)^{-1} v \, dy = 2\pi i\, \Pi_0 v
\end{equation*}
is a vector with purely imaginary entries.

Now if the statement holds for $n = 0$ and all $0 \leq m' < m$, then
\begin{eqnarray*}
 \int_{|y-1| = \delta}  G(y,m,0) v \, dy
 &=& \int_{|y-1| = \delta}  (y-1)^{-m} (y-\Lcomp_0)^{-1} v \, dy \\
 &=& \int_{|y-1| = \delta}  (y-1)^{-m} (y-\Lcomp_0)^{-1} (v-\int_{\Xcomp} v \, \diff\Leb)  \, dy\\
&&  + \int_{|y-1| = \delta} (y-1)^{-m}  (y-\Lcomp_0)^{-1} \int_{\Xcomp} v \, \diff\Leb  \, dy
\\
&=& \int_{|y-1| = \delta}  (y-1)^{-m}
\left[ (y-\Lcomp_0)^{-1} - (1-\Lcomp_0)^{-1} \right]  (v-\int_{\Xcomp} v \, \diff\Leb)  \, dy \\
&& +\
\int_{|y-1| = \delta}  (y-1)^{-m}
(1-\Lcomp_0)^{-1}  (v-\int_{\Xcomp} v \, d\mu)  \, dy \\
&& +\ \int v_{\Xcomp} \, d\mu \int_{|y-1| = \delta} (y-1)^{-(m+1)}   \, dy.
 \end{eqnarray*}
The third integral %equals $2\pi i \int v \, d\mu$ if $m=0$ and
vanishes because $m > 0$.
For the second integral, write $v_1 = (1-\Lcomp_0)^{-1}  (v-\int_{\Xcomp} v \, d\mu)$,
so $v_1 \in \cB$ is real (with $\int_{\Xcomp} v_1 \, \diff\Leb = 0$).
Hence this integral equals $v_1 \int_{|y-1| = \delta} (y-1)^{-m} \, dy = 2\pi i v_1$ for $m = 1$ and $0$  if $m > 0$.
For the first integral we use the resolvent identity:
\begin{align*}
 \int_{|y-1| = \delta} & (y-1)^{-m}
\left[ (y-\Lcomp_0)^{-1} - (1-\Lcomp_0)^{-1} \right]  \left(v-\int_{\Xcomp} v \, \diff\Leb \right)  \, dy\\
=& \ -\int_{|y-1| = \delta}  (y-1)^{1-m}
(y-\Lcomp_0)^{-1} (1-\Lcomp_0)^{-1}  (v-\int_{\Xcomp} v \, \diff\Leb)  \, dy \\
=&\ \lambda \int_{|y-1| = \delta}  (y-1)^{1-m}
(y-\Lcomp_0)^{-1} v_1  \, dy = -\int_{|y-1| = \delta} G(y,m-1,0) v_1  \, dy.
\end{align*}
This is purely imaginary by the induction hypothesis and since $m \geq 1$.
%When $m=0$, the factor $y-1$ removes the simple pole of $G(y,m-1,0) = (y - \Lcomp_0)^{-1}$,
% and the integral vanishes by Cauchy's Theorem.

Now we continue with the induction step over $n$;
in this case $G(y,m,n,r_1, \dots, r_n)$ contains $n+1$
factors $(y-\Lcomp_0)^{-1}$, and therefore it has a pole at $1$ of order $\leq n+1$.
Our induction hypothesis is that $\int_{|y-1| = \delta} G(y,m',n', r_1, \dots, r_{n'}) v  \, dy$
is purely imaginary for every real-valued $v \in \cB$ when
$0 \leq n' < n$ and $m' \geq 0$ or when $n' = n$ and $n \leq m' < m$.
Then
\begin{align*}
 \int_{|y-1| = \delta}  G(y,m,n,r_1, \dots, r_n) v \, dy
 &=  \int_{|y-1| = \delta}  G(y,m,n,r_1, \dots, r_n) (v-\int_{\Xcomp} v \, \diff\Leb) \, dy\\
  + & \int_{|y-1| = \delta}  G(y,m,n-1,r_1, \dots, r_{n-1})
 \Lcomp_0(F^{\otimes r_n}) \otimes (y-\Lcomp)^{-1} \int_{\Xcomp} v \, \diff\Leb  \, dy.
\end{align*}
The second integral is equal to $\int_{\Xcomp} v \, \diff\Leb \cdot \int_{|y-1| = \delta}  G(y,m+1,n-1,r_1, \dots, r_{n-1}) \Lcomp_0(F^{\otimes r_n}) \, dy$
and thus purely imaginary by the induction hypothesis.
We rewrite the first integral  to
\begin{eqnarray*}
 && \int_{|y-1| = \delta}  G(y,m,n-1,r_1, \dots, r_{n-1}) \Lcomp_0(F^{\otimes r_n})
 \otimes  \left[ (y-\Lcomp_0)^{-1} - (1-\Lcomp_0)^{-1} \right] (v-\int_{\Xcomp} v \, \diff\Leb) \, dy \\
&& \qquad + \int_{|y-1| = \delta}  G(y,m,n-1,r_1, \dots, r_{n-1}) \Lcomp_0(F^{\otimes r_n})
 \otimes  (1-\Lcomp_0)^{-1} (v-\int_{\Xcomp} v \, \diff\Leb) \, dy.
  \end{eqnarray*}
With $v_2 =  \Lcomp_0(F^{\otimes r_n}) (1-\Lcomp)^{-1} (v-\int_{\Xcomp} v \, \diff\Leb)$,
the second term becomes
$\int_{|y-1| = \delta}  G(y,m,n-1,r_1, \dots, r_{n-1}) v_2 \, dy$, which is purely imaginary by induction.
The resolvent identity applied to the first term gives
\begin{align*}
-\int_{|y-1| = \delta} & G(y,m-1,n-1,r_1, \dots, r_{n-1}) \Lcomp_0(F^{\otimes r_n})
\otimes (y-\Lcomp_0)^{-1} (v-\int_{\Xcomp} v \, \diff\Leb) \, dy \\
&=\
-\int_{|y-1| = \delta}  G(y,m-1,n,r_1, \dots, r_n)  (v-\int_{\Xcomp} v \, \diff\Leb) \, dy.
\end{align*}
This is purely imaginary by the induction hypothesis. If, however, $m = n$,
then the integrand contains a factor $(y-1)^{n+1}$, which removes the pole (of order $\leq n+1$) of the remaining part of the integrand, and hence Cauchy's Theorem gives again that the integral vanishes.
This completes the induction and the entire proof.
\end{proof}

We can now complete

\begin{proof}[Proof of \Cref{prop:opllt}(a).]
	%\textbf{\Cref{prop:itema}}
Recall that  $\Pi_u = \frac{1}{2\pi i} \int_{|y-1|=\delta} (y-\Lcomp_u)^{-1} \, dy$ is the
eigenprojection w.r.t.\ the leading eigenvalue.
Clearly $\Pi_0v$ is real for a real $v \in \cB$.
Taking the $j$-th derivative w.r.t.\ $u$ and then evaluating at $u = 0$, gives $2\pi i^{j+1}$ times the contour integral of a linear combination of
 terms of the form \eqref{eq:G}.
 These integrals are all purely imaginary by \Cref{lem:G}, so the $j$-th derivative produces alternatingly real and purely imaginary outcomes. So, $\Pi_0^{(j)} v$ has only real entries if $j$ is even
and purely imaginary entries if $j$ is odd.

 It remains to look at the $j$-th derivatives $\left(\frac{\lambda}{A}\right)_0^{(j)}$ of $\frac{\lambda_u}{A_u}$ evaluated at $0$. We recall that $\lambda_u$ is the eigenvalue and $A_u=e^{-\langle \Sigma u,\Sigma u\rangle}$.

 Let $v_u=\frac{\Pi_u1}{\int \Pi_u1\, \diff\Leb}$ be the normalized eigenvector associated with $\lambda_u$,
i.e., $\int_{\Xcomp} v_u\, \diff\Leb=1$. Since $\Pi_u$ is analytic, so is $v_u$. $\Pi_0^{(j)} 1$ has only real entries if $j$ is even
and purely imaginary entries if $j$ is odd, the same applies to
$v_0^{(j)} v$ has only real entries if $j$ is even
and purely imaginary entries if $j$ is odd.

A simple calculation starting from $\int_{\Xcomp}\Lcomp_u v_u\, \diff\Leb =\lambda_u\int_{\Xcomp}v_u\,  \diff\Leb =\lambda_u$
shows that we can write
\begin{align*}
 1-\lambda_u &=\int_{\Xcomp}(1-e^{iuF})\, \diff\Leb +\int_{\Xcomp}(1-e^{iuF})(v_u-v_0)\,  \diff\Leb\\
 &=-\sum_{j=1}^\infty\frac{i^j u^{\otimes j}}{j!} * \E(F^{\otimes j}) +\int_{\Xcomp} \left(\sum_{j=1}^\infty\frac{i^j u^{\otimes j}}{j!}\, F^{\otimes j}\right) * \left(\sum_{j=1}^\infty\frac{i^j u^{\otimes j}}{j!}\, v_0^{(j)}\right)\, \diff\Leb.
 \end{align*}

 Using that $v_0^{(m)}$ has only real entries if $m$ is even
and purely imaginary entries if $m$ is odd,
we can see that the same applies to every term
of product of sums $\left(\sum_{j=1}^\infty\frac{i^j u^{\otimes j}}{j!}\, F^{\otimes j}\right) * \left(\sum_{j=1}^\infty\frac{i^j u^{\otimes j}}{j!}\, v_0^{(j)}\right)$.
Clearly, every $j$ term in the sum $\sum_{j=1}^\infty\frac{i^j u^{\otimes j}}{j!}\E(F^{\otimes j})$ has real entries if $j$ is even
or purely imaginary entries if $j$ is odd.
Thus, $\lambda_0^{(j)}$ has real entries if $j$ is even
or purely imaginary entries if $j$ is odd.

Finally, dividing $\lambda_u$ by
$A_u=e^{-\frac12 \langle \Sigma u,\Sigma u\rangle}$ makes no difference since $A_u$ is real.
% \textbf{\Cref{prop:itemb}}
% From \Cref{lem:rel} we have
% \begin{align}\label{eq:bbbb}
% \Lcomp^{\KK} v(x,r) \1_{\{F_{\KK}(x)=\ell-r\} }
% &= \int_{[-\pi,\pi]^d} e^{-iu(\ell-r)} \Lcomp_u^{\KK} v(x,r) \, du \nonumber \\
% &= \int_{[-\pi,\pi]^d} e^{-iu(\ell-r)} (\lambda^{\KK}_u \Pi_u + Q^{\KK}_u) v(x,r) \, du.
% \end{align}
% The proof is standard, but we recall the main ingredients for completeness.
%
% First, it follows from \ref{prop:itema} that $v_0^{(j)}$ is purely imaginary for odd $j$ and real for even $j$. Thus, $\int_{\Xcomp} \left(\sum_{j=1}^\infty\frac{(iu)^j}{j!}\, F^j\right)\left(\sum_{j=1}^\infty\frac{(iu)^j}{j!}\, v_0^{(j)}\right)\, \diff\Leb$ has the same property.
% Using this information in~\eqref{eq:use10}, we see that
% the eigenvalue
% \begin{align}\label{evevev}
%  \lambda_u^{\KK}=e^{-\frac12 \langle \Sigma u, \Sigma u \rangle} e^{\sum_{j=3}^\infty \frac{i^j}{j!} \otimes^j(H_ju) }.
% \end{align}
% for real matrices $H_j$.
% By~\eqref{eq:delta0}, %the Aperiodicity \Cref{lem:aper},
% we have $\| \cL_u^{\KK} \|_{\cB} \leq \delta_0^{\KK}$
% for $|u | \geq \delta$,
% so that part of the integral
% can be captured in an operator $E_k$ satisfying $\| E_k v\|_{\cB} \leq \| v \|_{C^1(\Xcomp)}$.
%
% After a change of coordinates $u \mapsto u/2\pi \sqrt{\KK}$ in~\eqref{eq:bbbb}, the relevant integral reduces to one over
% $\frac{1}{2\pi \sqrt{\KK})^d} \int_{[-\delta \sqrt{\KK}, \delta \sqrt{\KK}]^d}$.
% This together with~\eqref{evevev}, \eqref{eq:decomp} and the analyticity of $\Pi_u$ gives the
% statement for $v\in\cB$. The statement for $v\in C^1$ follows from Lemma~\ref{lem:emb}.
\end{proof}

\subsection{Proof of \Cref{lemma:EKN}}
\label{sec:EKN}
Note that
\begin{align*}
 \left|\frac{1}{\lambda^{\KK}}\int_0^{\lambda^{\KK}T} E_{K,N} G\circ \phi_{r} (x_K) \, d r\right|=\left|\frac{1}{\lambda^{\KK}}\int_{\tilde W} E_{K,N} G \right|,
\end{align*}
where $\tilde W$ is the segment connecting $x_K$ and $\phi_{\lambda^{\KK} T} (x_K)$.

Recalling the definition of weak norm $\|h\|_{\cB_w}$ in subsection~\ref{subsec:ban}, we see that
\begin{align*}
 \left|\frac{1}{\lambda^{\KK}}\int_{\tilde W} E_{K,N} G\circ \phi_{r} (x_K) \, d r\right|\le \frac{|\tilde W|}{\lambda^{\KK}}\|E_{K,N} G\|_{\cB_w}.
\end{align*}
Recalling $C^1\subset\cB\subset \cB_w$, $\|E_{K,N} G\|_{\cB_w}\le \|E_{K,N}\|_{\cB} \|G\|_{C^1}$, and Lemma~\ref{lemma:EKN} follows.

\appendix
\section{On tensor products $\otimes$ and $*$}
\label{sec:tesnot}

We briefly review tensor products $\otimes$ and ``scalar'' product $*$.
If
$A$ is an $c \times d$ matrix, and $A'$ an $c' \times d'$-matrix, then
$A\otimes A'$ is an $cc' \times dd'$-matrix defined as
\begin{eqnarray*}
A\otimes A'
&=& \begin{pmatrix}
   a_{11} & \dots & a_{1d} \\
   \vdots & \ddots & \vdots \\
   a_{c1} & \dots & a_{cd}
  \end{pmatrix}
  \otimes
  \begin{pmatrix}
   a'_{11} & \dots & a'_{1d'} \\
   \vdots & \ddots & \vdots \\
   a'_{c'1} & \dots & a'_{c'd'}
  \end{pmatrix} \\[4mm]
  &=&
  \begin{pmatrix}
   a_{11} \begin{pmatrix}
   a'_{11} & \dots & a'_{1d'} \\
   \vdots & \ddots & \vdots \\
   a'_{c'1} & \dots &  a'_{c'd'}
  \end{pmatrix} & \dots &  a_{1d} \begin{pmatrix}
   a'_{11} & \dots & a'_{1d'} \\
   \vdots & \ddots & \vdots \\
   a'_{c'1} & \dots & a'_{c'd'}
  \end{pmatrix} \\
   \vdots & \ddots & \vdots \\
   a_{c1} \begin{pmatrix}
   a'_{11} & \dots & a'_{1d'} \\
   \vdots & \ddots & \vdots \\
   a'_{c'1} & \dots & a'_{c'd'}
  \end{pmatrix} & \dots & a_{cd} \begin{pmatrix}
   a'_{11} & \dots & a'_{1d'} \\
   \vdots & \ddots & \vdots \\
   a'_{c'1} & \dots & a'_{c'd'}
  \end{pmatrix}
  \end{pmatrix}.
\end{eqnarray*}
Note that this tensor product is not commutative: $A \otimes A'$ is only isomorphic but in general not equal to $A' \otimes A$.

Throughout, our vectors $u \in \C^d$ will be considered as column vectors, also if they appear as the arguments of (scalar) functions $f:\C^d \to \C$.
Now the $j$-fold tensor $u^{\otimes j}$ product of $u$ with itself is defined inductively:
$$
u^{\otimes 0} = 1 \in \C, \quad u^{\otimes 1} = u = \begin{pmatrix} u_1 \\ \vdots \\ u_d \end{pmatrix},
\quad  u^{\otimes 2} = u \otimes u =
\begin{pmatrix} u_1u_1 \\ u_1u_2\\ \vdots\\ u_1u_d \\ \vdots \\ \vdots \\ u_du_1 \\ u_du_2 \\ \vdots \\ u_du_d \end{pmatrix}
\quad \text{  and } \quad u^{\otimes j} = u \otimes u^{\otimes j-1}.
$$
Thus $u^{\otimes j}$ is a column vector with $d^j$ entries.

It goes similarly with derivatives of scalar functions $f:\C^d \to \C$:
$$
f^{(0)} = f, \qquad f' = f^{(1)} = \nabla f =
\begin{pmatrix} \frac{\partial f}{\partial u_1} \\ \vdots \\ \frac{\partial f}{\partial u_d} \end{pmatrix}, \quad
f^{(2)} =
\begin{pmatrix} \frac{\partial^2 f}{\partial u_1 \partial u_1} \\
 \frac{\partial^2 f}{\partial u_1 \partial u_2}  \\
 \vdots \\ \frac{\partial^2 f}{\partial u_1 \partial u_d}  \\ \vdots \\ \vdots \\
 \frac{\partial^2 f}{\partial u_d \partial u_1}  \\ \frac{\partial^2 f}{\partial u_d \partial u_2}  \\
 \vdots \\ \frac{\partial^2 f}{\partial u_d\partial u_d} \end{pmatrix}
 \quad \text{ and } \quad f^{(j)} = \begin{pmatrix} \frac{\partial}{\partial u_1} \\ \frac{\partial}{\partial u_2} \\ \vdots \\ \frac{1}{\partial u_d}\end{pmatrix}  \otimes f^{(j-1)}.
$$
%However, Leibniz formula doesn't work except in some specific cases.

Next we define the $*$-product $u * v$ on matrices of the same size as $A*A' = \sum_{i,j} a_{ij} a'_{ij}$, if $u, u' \in \C^d$ are column vectors, then $u * u = \sum_{i=1}^d u_i u'_i$ is the usual scalar product.
We can extend this to $c \times d$-matrix $A$
and $c' \times d'$-matrix $A'$ provided $cd = c'd'$ and $c'$ is a multiple of $c$ (or vice versa). For this, we divide $A$ into $d/d'$ $c \times d'$-matrices and stack them up to a single $c' \times d'$-matrix an then $*$-multiply with $A'$.
For example
\begin{align*}
\begin{pmatrix}
 a_{11} & a_{12} & a_{13} & a_{14} & a_{15} & a_{16} \\
 a_{21} & a_{22} & a_{23} & a_{24} & a_{25} & a_{26}
\end{pmatrix} *&
\begin{pmatrix}
 a'_{11} & a'_{12} \\
 a'_{21} & a'_{22} \\
 a'_{31} & a'_{32} \\
 a'_{41} & a'_{42} \\
 a'_{51} & a'_{52} \\
 a'_{61} & a'_{62}
\end{pmatrix}
=
\begin{pmatrix}
 a_{11} & a_{12} \\
 a_{21} & a_{22} \\
 a_{13} & a_{14} \\
 a_{23} & a_{24} \\
 a_{15} & a_{16} \\
 a_{26} & a_{26}
\end{pmatrix} *
\begin{pmatrix}
 a'_{11} & a'_{12} \\
 a'_{21} & a'_{22} \\
 a'_{31} & a'_{32} \\
 a'_{41} & a'_{42} \\
 a'_{51} & a'_{52} \\
 a'_{61} & a'_{62}
\end{pmatrix}  \\[2mm]
=\ & a_{11}a'_{11} + a_{12}a'_{12}
+ a_{21}a'_{21} + a_{22}a'_{22}
+ a_{13}a'_{31} + a_{14}a'_{32} \\[1mm]
& + a_{23}a'_{41} + a_{24}a'_{42}
+ a_{15}a'_{16} + a_{51}a'_{52}
+ a_{25}a'_{26} + a_{61}a'_{62}.
\end{align*}
For example, for a symmetric matrix $\Sigma = \binom{\sigma_{11} \ \sigma_{12}}{\sigma_{21} \ \sigma_{22}}$ and vector $u = \binom{u_1}{u_2}$, we get
\begin{align}\label{eq:Sigmau}
\Sigma^2 * u^{\otimes 2} =&
\begin{pmatrix}
 \sigma_{11}^2 + \sigma_{12}^2 & 2\sigma_{12}^2 \\
 2\sigma_{12}^2 & \sigma_{12}^2+\sigma_{22}^2 \\
\end{pmatrix}
*
\begin{pmatrix}
 u_1^2 \\ u_1u_2 \\ u_1 u_2 \\ u_2^2
\end{pmatrix} \nonumber \\
=\ &
(\sigma_{11}^2 + \sigma_{12}^2) u_1^2
+ (\sigma_{11} \sigma_{12} + 2\sigma_{12}^2 + \sigma_{12}\sigma_{22}) u_1u_2
+ (\sigma_{12}^2 + \sigma_{22}^2) u_2^2 \nonumber \\[2mm]
=\ & \begin{pmatrix}
 \sigma_{11}u_1 + \sigma_{21}u_2 \\
  \sigma_12u_1 + \sigma_{22}u_2
\end{pmatrix}
*
\begin{pmatrix}
 \sigma_{11}u_1 + \sigma_{21}u_2 \\
  \sigma_{12}u_1 + \sigma_{22}u_2
\end{pmatrix} = \langle \Sigma u , \Sigma u\rangle,
\end{align}
as used in~\Cref{sec:limitlaws}

We can extend $*$ even further to column vectors $u \in \C^d$ and $v \in \C^{d'}$,
where we assume that $d'$ is a multiple of $d$.
To compute $u * v$, divide $v$ into $d'/d$ blocks of height $d$, multiply the $k$-th entry of $u$ with the $k$-th block, and add up these blocks.
For example,
$$
\begin{pmatrix} u_1 \\ u_2  \end{pmatrix} *
\begin{pmatrix} v_1 \\ v_2 \\ v_3 \\ v_4 \\ v_5 \\ v_6  \end{pmatrix} =
u_1 \cdot \begin{pmatrix} v_1 \\ v_2 \\ v_3  \end{pmatrix}
+ u_2 \cdot \begin{pmatrix}  v_4 \\ v_5 \\ v_6  \end{pmatrix}
= \begin{pmatrix} u_1v_1 + u_2v_4 \\ u_1v_2 + u_2v_5 \\ u_1 v_3 + u_2v_6  \end{pmatrix},
$$
and if $d$ is a multiple of $d'$, we divide $u$ into $d/d'$ block, etc. Thus $*$ acts commutatively on column vectors.
Also, if $d=d'$, then $u * v = \sum_{i=1}^d u_i v_i$ is the usual scalar product.
In this specific case $d = d'$, Leibniz rule works:
$\nabla(u * v) = \nabla u * v + u * \nabla v$, but in general, the dimensions don't match.

However, we have the Taylor formula for a $C^\infty$ scalar function $f:\C^d \to \C$ and column vector $u \in \C^d$:
$$
f(u) = \sum_{j=0}^\infty \frac{1}{j!} f^{(j)}(0) * u^{\otimes j} = 1 \cdot f(0) \cdot 1 +
1 \cdot
\begin{pmatrix}
\frac{\partial f(0)}{\partial u_1} \\ \frac{\partial f(0)}{\partial u_2} \\ \vdots \\ \frac{\partial f(0)}{\partial u_d}
\end{pmatrix}
*
\begin{pmatrix}
u_1 \\ u_2 \\ \vdots \\ u_d
\end{pmatrix}
+
\frac12 \cdot
\begin{pmatrix} \frac{\partial^2 f(0)}{\partial u_1 \partial u_1} \\
 \frac{\partial^2 f(0)}{\partial u_1 \partial u_2}  \\
 \vdots \\ \frac{\partial^2 f(0)}{\partial u_1 \partial u_d}  \\ \vdots \\ \vdots \\
 \frac{\partial^2 f(0)}{\partial u_d \partial u_1}  \\ \frac{\partial^2 f(0)}{\partial u_d \partial u_2}  \\
 \vdots \\ \frac{\partial^2 f(0)}{\partial u_d\partial u_d} \end{pmatrix}
 *
 \begin{pmatrix} u_1u_1 \\ u_1u_2\\ \vdots\\ u_1u_d \\ \vdots \\ \vdots \\ u_du_1 \\ u_du_2 \\ \vdots \\ u_du_d \end{pmatrix}
 + \dots
$$

\section{Integrals used in the proof of mains results}
\label{sec:integrals}

In the proof of \Cref{thm:f}, we need the following lemma.

\begin{lemma}\label{lem:integrals}
	\begin{enumerate}[A.]
		\item \label{lem:integral_a}
		Assume $d=1$. Given $\sigma, L \in \R$ and $j \in \{0,1,2,\dots\}$, write
$$
I_j(\sigma, L) = \int_{\R} e^{-\frac{\sigma^2}{2} u^2} e^{iLu} u^j \, du.  $$
		Then
\[ I_0(\sigma, L) = \frac{\sqrt{2\pi}}{\sigma} e^{-\frac{ L^2}{2\sigma^2} },
\quad I_1 = \frac{i\sqrt{2\pi} L}{\sigma^3}  e^{-\frac{ L^2}{2\sigma^2} },
\ \text{ and } \
I_j(\sigma, L) = \frac{1}{\sigma^2} (i L I_{j-1} + (j-1) I_{j-2}).
\]

		\item \label{lem:integral_b}
		Assume $d=2$. Given a $2 \times 2$ covariance matrix $\Sigma^2$, $L \in \R^2$ and $j \in \{0,1,2,\dots\}$, write
$$
I_j(\Sigma, L) = \int_{\R^2} e^{-\frac{1}{2} \langle \Sigma u, \Sigma u\rangle} e^{i\langle L, u\rangle} u^{\otimes j} \, du,
$$
where $u^{\otimes j}$ is the $j$-fold tensor product of the vector $u = \binom{u_1}{u_2}$ with itself.
Then $\vec I_j$ is an alternatingly real and purely imaginary vector with $2^j$ components. Specifically:
$$
I_0(\Sigma, L) = \frac{2\pi}{\sigma_1 \sigma_2}
e^{- \frac12 \langle \Sigma^{-1} L, \Sigma^{-1} L \rangle  }
\quad \text{ and } \quad
I_1(\Sigma, L) =
\frac{2\pi i}{\sigma_1 \sigma_2} e^{- \frac12 \langle \Sigma^{-1} L, \Sigma^{-1} L \rangle }
A \begin{pmatrix}
 \frac{1}{\sigma_1^2} (A^*L)_1 \\[1mm]
 \frac{1}{\sigma_2^2} (A^*L)_2
     \end{pmatrix}
$$
where $\Sigma = A J A^{-1}$, $J = \binom{\sigma_1\ \\ 0}{\ 0\  \\ \ \sigma_2}$, is the diagonalization of $\Sigma$ with unitary matrix $A$.
Also
$$
I_2(\Sigma, L) = \frac{2\pi}{\sigma_1\sigma_2} e^{- \frac12 \langle  \Sigma^{-1} L,  \Sigma^{-1} L \rangle }
(Q(L \otimes L) + Q')
$$
with the quadratic form $Q$ and vector $Q' \in \R^{d^2}$ made explicit in the proof.
%  \[
% (Q(L \otimes L) + Q') = \begin{pmatrix}
%  (A_{11})^2 \frac{\sigma_1^2 - (A^*L)_1^2}{\sigma_1^4}
%  - 2A_{11} A_{12} \frac{(A^*L)_1}{\sigma_1^2} \frac{(A^*L)_2}{\sigma_2^2}
%  + (A_{12})^2 \frac{\sigma_2^2 - (A^*L)_2^2}{\sigma_2^4} \\[2mm]
%  A_{11}A_{21} \frac{\sigma_1^2 - (A^*L)_1^2}{\sigma_1^4}
%  - (A_{11} A_{22} + A_{12} A_{21}) \frac{(A^*L)_1}{\sigma_1^2} \frac{(A^*L)_2}{\sigma_2^2} + A_{12} A_{22} \frac{\sigma_2^2 - (A^*L)_2^2}{\sigma_2^4} \\[2mm]
%  A_{21}A_{11} \frac{\sigma_1^2 - (A^*L)_1^2}{\sigma_1^4}
%  - (A_{21} A_{12} + A_{22} A_{11}) \frac{(A^*L)_1}{\sigma_1^2} \frac{(A^*L)_2}{\sigma_2^2}
%  + A_{22} A_{12} \frac{\sigma_2^2 - (A^*L)_2^2}{\sigma_2^4}\\[2mm]
%  (A_{21})^2 \frac{\sigma_1^2 - (A^*L)_1^2}{\sigma_1^5}
%  - 2A_{21} A_{22} \frac{(A^*L)_1}{\sigma_1^2} \frac{(A^*L)_2}{\sigma_2^2}
%  + (A_{22})^2 \frac{\sigma_2^2 - (A^*L)_2^2}{\sigma_1^4}
% \end{pmatrix}.
% \]
\end{enumerate}
\end{lemma}

\begin{proof}[Proof of \Cref{lem:integrals}.]
	{\bf \Cref{lem:integral_a}, d=1}. The integrals $I_j = I_j(\sigma,L)$ can be computed via integration by parts, namely
for $j \geq 1$ we have (taking into account that integrals over odd real or imaginary parts of the integrand vanish):
\begin{eqnarray*}
 I_j = I_j(\sigma,L) &=&  \int_{-\infty}^\infty u e^{-\frac{\sigma^2}{2} u^2} \, e^{ i L u} u^{j-1}\, du\\
 &=& \int_{-\infty}^\infty \frac{1}{\sigma^2} e^{-\frac{\sigma^2}{2} u^2} \, e^{ i L u}
 \left( i L u^{j-1} + (j-1)u^{j-2} \right) \, du
 = \frac{1}{\sigma^2} (i L I_{j-1} + (j-1) I_{j-2}).
\end{eqnarray*}
This recursion show that $I_j$ is alternatingly  real and purely imaginary.
We compute $I_0$ via a change of coordinates:
\begin{eqnarray}\label{eq:I0}
 I_0 &=& \int_{-\infty}^\infty e^{i L u - \frac{\sigma^2}{2} u^2} \, du = \int_{-\infty}^\infty e^{-( \frac{\sigma u}{\sqrt{2}} - \frac{i L}{\sqrt{2} \sigma} )^2} e^{-\frac{ L^2}{2\sigma^2} } \, du \nonumber\\
 &=& e^{-\frac{L^2}{2\sigma^2} } \frac{\sqrt{2}}{\sigma} \int_{-\infty}^\infty e^{-u^2} \, du
 = \frac{\sqrt{2\pi}}{\sigma} e^{-\frac{ L^2}{2\sigma^2} }.
\end{eqnarray}
Then the recursion gives
\begin{eqnarray*}
I_1  = \frac{i L}{\sigma^2} I_0 =
\frac{i\sqrt{2\pi} L}{\sigma^3}  e^{-\frac{ L^2}{2\sigma^2} },
&&
I_2 = \frac{1}{\sigma^2} (i L I_1 + I_0)
= \frac{\sigma^2-L^2}{\sigma^5} \sqrt{2\pi} e^{-\frac{ L^2}{2\sigma^2} },
\\
I_3 =  i L \frac{\sqrt{2\pi}}{\sigma^5} e^{-\frac{ L^2 }{2\sigma^2} } \left(3-\frac{L^2}{\sigma^2} \right),
&&
I_4 = \frac{\sqrt{2\pi}}{\sigma^5} e^{-\frac{ L^2 }{2 \sigma^2} }
\left( 3 - 6 \frac{L^2}{\sigma^2} +\frac{L^4}{\sigma^4}\right),
\end{eqnarray*}
and so on.

{\bf \Cref{lem:integral_b}, $d=2$}. Using diagonalization and the unitary change of coordinates $u = Av$ (so
$\langle \Sigma u, \Sigma u \rangle =
\langle AJA^{-1} u, AJA^{-1} u \rangle =
\langle JA^{-1} u, JA^{-1} u \rangle
= \sigma_1^2 v_1^2 + \sigma_2^2 v_2^2$ and
$\langle L, u \rangle = \langle A^* L, v\rangle$), we get
\begin{eqnarray*}
 I_0(\Sigma,L) &=& \int_{-\infty}^\infty e^{-\frac12 \sigma_1^2 v_1^2} e^{i (A^*L)_1 v_1} \, dv_1
 \int_{-\infty}^\infty e^{-\frac12 \sigma_1^2 v_2^2} e^{i (A^*L)_2 v_2} \, dv_2  \\
 &=& I_0(\sigma_1, (A^*L)_1) \cdot I_0(\sigma_2, (A^*L)_2)
= \frac{2\pi}{\sigma_1 \sigma_2}
e^{- \frac12 \left( \frac{1}{\sigma_1^2} (A^*L)_1^2 + \frac{1}{\sigma_2^2} (A^*L)_2^2 \right) }.
\end{eqnarray*}
Since
$\frac{1}{\sigma_1^2} (A^*L)_1^2 + \frac{1}{\sigma_2^2} (A^*L)_2^2 = \langle J^{-1}A^{-1}L, J^{-1}A^{-1}L \rangle = \langle A^{-1}\Sigma^{-1}L
A^{-1}\Sigma^{-1}L \rangle
= \langle \Sigma^{-1}L, \Sigma^{-1}L \rangle$,
the result follows.

Using the same change of coordinates, we get
\begin{eqnarray*}
 I_1(\Sigma,L) &=& \int_{-\infty}^\infty
 \int_{-\infty}^\infty
 e^{-\frac12 \sigma_1^2 v_1^2} e^{i (A^*L)_1 v_1}
  e^{-\frac12 \sigma_1^2 v_2^2} e^{i (A^*L)_1 v_2}  \cdot  A\binom{v_1}{v_2}  \, dv_2 \, dv_1 \\
 &=&
 A \begin{pmatrix}
 I_1(\sigma_1, (A^*L)_1) \cdot I_0(\sigma_2, (A^*L)_2) \\[1mm]
  I_0(\sigma_1, (A^*L)_1) \cdot I_1(\sigma_2, (A^*L)_2)
     \end{pmatrix}
\\
&=&
\frac{2\pi i}{\sigma_1 \sigma_2} e^{- \frac12 \left( \frac{1}{\sigma_1^2} (A^*L)_1^2 + \frac{1}{\sigma_2^2} (A^*L)_2^2 \right) }
A \begin{pmatrix}
 \frac{1}{\sigma_1^2} (A^*L)_1 \\[1mm]
 \frac{1}{\sigma_2^2} (A^*L)_2
     \end{pmatrix} =
\frac{2\pi i}{\sigma_1 \sigma_2} e^{- \frac12 \langle  \Sigma^{-1} L,  \Sigma^{-1} L \rangle }
A \begin{pmatrix}
 \frac{1}{\sigma_1^2} (A^*L)_1 \\[1mm]
 \frac{1}{\sigma_2^2} (A^*L)_2
     \end{pmatrix}.
\end{eqnarray*}
For $I_j(\Sigma,L)$, $j \geq 2$, the same methods works, but the computations are getting
increasingly lengthy. To explain a bit about $j=2$,
the change of coordinates now leads to the factor
\[
(Av) \otimes (Av) =
\begin{pmatrix}
 (A_{11})^2 v_1^2 + 2A_{11} A_{12} v_1v_2 + (A_{12})^2 v_2^2 \\[2mm]
 A_{11}A_{21} v_1^2 + (A_{11} A_{22} + A_{12} A_{21}) v_1v_2 + A_{12} A_{22} v_2^2 \\[2mm]
 A_{21}A_{11} v_1^2 + (A_{21} A_{12} + A_{22} A_{11}) v_1v_2 + A_{22} A_{12} v_2^2 \\[2mm]
 (A_{21})^2 v_1^2 + 2A_{21} A_{22} v_1v_2 + (A_{22})^2 v_2^2
\end{pmatrix}
\]
Multiplying this with $e^{-\frac12 \sigma_1^2 v_1^2} e^{i (A^*L)_1} e^{-\frac12 \sigma_1^2 v_2^2} e^{i (A^*L)_1 v_2}$
and integrating over $\R^2$ means replacing
$v_1^2$ by $I_2(\sigma_1, (A^*L)_1) I_0(\sigma_2, (A^*L)_2)
= \frac{\sigma_1^2 - (A^*L)_1^2}{\sigma_1^5\sigma_2}$, etc.
This way we get
$$
I_2(\Sigma,L)
% &=&
% \begin{pmatrix}
%  (A_{11})^2 I_2(\sigma_1,(A^*L)_1) I_0(\sigma_2,(A^*L)_2) +
%  2A_{11} A_{12} I_1(\sigma_1,(A^*L)_1) I_1(\sigma_2,(A^*L)_2)
%  + (A_{12})^2 I_0(\sigma_1,(A^*L)_1) I_2(\sigma_2,(A^*L)_2) \\[2mm]
%  A_{11}A_{21} I_2(\sigma_1,(A^*L)_1) I_0(\sigma_2,(A^*L)_2)
%  + (A_{11} A_{22} + A_{12} A_{21}) I_1(\sigma_1,(A^*L)_1) I_1(\sigma_2,(A^*L)_2) + A_{12} A_{22} I_0(\sigma_1,(A^*L)_1) I_2(\sigma_2,(A^*L)_2) \\[2mm]
%  A_{21}A_{11} I_2(\sigma_1,(A^*L)_1) I_0(\sigma_2,(A^*L)_2)
%  + (A_{21} A_{12} + A_{22} A_{12}) I_1(\sigma_1,(A^*L)_1) I_1(\sigma_2,(A^*L)_2) + A_{22} A_{22}
%  I_0(\sigma_1,(A^*L)_1) I_2(\sigma_2,(A^*L)_2)\\[2mm]
%  (A_{21})^2 I_2(\sigma_1,(A^*L)_1) I_0(\sigma_2,(A^*L)_2) +
%  2A_{21} A_{22} I_1(\sigma_1,(A^*L)_1) I_1(\sigma_2,(A^*L)_2) + (A_{22})^2
%  I_0(\sigma_1,(A^*L)_1) I_2(\sigma_2,(A^*L)_2)
% \end{pmatrix} \\[4mm]
= \frac{2\pi}{\sigma_1\sigma_2} e^{- \frac12 \langle  \Sigma^{-1} L,  \Sigma^{-1} L \rangle } (Q(L \otimes L) + Q')
$$
for
$$
Q(L \otimes L) + Q' =
\begin{pmatrix}
 (A_{11})^2 \frac{\sigma_1^2 - (A^*L)_1^2}{\sigma_1^4}
 - 2A_{11} A_{12} \frac{(A^*L)_1}{\sigma_1^2} \frac{(A^*L)_2}{\sigma_2^2}
 + (A_{12})^2 \frac{\sigma_2^2 - (A^*L)_2^2}{\sigma_2^4} \\[2mm]
 A_{11}A_{21} \frac{\sigma_1^2 - (A^*L)_1^2}{\sigma_1^4}
 - (A_{11} A_{22} + A_{12} A_{21}) \frac{(A^*L)_1}{\sigma_1^2} \frac{(A^*L)_2}{\sigma_2^2} + A_{12} A_{22} \frac{\sigma_2^2 - (A^*L)_2^2}{\sigma_2^4} \\[2mm]
 A_{21}A_{11} \frac{\sigma_1^2 - (A^*L)_1^2}{\sigma_1^4}
 - (A_{21} A_{12} + A_{22} A_{11}) \frac{(A^*L)_1}{\sigma_1^2} \frac{(A^*L)_2}{\sigma_2^2}
 + A_{22} A_{12} \frac{\sigma_2^2 - (A^*L)_2^2}{\sigma_2^4}\\[2mm]
 (A_{21})^2 \frac{\sigma_1^2 - (A^*L)_1^2}{\sigma_1^5}
 - 2A_{21} A_{22} \frac{(A^*L)_1}{\sigma_1^2} \frac{(A^*L)_2}{\sigma_2^2}
 + (A_{22})^2 \frac{\sigma_2^2 - (A^*L)_2^2}{\sigma_1^4}
\end{pmatrix}.
$$
In general, the terms in $I_j(\Sigma,L)$ are all scalars of the form $c I_a I_b$, where $c \in \R$ and $a+b = j$,
so these are alternatingly real and purely imaginary in $j$.
\end{proof}

\end{document}